\documentclass[reqno,11pt]{amsart}

\usepackage{latexsym}
\usepackage{float}
\usepackage{amssymb}
\usepackage{amsmath}
\usepackage{amsthm}
\usepackage{mathrsfs}
\usepackage{euscript}
\usepackage[cmtip,all]{xy}

\usepackage{bbold}
\usepackage{tabu}
\usepackage{enumerate}

\usepackage{cite}

\usepackage{euscript}

\usepackage{graphicx}

\usepackage{bm}
\usepackage{stackrel}
\usepackage{accents}

\makeatletter
\g@addto@macro \normalsize {%
 \setlength\abovedisplayskip{7pt}%
 \setlength\belowdisplayskip{6pt}%
}
\makeatother

\newtheorem{thm}[equation]{Theorem}
\newtheorem{lem}[equation]{Lemma}
\newtheorem{cor}[equation]{Corollary}
\newtheorem{prop}[equation]{Proposition}

\theoremstyle{remark}
\newtheorem{rem}[equation]{Remark}
\newtheorem{fact}[equation]{Fact}

\newtheorem{claim}[equation]{Claim}

\theoremstyle{definition}

\newtheorem{prob}{Problem}

\numberwithin{equation}{section}

\newcommand{\gb}{\beta}
\newcommand{\ga}{\alpha}

\newcommand{\gL}{\Lambda}

\newcommand{\gD}{\Delta}

\newcommand{\eps}{\varepsilon}

\newcommand{\fa}{{\mathfrak a}}             
\newcommand{\fb}{{\mathfrak b}}
\newcommand{\fg}{{\mathfrak g}}
\newcommand{\fh}{{\mathfrak h}}

\newcommand{\fl}{{\mathfrak l}}
\newcommand{\fm}{{\mathfrak m}}
\newcommand{\fn}{{\mathfrak n}}

\newcommand{\fp}{{\mathfrak p}}

\newcommand{\fu}{{\mathfrak u}}

\newcommand{\fy}{{\mathfrak y}}

\newcommand{\f}{\mathfrak}

\newcommand{\nbar}{\bar{n}}                 

\newcommand{\R}{\mathbb{R}}          
\newcommand{\C}{\mathbb{C}}          
           
\newcommand{\Z}{\mathbb{Z}}

\newcommand{\ad}{\mathrm{ad}}
\newcommand{\Ad}{\mathrm{Ad}}
\newcommand{\Cal}{\mathcal}

\newcommand{\Hom}{\operatorname{Hom}}

\renewcommand{\Im}{\mathrm{Im}}
\newcommand{\Ind}{\mathrm{Ind}}
\newcommand{\IP}[2]{\langle#1 , #2\rangle}     

\newcommand{\spn}{\text{span}}

\newcommand{\Tr}{\text{Tr}}

\newcommand{\Sol}{\mathrm{Sol}}

\newcommand{\Pol}{\mathrm{Pol}}
\newcommand{\poly}{\mathrm{poly}}

\newcommand{\To}{\longrightarrow}

\newcommand{\Diff}{\mathrm{Diff}}
\newcommand{\Irr}{\mathrm{Irr}}

\newcommand{\triv}{\mathrm{triv}}
\newcommand{\fin}{\mathrm{fin}}

\newcommand{\sym}{\mathrm{sym}}

\newcommand{\wH}{\widetilde{H}}

\newcommand{\diag}{\mathrm{diag}}

\newcommand{\dpi}{d\pi}

\newcommand{\id}{\mathrm{id}}

\newcommand{\symb}{\mathrm{symb}}

\newcommand{\sgn}{\mathrm{sgn}}

\newcommand{\RP}{\mathbb{R}\mathbb{P}}

\newcommand{\Ker}{\mathrm{Ker}}

\newcommand{\Mp}{M^{\fg}_{\fp}}
\newcommand{\Mpp}{M^{\fg'}_{\fp'}}

\newcommand{\balpha}{\bm{\alpha}}
\newcommand{\bbeta}{\bm{\beta}}
\newcommand{\bgamma}{\bm{\gamma}}
\newcommand{\blambda}{\bm{\lambda}}
\newcommand{\bmu}{\bm{\mu}}
\newcommand{\bnu}{\bm{\nu}}
\newcommand{\btau}{\bm{\tau}}

\newcommand{\abs}[1]{\left\vert#1\right\vert}

\newcommand{\Emb}{\mathrm{Emb}}
\newcommand{\wEmb}{\widetilde{\mathrm{Emb}}}

\newcommand{\wProj}{\widetilde{\mathrm{Proj}}}

\newcommand{\wy}{\widetilde{y}}

\newcommand{\EuD}{\EuScript{D}}

\newcommand{\Rest}{\mathrm{Rest}}

\newcommand{\w}{\widetilde}

\newcommand{\D}{\Cal{D}}
\newcommand{\bD}{\mathbb{D}}

\newcommand{\A}{\Cal{A}}
\newcommand{\bA}{\mathbb{A}}

\newcommand{\Np}{(N^-)'}
\newcommand{\Nn}{N_n^-}

\providecommand*{\donothing}[1]{}

\makeatletter
\@namedef{subjclassname@2020}{%
  \textup{2020} Mathematics Subject Classification}
\makeatother

\begin{document}

\baselineskip=16pt
\tabulinesep=1.2mm


\title[]{
Differential symmetry breaking operators from a line bundle to a vector bundle
over real projective spaces
}

\dedicatory{Dedicated to Professor Toshiyuki Kobayashi,
whose incredible insight always leads us to a whole 
new world in mathematics}


\author{Toshihisa Kubo}

\address{Faculty of Economics, 
Ryukoku University,
67 Tsukamoto-cho, Fukakusa, Fushimi-ku, Kyoto 612-8577, Japan}
\email{toskubo@econ.ryukoku.ac.jp}



\date{\today}

\maketitle


\begin{abstract} 
In this paper we classify and construct differential symmetry breaking operators $\bD$ from a line bundle over the real projective space $\RP^n$ to a vector bundle
over $\RP^{n-1}$. We further determine the factorization identities of  
$\bD$ and the branching laws of the corresponding generalized Verma modules
of $\f{sl}(n+1,\C)$.
By utilizing the factorization identities, 
the $SL(n,\R)$-representations realized on the image $\Im(\bD)$ 
are also investigated.
\end{abstract}

\setcounter{tocdepth}{1}
\tableofcontents

\section{Introduction}
\label{sec:intro}


Intertwining operators are some of the fundamental objects in representation theory. For instance, Knapp--Stein operators play a key role in the representation theory of real reductive groups. In this paper,  we consider the classification and 
construction of
certain intertwining differential operators called 
\emph{differential symmetry breaking operators}.
In order to describe the main problems of this paper,
we start the introduction with the definition of such operators.

\subsection{Differential symmetry breaking operators and main problems}

Let $X$ be a smooth manifold and $Y$ a smooth submanifold of $X$.
Take $G' \subset G$ to be a pair of Lie groups that act transitively 
on $Y$ and $X$, respectively.
Suppose that 
$\mathcal{V}\to X$ and $\mathcal{W}\to Y$ are
$G$- and $G'$-equivariant 
vector bundles over $X$ and $Y$ with fibers $V$ and $W$, respectively.
Then a continuous linear operator 
 $\bA \colon C^\infty(X,\mathcal{V}) \to C^\infty(Y,\mathcal{W})$ between
the spaces of smooth sections is called
a \emph{symmetry breaking operator}
if $\bA$ is  $G'$-intertwining \cite{KS15}. 
Hereafter, we shall often abbreviate it as an SBO.

Suppose  that there exists an inclusion
$\iota \colon Y \hookrightarrow X$. In this case, although the base manifolds
$X$ and $Y$ are different, one can define a differential operator 
$\bD\colon C^\infty(X,\mathcal{V}) \to C^\infty(Y,\mathcal{W})$
with respect to the inclusion  $\iota$
(see \cite{KP1} for the details).
We call the differential operator $\bD$
a \emph{differential symmetry breaking operator} 
if $\bD$ is also an SBO (cf.\ \cite{Kobayashi13, Kobayashi14, KP1, KOSS15}). 
The systematic study of SBOs $\bA$ and differential SBOs $\bD$ 
have been initiated by 
T.\ Kobayashi with his collaborators over a decade
(cf.\ \cite{Kobayashi13, Kobayashi14, KKP16, KP1, KP2, KOSS15, KS15, KS18}).

There are three main problems in this paper. To describe them in a general framework, let $G$ be a real reductive Lie group and $G' \subset G$ a reductive subgroup of $G$.
Let $P=MAN_+$ and $P'=M'A'N_+'$ be Langlands decompositions 
of $P$ and $P'$, respectively, with $M'A' \subset MA$ and $N_+' \subset N_+$.
We denote by $\Irr(M)_{\fin}$ and $\Irr(M')_{\fin}$ the sets of equivalence classes of 
finite-dimensional irreducible representations of $M$ and $M'$, respectively.
Likewise, we write $\Irr(A)$ and $\Irr(A')$ for the sets of characters of $A$ and 
$A'$, respectively. Then, for the outer tensor products 
$\xi \boxtimes \lambda \boxtimes \triv$ 
of $(\xi, \lambda) \in \Irr(M)_\fin \times \Irr(A)$ and
$\varpi \boxtimes \nu \boxtimes \triv$
of
$(\varpi, \nu) \in \Irr(M')_\fin \times \Irr(A')$ with
the trivial representations $\triv$ of $N_+$ and $N_+'$,
we put
\begin{equation*}
I(\xi,\lambda) = \Ind_{P}^G(\xi \boxtimes \lambda \boxtimes \triv)
\quad
\text{and}
\quad
J(\varpi,\nu) = \Ind_{P'}^{G'}(\varpi \boxtimes \nu \boxtimes \triv)
\end{equation*}
for (unnormalized) parabolically induced representations of $G$ and $G'$,
respectively.
We denote by
$\Diff_{G'}(I(\xi,\lambda), J(\varpi,\nu))$ the space of differential SBOs
$\bD\colon I(\xi,\lambda) \to J(\varpi,\nu)$. 

\subsection{Classification and construction of $\bD$}

The first problem concerns the classification and construction of differential SBOs $\bD$.
More precisely, we consider the following problem.

\begin{prob}[{Classification and construction of $\bD$}]
\label{prob:A}
Do the following.
\begin{enumerate}
\item[(A1)] 
Classify $(\xi,\lambda, \varpi, \nu)$ such that
\begin{equation*}
\Diff_{G'}(I(\xi,\lambda), J(\varpi,\nu)) \neq \{0\}.
\end{equation*}

\item[(A2)] 
Determine 
\begin{equation*}
\dim \Diff_{G'}(I(\xi,\lambda), J(\varpi,\nu)).
\end{equation*}

\item[(A3)] 
Construct generators
\begin{equation*}
\bD \in \Diff_{G'}(I(\xi,\lambda), J(\varpi,\nu)).
\end{equation*}

\end{enumerate}

\end{prob}

\subsection{Factorization identity of $\bD$}

The second problem concerns a decomposition formula of $\bD$.
Observe that the composition $\A_J\circ \bA_1$ of a
(not necessarily differential) SBO 
$\bA_1 \colon I(\xi_1,\lambda_1) \to J(\varpi_1, \nu_1)$ with 
a $G'$-intertwining operator 
$\A_J\colon J(\varpi_1, \nu_1) \to J(\varpi_2, \nu_2)$ is an SBO
\begin{equation*}
\A_J\circ \bA_1 \colon I(\xi_1,\lambda_1) \to J(\varpi_2, \nu_2).
\end{equation*}
In a diagram we have
\begin{equation*}
\begin{aligned}
\xymatrix@=13pt{
I(\xi_1,\lambda_1)
  \ar[rrdd]_{\A_J\circ \bA_1}
  \ar[rr]^{\bA_1}
 && J(\varpi_1,\nu_1)
   \ar[dd]^{\A_J}
  \ar@{}[ldd]|{\circlearrowleft}
 \\
&&\\
&&
J(\varpi_2,\nu_2)
}
\end{aligned}
\end{equation*}

\noindent
Likewise, the composition $\bA_2 \circ \A_I$
of a $G$-intertwining operator 
$\A_I \colon I(\xi_1, \lambda_1) \to I(\xi_2, \lambda_2)$ with 
an SBO $\bA_2 \colon I(\xi_2,\lambda_2) \to J(\varpi_2, \nu_2)$
is also an SBO
\begin{equation*}
\bA_2\circ \A_I \colon I(\xi_1,\lambda) \to J(\varpi_2, \nu_2),
\end{equation*}
that is,
\begin{equation*}\label{eqn:intro-factor2}
\begin{aligned}
\xymatrix@=13pt{
I(\xi_1,\lambda_1)
 \ar[dd]_{\A_I}
  \ar[rrdd]^{\bA_2\circ \A_I }
 && 
 \\
&&\\
I(\xi_2,\lambda_2)
\ar[rr]_{\bA_2}   \ar@{}[ruu]|{\circlearrowleft}
&&
J(\varpi_2,\nu_2)
}
\end{aligned}
\end{equation*}

\noindent
For a given SBO 
$\bA \colon I(\xi_1,\lambda_1) \to J(\varpi_2, \nu_2)$,
we call the identities 
\begin{equation*}
\bA = \A_J\circ \bA_1 =\bA_2\circ \A_I 
\end{equation*}
the \emph{factorization identities} of $\bA$.
We remark that such an identity is also known as 
a \emph{functional equation} in the literature (cf.\ \cite{KS15, KS18}).
The diagram \eqref{eqn:intro-factor3} below illustrates the identities.
\begin{equation*}\label{eqn:intro-factor3}
\begin{aligned}
\xymatrix@=13pt{
I(\xi_1,\lambda_1)
 \ar[dd]_{\A_I}
  \ar[rrdd]^{\bA}
  \ar[rr]^{\bA_1}
 && J(\varpi_1,\nu_1)
   \ar[dd]^{\A_J}
  \ar@{}[ldd]|{\circlearrowleft}
 \\
&&\\
I(\xi_2,\lambda_2)
\ar[rr]_{\bA_2}   \ar@{}[ruu]|{\circlearrowleft}
&&
J(\varpi_2,\nu_2)
}
\end{aligned}
\end{equation*}

In this paper we consider the factorization identities for differential SBOs $\bD$.

\begin{prob}[{Factorization identities of $\bD$}]\label{prob:B}
Compute the factorization identities 
\begin{equation}\label{eqn:intro-factor}
\bD = \D_J\circ \bD_1 =\bD_2\circ \D_I 
\end{equation}
of a differential SBO 
$\bD \in \Diff_{G'}(I(\xi,\lambda), J(\varpi,\nu))$.
\end{prob}

For preceding works on the factorization identities, see, for instance,
Fischmann--Juhl--Somberg \cite{FJS20},
Juhl \cite{Juhl09}, 
Kobayashi--Kubo--Pevzner \cite{KKP16}, 
Kobayashi--Pevzner \cite{KP2}, and
Kobayashi--{\O}rsted--Somberg--Sou{\v c}ek \cite{KOSS15}, 
for differential SBOs $\bD$.
For (not necessarily differential) SBOs $\bA$, see, for instance, 
Hader \cite{Hader22} and
Kobayashi--Speh \cite{KS15, KS18}.
We remark that although \eqref{eqn:intro-factor} expresses 
$\bD$ in two ways, namely, $\bD=\D_J\circ \bD_1$ and 
$\bD =\bD_2\circ \D_I$, it seems more natural for $\bD$ (or $\bA$)
to satisfy only one of the 
identities. We consider the ``double factorization identities''
$\bD = \D_J\circ \bD_1 =\bD_2\circ \D_I$ in this paper.

\subsection{The image $\Im(\bD)$}

As an SBO $\bA \colon I(\xi,\lambda) \to J(\varpi,\nu)$
is $G'$-intertwining, the image $\Im(\bA)$ is naturally a $G'$-invariant subspace
 of $J(\varpi,\nu)$. In fact, in the fundamental work of Kobayashi--Speh 
\cite{KS15, KS18} over SBOs $\bA$ for the pair 
$(G,G')=(O(n+1,1), O(n,1))$, they classified $\Im(\bA)$ at the level of 
$(\mathfrak{g}', K')$-modules, among many other things. 
In this paper, we also aim to determine $\Im(\bD)$ for differential SBOs $\bD$. 

\begin{prob}[{Determination of  $\Im(\bD)$}]\label{prob:C}
Determine $\Im(\bD)$ of 
$\bD \in \Diff_{G'}(I(\xi,\lambda), J(\varpi,\nu))$.
\end{prob}

Suppose that $\bD$ satisfies a factorization identity 
$\bD=\D_J\circ \bD_1=\bD_2 \circ \D_I$
as in \eqref{eqn:intro-factor}. Then 
$\Ker(\D_I)$ is a $G'$-subrepresentation of $I(\xi_1,\lambda_1)$, 
and $\Ker(\D_J)$ and $\Im(\bD_1)$  are
both $G'$-invariant subspaces of $J(\varpi_1,\nu_1)$, that is, we have
\begin{equation}\label{eqn:intro-factor4}
\begin{aligned}
\xymatrix@=13pt{
\text{Ker}(\D_I)\ar@{}[d]|{\bigcap}
\ar[rr]^{\stackrel{\bD_1\vert_{\text{Ker}(\D_I)}}{\phantom{a}}} 
&&\text{Ker}(\D_J)
\ar@{}[d]|{\bigcap}&\\
I(\xi_1,\lambda_1)
 \ar[dd]_{\D_I}
  \ar[rrdd]_{\bD}
  \ar[rr]^{\bD_1}
 && J(\varpi_1,\nu_1)
   \ar[dd]^{\D_J}
  \ar@{}[ldd]|{\circlearrowleft}
\ar@{}[r]|*{\supset}&\hspace{-5pt} \text{Im}(\bD_1)
 \\
&&\\
I(\xi_2,\lambda_2)
\ar[rr]_{\bD_2}   \ar@{}[ruu]|{\circlearrowleft}
&&
J(\varpi_2,\nu_2)
}
\end{aligned}
\end{equation}

\noindent
In this paper we also investigate
a relationship between $\Ker(\D_J)$ and $\Im(\bD_1)$.
Further, observe that as
\begin{equation*}
(\D_J\circ \bD_1)\vert_{\text{Ker}(\D_I)}
=(\bD_2 \circ \D_I)\vert_{\text{Ker}(\D_I)}
=0,
\end{equation*}
we have $\Im(\bD_1\vert_{\text{Ker}(\D_I)}) \subset \text{Ker}(\D_J)$.
We then consider whether or not the equality
$\Im(\bD_1\vert_{\Ker(\D_I)}) = \Ker(\D_J)$ holds.
This is somewhat analogous to the work \cite{Hader22} of 
Hader for his study of the Heisenberg wave operator in sprit.

\subsection{$SL$ vs $GL$}

The aim of this paper is to answer Problems \ref{prob:A}, 
\ref{prob:B}, and \ref{prob:C} for differential SBOs
$\bD \in \Diff_{G'}(I(\xi, \lambda), J(\varpi, \nu))$
for the case
$(G/P, G'/P')\simeq (\RP^n, \RP^{n-1})$, 
the real projective spaces of dimension
$n$ and $n-1$, respectively, with 
$\xi \in \Irr(M)_{\fin}$ for $\dim \xi =1$.
We allow the inducing representation $\varpi$ for $J(\varpi, \nu)$
 to be any $\varpi \in \Irr(M')_{\fin}$.
For the purpose, there are at least two choices on $(G,G')$, namely,
$(G, G') = (SL(n+1,\R), SL(n,\R))$ or $(GL(n+1,\R), GL(n,\R))$.
In this paper, we consider both cases for Problems \ref{prob:A} and \ref{prob:B}.
Only the $SL$-case is considered for Problem \ref{prob:C}.
We utilize our preceding results on $SL(n,\R)$-intertwining differential operators $\D$
to consider Problem \ref{prob:C} for the $SL$-case.

For Problem \ref{prob:A},  there is a significant difference between
the $GL$-case and $SL$-case for $n=2$. In the $GL$-case, it turned out that the space
$\Diff_{G'}(I(\xi, \lambda), J(\varpi, \nu))$ of differential SBOs for $\xi$ with $\dim \xi =1$
is multiplicity-free for all $n\geq 2$. In contrast, the dimension could be 
$\dim \Diff_{G'}(I(\xi, \lambda), J(\varpi, \nu)) =2$ in the $SL$-case for $n=2$.
This difference arises because there are more parameters in the $GL$-case 
than the $SL$-case. For the details of the results, see Theorems \ref{thm:DSBO2a}
and \ref{thm:DSBO2b} for the $SL$-case and Theorem \ref{thm:GL-DSBO} for the 
$GL$-case.

\subsection{Classification of SBOs $\bA$ between line bundles
over $\RP^n$ and $\RP^{n-1}$}
\label{sec:FW}

All symmetry breaking operators $\bA\colon 
I(\triv, \lambda) \to J(\triv, \nu)$
including non-local operators 
for $(G,G')=(GL(n+1,\R), GL(n,\R))$
with $\xi=\triv$ and $\varpi=\triv$, 
the trivial representations of $M$ and $M'$, respectively,
are already classified by Frahm--Weiske in \cite{FW19}. 
At first glance,
their classification seems to miss a family of differential SBOs
$\bD_{(m,\ell)} \colon
I(\triv, \lambda) \to J(\triv, \nu)$ for $n=2$
in Theorem \ref{thm:GL-DSBO}.
However, the differential SBOs $\bD_{(m,\ell)}$ can be thought of as 
the residue operators of SBOs $\mathcal{C}_{\lambda,\nu}$ of 
\cite[Thm.\ 3.12]{FW19}.
Therefore, their classification indeed includes $\bD_{(m,\ell)}$ implicitly.

It seems that one needs to be careful with the comments at the end of Section 1.3 of \cite{FW19} on Knapp--Stein operators (standard intertwining operators).
The authors of the cited paper briefly commented that
there do not exist non-trivial Knapp--Stein operators for 
the parabolically induced representations for $P' \subset G'$. 
Nevertheless, if $n=2$, then $G' = GL(2,\R)$ and $P'$ is a minimal parabolic
subgroup of $G'$. Thus, non-trivial Knapp--Stein operators for $P'\subset G'$ do 
exist. Indeed, the aforementioned differential SBOs $\bD_{(m,\ell)}$
satisfy a factorization identity with a normal derivative and the residue operator
of a Knapp--Stein operator. We shall discuss some details in Remark \ref{rem:FW2}.

\subsection{F-method}

The main machinery for us to classify differential SBOs $\bD$
on Problem \ref{prob:A} is the 
F-method (cf.\ \cite{FJS20, Kobayashi13, Kobayashi14, KKP16, KP1, KP2, KOSS15}).
Via the so-called \emph{algebraic Fourier transform of generalized Verma modules},
this method allows one to classify and construct differential SBOs $\bD$ simultaneously,
by solving a certain system of partial differential equations.
For the recent study of the F-method, see, for instance,  \cite{Frahm20, KuOr24, Nakahama19, Nakahama22, Nakahama23, Perez23} and the references therein.

\subsection{Branching law of generalized Verma modlues}

Via the duality between differential SBOs and $(\fg',P')$-homomorphisms
between generalized Verma modules (Theorem \ref{thm:duality}),
the classification of differential SBOs is closely related to the branching law
of a generalized Verma module $\Mp(\lambda)$ of $\fg$.
The character identity in the Grothendieck group of the BGG category $\Cal{O}^{\fp'}$
for a parabolic subalgebra $\fp' = \fl'\oplus \fn_+'$
yields the branching law of $\Mp(\lambda)$ for generic parameter $\lambda$.
The branching law for singular $\lambda$  requires 
some information of the structures 
of $\Mp(\lambda)$ such as the classification of $\fn_+'$-invariant subspaces of 
$\Mp(\lambda)$. See, for instance, \cite{Kobayashi12, KOSS15} 
for the study of the branching laws of generalized Verma modules.

In this paper, in addition to the three main problems,
we also discuss the branching laws of generalized Verma modules
in consideration. For singular parameters, the results from the F-method play a key role.
We further present the branching law of the image of $\fg'$-homomorphisms as well.
These are accomplished in Theorems \ref{thm:GVM31a} and \ref{thm:GVM31b}.
Via the aforementioned duality, the resulting branching laws support our classification of differential SBOs.

\subsection{Organization of the paper}

Now we describe the rest of the paper.
There are ten sections including the introduction.
The aim of Section \ref{sec:Fmethod} is to 
give a quick overview of the F-method in a general
framework. At the end of the section, a recipe of the F-method will be presented,
which plays a central role for the classification and construction of differential 
SBOs $\bD$ in this paper. 
In Section \ref{sec:SL}, we then specialize the framework to the case 
$(G,G')=(SL(n+1,\R), SL(n,\R))$ with maximal parabolic subgroups 
$P\subset G$ and $P'\subset G'$ such that $G/P\simeq \RP^n$ and 
$G'/P' \simeq \RP^{n-1}$. Some necessary notation is introduced in this section.

The objective of Section \ref{sec:results} is to summarize the main results
of the classification and construction of differential SBOs 
$\bD$ (Problem \ref{prob:A}). Via the duality theorem (Theorem \ref{thm:duality}),
we also discuss $(\fg',P')$- and $\fg'$-homomorphisms $\Phi$ between
certain generalized Verma modules. 
The proofs of the theorems in 
Section \ref{sec:results} are discussed in 
Sections \ref{sec:proof1} and \ref{sec:proof2}.  
As mentioned above, for the $SL$-case, 
the multiplicity-two phenomenon appears for $n=2$. Thus, we separate the proofs
into two cases, namely, the cases for $n\geq 3$ and $n=2$. Section \ref{sec:proof1}
deals with the former case; we handle the latter case in Section \ref{sec:proof2}.
In both cases, we follow the recipe of the F-method.

Section \ref{sec:factor1} is devoted to the factorization identities of 
differential SBOs $\bD$ constructed in Section \ref{sec:results}. 
We first give such identities for 
$(\fg',P')$-homomorphisms $\Phi$. We then convert them to ones for differential SBOs 
$\bD$ via the duality theorem. The factorization identities of $\bD$ are obtained in
Theorem \ref{thm:Proj}. 
Some relevant results on intertwining differential operators and 
$(\fg, P)$- and $(\fg', P')$-homomorphisms are also recalled from \cite{KuOr24} in this section.

The images $\Im(\bD)$ are determined in Section \ref{sec:image} 
for $\bD$ that satisfies the factorization identities (Problem \ref{prob:C}). 
In this section we make use of some results
from \cite{KuOr24} to determine them. 
These are achieved in Section \ref{sec:image2}.

The aim of Section \ref{sec:GVM} is to discuss
the branching laws of generalized Verma modules in consideration.
In this section, we first recall from \cite{Kobayashi12} a character identity
of a generalized Verma module in a general framework to give 
branching laws in the Grothendieck group 
of the BGG category $\Cal{O}^{\fp'}$.
We then apply 
the character identity to the case $(\fg,\fg') = (\f{sl}(n+1,\C), \f{sl}(n,\C))$
with maximal parabolic subalgebras $\fp \subset \fg$ and $\fp' \subset \fg'$
considered above.
After discussing the decomposition of formal characters,
we give actual branching laws by utilizing the results from the F-method.
In this section, we also discucss the branching 
law of the image of a $\fg'$-homomorphism 
from a generalized Verma module for $\fg'$ to the one for $\fg$.
This supports the aforementioned multiplicity-two phenomenon
of $SL(n,\R)$-differential SBOs $\bD$ for $n=2$.
The explicit branching laws are given in
Theorems \ref{thm:GVM31a} and \ref{thm:GVM31b}.

The last section, namely, Section \ref{sec:GL}, is for the $GL$-counterpart of the results
of Problems \ref{prob:A} and \ref{prob:B}. In this section we give the classification of 
$GL(n,\R)$-differential SBOs $\bD$ as well as their factorization identities. In principle,
the $GL(n,\R)$-operators $\bD$ are the same as 
the $SL(n,\R)$-operators;
what one wishes to do is to classify the appropriate parameters. The main results are obtained in Theorems \ref{thm:GL-DSBO} and \ref{thm:GL-factor}.

\section{Preliminaries: the F-method}
\label{sec:Fmethod}

The aim of this section is to recall the so-called F-method.
In particular, we present a recipe of the F-method in Section \ref{sec:recipe}.
In Sections \ref{sec:proof1} and \ref{sec:proof2}, we shall follow the recipe to 
classify and construct differential symmetry breaking operators in concern.
For the details of the F-method, 
consult, for instance, \cite{KP1} and \cite{KuOr24}.
In this section we mainly take the expositions from \cite{KuOr24}.

\subsection{Notation}\label{sec:prelim}

Let $G$ be a real reductive Lie group and $P=MAN_+ $ a Langlands decomposition of a parabolic subgroup $P$ of $G$. We denote by $\fg(\R)$ and 
$\fp(\R) = \fm(\R) \oplus \fa(\R) \oplus \fn_+(\R)$ the Lie algebras of $G$ and 
$P=MAN_+$, respectively.

 For a real Lie algebra $\f{y}(\R)$, we write $\f{y}$
and $\Cal{U}(\fy)$ for its complexification and the universal enveloping algebra of 
$\fy$, respectively. 
For instance, $\fg, \fp, \fm, \fa$, and $\fn_+$ are the complexifications of $\fg(\R), \fp(\R), \fm(\R), \fa(\R)$, and $\fn_+(\R)$, 
respectively.

For $\lambda \in \fa^* \simeq \Hom_\R(\fa(\R),\C)$,
we denote by $\C_\lambda$ 
the one-dimensional representation of $A$ defined by 
$a\mapsto a^\lambda:=e^{\lambda(\log a)}$. 
For a finite-dimensional irreducible 
representation $(\sigma, V)$ of $M$ and $\lambda \in \fa^*$,
we denote by $\sigma_\lambda$ the outer tensor product representation $\sigma \boxtimes \C_\lambda$
on $V$,
namely, 
$\sigma_\lambda \colon
ma \mapsto a^\lambda\sigma(m)$. By letting $N_+$ act trivially, we regard 
$\sigma_\lambda$ as a representation of $P$. Let $\Cal{V}:=G \times_P V \to G/P$
be the $G$-equivariant vector bundle over the real flag variety $G/P$
 associated with the 
representation $(\sigma_\lambda, V)$ of $P$. We identify the Fr{\' e}chet space 
$C^\infty(G/P, \Cal{V})$ of smooth sections with 
\begin{equation*}
C^\infty(G, V)^P:=\{f \in C^\infty(G,V) : 
f(gp) = \sigma_\lambda^{-1}(p)f(g)
\;\;
\text{for any $p \in P$}\},
\end{equation*}
the space of $P$-invariant smooth functions on $G$.
Then, via the left regular representation $L$ of $G$ on $C^\infty(G)$,
we realize the parabolically induced representation 
$\pi_{(\sigma, \lambda)} = \Ind_{P}^G(\sigma_\lambda)$ on $C^\infty(G/P, \Cal{V})$.
We denote by $R$ the right regular representation of $G$ on $C^\infty(G)$.

Let $G'$ be a reductive subgroup of $G$ and $P'=M'A'N_+'$ 
a parabolic subgroup of $G'$
with $P' \subset P$
so that there exists a natural morphism $G'/P' \to G/P$
of the real flag variety $G'/P'$ to $G/P$.
We further assume that $M'A' \subset MA$ and $N_+' \subset N_+$.

As for $G/P$, for a finite-dimensional irreducible
representation $(\varpi_\nu, W)$ of $M'A'$, we define
an induced representation $\pi'_{(\varpi, \nu)}=\Ind_{P'}^{G'}(\varpi_\nu)$ on the 
space $C^\infty(G'/P', \Cal{W})$ of smooth sections for a 
$G'$-equivariant vector bundle $\Cal{W}:=G'\times_{P'}W \to G'/P'$.

Via the morphism $G'/P' \to G/P$, one can define 
differential operators $\bD\colon C^\infty(G/P, \Cal{V}) \to C^\infty(G'/P', \Cal{W})$
although $G'/P'$ and $G/P$ are different manifolds 
(see, for instance, \cite[Def.\ 2.1]{KP1}).
As $C^\infty(G/P, \Cal{V})$ is a $G$-representation and $G' \subset G$,
the space $C^\infty(G/P, \Cal{V})$ is also a $G'$-representation.
We then write $\Diff_{G'}(\Cal{V},\Cal{W})$ for the space of 
\emph{differential symmetry breaking operators} 
($G'$-intertwining differential operators)
$\bD \colon C^\infty(G/P, \Cal{V}) \to C^\infty(G'/P', \Cal{W})$.

Let 
$\mathfrak{g}(\R)=\mathfrak{n}_-(\R) \oplus \mathfrak{m}(\R) 
\oplus \mathfrak{a}(\R) \oplus \mathfrak{n}_+(\R)$ 
be the 
Gelfand--Naimark decomposition of $\mathfrak{g}(\R)$,
and write $N_- = \exp(\fn_-(\R))$. We identify $N_-$ with the 
open Bruhat cell $N_-P$ of $G/P$ via the embedding 
$\iota\colon N_- \hookrightarrow G/P$, $\bar{n} \mapsto \bar{n}P$.
Via the restriction of the vector bundle $\Cal{V} \to G/P$ to the open Bruhat cell
$N_-\stackrel{\iota}{\hookrightarrow} G/P$,
we regard $C^\infty(G/P,\mathcal{V})$ as a subspace of 
$C^\infty(N_-) \otimes V$.

Likewise,
let 
$\mathfrak{g}'(\R)=\mathfrak{n}_-'(\R) \oplus \mathfrak{m}'(\R) 
\oplus \mathfrak{a}'(\R) \oplus \mathfrak{n}_+'(\R)$ 
be a 
Gelfand--Naimark decomposition of $\mathfrak{g}'(\R)$,
and write $N_-' = \exp(\fn_-'(\R))$.
As for $C^\infty(G/P,\Cal{V})$, we regard $C^\infty(G'/P',\Cal{W})$
as a subspace of $C^\infty(N_-')\otimes W$.

We often view differential symmetry breaking operators
$\bD \colon
C^\infty(G/P,\mathcal{V})
\to C^\infty(G'/P',\mathcal{W})$
as differential operators 
$\w{\bD} \colon C^\infty(N_-) \otimes V
\to C^\infty(N_-') \otimes W$ 
such that
the restriction $\w{\bD}\vert_{C^\infty(G/P,\mathcal{V})}$
to $C^\infty(G/P,\mathcal{V})$ is a map
$\w{\bD}\vert_{C^\infty(G/P,\mathcal{V})}\colon
C^\infty(G/P,\mathcal{V})
\to C^\infty(G'/P',\mathcal{W})$ (see \eqref{eqn:21} below).
\begin{equation}\label{eqn:21}
\begin{aligned}
\xymatrix{
C^\infty(N_-) \otimes V 
\ar[r]^{\w{\bD} } 
& C^\infty(N_-') \otimes W\\ 
C^\infty(G/P,\mathcal{V}) 
\;\; \ar[r]_{\stackrel{\phantom{a}}{\hspace{20pt}\bD=\w{\bD} \vert_{\small{C^\infty(G/P,\mathcal{V})}}} }
 \ar@{^{(}->}[u]^{\iota^*}
& \;\; C^\infty(G'/P',\mathcal{W}) \ar@{^{(}->}[u]_{\iota^*}
}
\end{aligned}
\end{equation}

\noindent
In particular, we often regard $\Diff_{G'}(\Cal{V},\Cal{W})$ as 
\begin{equation}\label{eqn:DN}
\Diff_{G'}(\Cal{V},\Cal{W}) \subset 
\Diff_\C(C^\infty(N_-)\otimes V, C^\infty(N_-')\otimes W).
\end{equation}

\subsection{Duality theorem}\label{sec:duality}

For a finite-dimensional irreducible representation $(\sigma_\lambda,V)$ of $MA$,
we write $V^\vee = \Hom_\C(V,\C)$ and $((\sigma_\lambda)^\vee, V^\vee)$ for
the contragredient representation of $(\sigma_\lambda,V)$. By letting $\fn_+$ act 
on $V^\vee$ trivially, we regard the infinitesimal representation 
$d\sigma^\vee \boxtimes \C_{-\lambda}$ of $(\sigma_\lambda)^\vee$ as a $\fp$-module. For a finite-dimensional irreducible representation 
$(\varpi_\nu, W)$ of $M'A'$, a $\fp'$-module 
$d\varpi^\vee \boxtimes \C_{-\nu}$ is defined similarly.
We write
\begin{equation*}
\Mp(V^\vee) = \Cal{U}(\fg)\otimes_{\Cal{U}(\fp)}V^\vee
\quad
\text{and}
\quad
\Mpp(W^\vee) = \Cal{U}(\fg')\otimes_{\Cal{U}(\fp')}W^\vee
\end{equation*}
for generalized Verma modules for $(\fg, \fp)$ and $(\fg',\fp')$ induced from 
$d\sigma^\vee \boxtimes \C_{-\lambda}$ and $d\varpi^\vee \boxtimes \C_{-\nu}$,
respectively.
Via the diagonal action of $P$ on $\Mp(V^\vee)$, we regard
$M_\fp(V^\vee)$ as a $(\fg, P)$-module. Likewise, we regard 
$\Mpp(W^\vee)$ as a $(\fg', P')$-module.

The following theorem is often called the \emph{duality theorem}. 
For the proof, see \cite{KP1}.

\begin{thm}[Duality theorem]\label{thm:duality}
There is a natural linear isomorphism
\begin{equation}\label{eqn:duality1}
\EuD_{H\to D}
\colon
\operatorname{Hom}_{P'}(W^\vee,\Mp(V^\vee))
\stackrel{\sim}{\To} 
\operatorname{Diff}_{G'}(\mathcal V, \mathcal W).
\end{equation}
Equivalently,
\begin{equation}\label{eqn:duality2}
\EuD_{H\to D}
\colon
\operatorname{Hom}_{\fg', P'}(\Mpp(W^\vee),\Mp(V^\vee))
\stackrel{\sim}{\To} 
\operatorname{Diff}_{G'}(\mathcal V, \mathcal W).
\end{equation}
Here,
for $\varphi \in \Hom_{P'}(W^\vee, \Mp(V^\vee))$ and
$F \in C^\infty(G/P,\Cal{V})\simeq C^\infty(G,V)^P$,
the section $\EuD_{H\to D}(\varphi)F \in C^\infty(G'/P',\Cal{W})\simeq C^\infty(G',W)^{P'}$
is given by 
\begin{equation}\label{eqn:HD}
\IP{\EuD_{H\to D}(\varphi)F}{w^\vee} =
\sum_{j}\IP{dR(u_j)F}{v_j^\vee}\big\vert_{G'}
\;\; \text{for $w^\vee \in W^\vee$},
\end{equation}
where $\varphi(w^\vee)=\sum_j u_j\otimes v_j^\vee \in \Mp(V^\vee)$.
\end{thm}

\subsection{Algebraic Fourier transform $\widehat{\;\;\cdot\;\;}$ of Weyl algebras}
\label{sec:Weyl}

Let $U$ be a complex finite-dimensional vector space with $\dim _\C U=n$.
Fix a basis $u_1,\ldots, u_n$ of $U$ and let 
$(z_1, \ldots, z_n)$ denote the coordinates of $U$ with respect to the basis.
Then the algebra
\begin{equation*}
\C[U;z, \tfrac{\partial}{\partial z}]:=
\C[z_1, \ldots, z_n, 
\tfrac{\partial}{\partial z_1}, \ldots, \tfrac{\partial}{\partial z_n}]
\end{equation*}
with relations $z_iz_j = z_jz_i$, 
$\frac{\partial}{\partial z_i}\frac{\partial}{\partial z_j} 
=\frac{\partial}{\partial z_j}\frac{\partial}{\partial z_i}$,
and $\frac{\partial}{\partial z_j} z_i=\delta_{i,j} + z_i \frac{\partial}{\partial z_j}$
is called the Weyl algebra of $U$, where $\delta_{i,j}$ is the Kronecker delta.
Similarly,
let $(\zeta_1,\ldots, \zeta_n)$ denote the coordinates of 
the dual space $U^\vee$ of $U$ with respect to the dual basis of 
$u_1,\ldots, u_n$. We write 
$\C[U^\vee;\zeta, \tfrac{\partial}{\partial \zeta}]$
for the Weyl algebra of $U^\vee$.
Then the map determined by
\begin{equation}\label{eqn:Weyl}
\widehat{\frac{\partial}{\partial z_i}}:= -\zeta_i, \quad
\widehat{z_i}:=\frac{\partial}{\partial \zeta_i}
\end{equation}
gives a Weyl algebra isomorphism 
\begin{equation}\label{eqn:Weyl2}
\widehat{\;\;\cdot\;\;}\;\colon\C[U;z, \tfrac{\partial}{\partial z}] 
\stackrel{\sim}{\To} 
\C[U^\vee;\zeta, \tfrac{\partial}{\partial \zeta}], 
\quad T \mapsto \widehat{T}.
\end{equation}
\vskip 0.1in
\noindent
The map \eqref{eqn:Weyl2} is called
the \emph{algebraic Fourier transform of Weyl algebras}
(\cite[Def.\ 3.1]{KP1}). We remark that the minus sign for ``$-\zeta_i$'' 
in \eqref{eqn:Weyl} is put in such a way that the resulting map $\widehat{\;\;\cdot\;\;}$ 
is indeed a Weyl algebra homomorphism.

\subsection{Fourier transformed representation $\widehat{d\pi_{(\sigma,\lambda)^*}}$}
\label{sec:dpi2}

For a representation $\eta$ of $G$, we denote by $d\eta$ 
the infinitesimal representation of $\fg(\R)$. 
For instance,  $dL$ and $dR$ denote the infinitesimal representations $\fg(\R)$ of
the left and right regular representations of $G$ on $C^\infty(G)$.
As usual, we naturally extend representations of $\fg(\R)$ 
to ones for its universal enveloping algebra $\Cal{U}(\fg)$ of its complexification $\fg$.
The same convention is applied for closed subgroups of $G$.

For $g \in N_-MAN_+$, we write
\begin{equation*}
g = p_-(g)p_0(g)p_+(g),
\end{equation*}
where $p_\pm(g) \in N_{\pm}$ and $p_0(g) \in MA$. 
Similarly, for $Y \in \fg = \fn_- \oplus \fl\oplus \fn_+$ with $\fl= \fm \oplus \fa$,
we write
\begin{equation*}
Y=Y_{\fn_-} + Y_{\fl} + Y_{\fn_+},
\end{equation*}
where $Y_{\fn_{\pm}} \in \fn_{\pm}$ and $Y_\fl \in \fl$.

For $2\rho\equiv 2\rho(\fn_+)= \mathrm{Trace}(\ad\vert_{\fn_+})\in \mathfrak{a}^*$,
we denote by  $\C_{2\rho}$ the one-dimensional representation of $P$
defined by
$p \mapsto \chi_{2\rho}(p)=
\abs{\mathrm{det}(\mathrm{Ad}(p)\colon 
\mathfrak{n}_+ \to \mathfrak{n}_+)}$.
For the contragredient representation
$((\sigma_\lambda)^\vee, V^\vee)$ of $(\sigma_\lambda,V)$,
we put
$\sigma^*_\lambda := \sigma^\vee \boxtimes \C_{2\rho-\lambda}$.
As for $\sigma_\lambda$, we regard $\sigma^*_\lambda$ as a representation of $P$.
Define the induced representation
$\pi_{(\sigma, \lambda)^*} = \mathrm{Ind}_P^G(\sigma^*_\lambda)$
on the space $C^\infty(G/P,\mathcal{V}^*)$ of smooth sections 
for the vector bundle $\mathcal{V}^*=G\times_P (V^\vee\otimes \C_{2\rho})$
associated with $\sigma^*_\lambda$, which is isomorphic to the tensor bundle
of the dual vector bundle $\mathcal{V}^\vee = G\times_PV^\vee$ and the bundle
of densities over $G/P$.
Then the integration on $G/P$ gives a 
$G$-invariant non-degenerate bilinear form
$
\mathrm{Ind}^G_P(\sigma_\lambda) 
\times \mathrm{Ind}^G_P(\sigma^*_\lambda) \to \C
$
for $\mathrm{Ind}^G_P(\sigma_\lambda)$ and 
$\mathrm{Ind}^G_P(\sigma^*_\lambda)$.

As for $C^\infty(G/P,\mathcal{V})$,
the space $C^\infty(G/P,\mathcal{V}^*)$ can be regarded as 
a subspace of $C^\infty(N_-) \otimes V^\vee$.
Then the infinitesimal representation $d\pi_{(\sigma,\lambda)^*}(X)$
on $C^\infty(N_-)\otimes V^\vee$
for $X \in \fg$ is given by
\begin{equation}\label{eqn:dpi3}
d\pi_{(\sigma,\lambda)^*}(X)f(\bar{n})
=d\sigma_\lambda^*((\Ad(\bar{n}^{-1})X)_\fl)f(\bar{n})
-\left(dR((\Ad(\cdot^{-1})X)_{\fn_-})f\right)(\bar{n}).
\end{equation}
(For the details, see, for instance, \cite[Sect.\ 2]{KuOr24}.)
Via the exponential map $\exp\colon \fn_-(\R) \simeq N_-$, one can regard
$\dpi_{(\sigma,\lambda)^*}(X)$ as a representation 
on $C^\infty(\mathfrak{n}_-(\R)) \otimes V^\vee$.
It then follows from \eqref{eqn:dpi3} that 
$\dpi_{(\sigma,\lambda)^*}$ gives a Lie algebra homomorphism
\begin{equation*}
d\pi_{(\sigma,\lambda)^*}\colon \mathfrak{g} \To 
\C[\fn_-(\R);x, \tfrac{\partial}{\partial x}]  \otimes \mathrm{End}(V^\vee),
\end{equation*}
where $(x_1, \ldots, x_n)$ are coordinates of $\fn_-(\R)$ with $n=\dim \fn_-(\R)$.
We extend the coordinate functions $x_1,\ldots, x_n$ for $\fn_-(\R)$ 
holomorphically to the ones $z_1, \ldots, z_n$ for $\fn_-$.
Thus we have 
\begin{equation*}
d\pi_{(\sigma,\lambda)^*}\colon \mathfrak{g} \To 
\C[\fn_-;z, \tfrac{\partial}{\partial z}]  \otimes \mathrm{End}(V^\vee).
\end{equation*}

Now we fix a non-degenerate $\Ad$-invariant symmetric bilinear form $\kappa$ on
$\fg$. Via $\kappa$, we identify 
$\fn_+$  with the dual space $\fn_-^\vee$ of 
$\fn_-$. Then
the algebraic Fourier transform $\widehat{\;\;\cdot\;\;}$
of Weyl algebras \eqref{eqn:Weyl2} gives a Weyl algebra isomorphism
\begin{equation*}
\widehat{\;\;\cdot\;\;}\;\colon
\C[\fn_-;z, \tfrac{\partial}{\partial z}] 
\stackrel{\sim}{\To} 
\C[\fn_+;\zeta, \tfrac{\partial}{\partial \zeta}].
\end{equation*}
In particular, it gives a Lie algebra homomorphism
\begin{equation}\label{eqn:hdpi}
\widehat{d\pi_{(\sigma,\lambda)^*}}\colon \mathfrak{g} \To 
\C[\fn_+;\zeta, \tfrac{\partial}{\partial \zeta}]\otimes \mathrm{End}(V^\vee).
\end{equation}

Now we define a map
\begin{equation}\label{eqn:Fc}
F_c\colon M_\fp(V^\vee) 
\To
\Pol(\fn_+) \otimes V^\vee,
\quad u\otimes v^\vee \longmapsto 
\widehat{d\pi_{(\sigma,\lambda)^*}}(u)(1\otimes v^\vee).
\end{equation}

\begin{thm}[{\cite[Sect.\ 3.4]{KP1}}]\label{def:AFT2}
The map $F_c$  is  a $(\mathfrak{g}, P)$-module isomorphism.
\end{thm}

We call 
the $(\mathfrak{g}, P)$-module isomorphism $F_c$ in \eqref{eqn:Fc} 
the \emph{algebraic Fourier transform of the generalized Verma module  $M_\fp(V^\vee)$}.

\subsection{The F-method}
\label{sec:Fmethod2}
Observe that
the algebraic Fourier transform
$F_c$  in \eqref{eqn:Fc} 
gives an $M'A'$-representation isomorphism
\begin{equation}\label{eqn:VP2}
\Mp(V^\vee)^{\fn_+'}
\stackrel{\sim}{\To}
(\Pol(\fn_+) 
\otimes V^\vee)^{\widehat{d\pi_{(\sigma,\lambda)^*}}(\fn_+')},
\end{equation}
which induces a linear isomorphism
\begin{equation}\label{eqn:VP31}
\operatorname{Hom}_{M'A'}(W^\vee,
\Mp(V^\vee)^{\fn_+'})
\stackrel{\sim}{\To}
\operatorname{Hom}_{M'A'}
\left(W^\vee,
(\mathrm{Pol}(\mathfrak{n}_+) 
\otimes V^\vee)^{\widehat{d\pi_{(\sigma,\lambda)^*}}(\fn_+')}\right).
\end{equation}
Here $M'A'$ acts on $\Pol(\fn_+)$ via the action
\begin{equation}\label{eqn:sharp}
\Ad_{\#}(l) \colon p(X) \mapsto p(\Ad(l^{-1})X)
\;\; \text{for $l \in M'A'$}.
\end{equation}

Now we set
\begin{equation}\label{eqn:Sol2a}
\mathrm{Sol}(\mathfrak n_+;V,W):=
\operatorname{Hom}_{M'A'}(W^\vee,
(\mathrm{Pol}(\mathfrak{n}_+) 
\otimes V^\vee)^{\widehat{d\pi_{(\sigma,\lambda)^*}}(\fn_+')}).
\end{equation}
Via the identification
$\textrm{Hom}_{M'A'}(W^\vee, \Pol(\fn_+)\otimes V^\vee)
\simeq
\big((\Pol(\fn_+)\otimes V^\vee) \otimes W\big)^{M'A'}$,
we have
\begin{align}\label{eqn:Sol}
&\Sol(\fn_+;V,W)\nonumber \\[3pt]
&=\{ \psi \in \Hom_{M'A'}(W^\vee, \Pol(\fn_+)\otimes V^\vee): 
\text{
$\psi$ satisfies the system \eqref{eqn:Fsys} of PDEs below.}\}
\end{align}
\begin{equation}\label{eqn:Fsys}
(\widehat{d\pi_{(\sigma,\lambda)^*}}(C) \otimes \id_W)\psi=0
\,\,
\text{for all $C \in \fn_+'$}.
\end{equation}
We refer to the system \eqref{eqn:Fsys} of PDEs 
as the \emph{F-system} (\cite[Fact 3.3 (3)]{KKP16}).
Since
\begin{equation*}
\operatorname{Hom}_{P'}(W^\vee,\Mp(V^\vee))
=
\operatorname{Hom}_{M'A'}(W^\vee,\Mp(V^\vee)^{\fn_+'}),
\end{equation*}
the isomorphism \eqref{eqn:VP31} together with \eqref{eqn:Sol2a} shows the following.

\begin{thm}
[F-method, {\cite[Thm.\ 4.1]{KP1}}]
\label{thm:Fmethod}
There exists a linear isomorphism
\begin{equation*}
F_c \otimes \mathrm{id}_{W}\colon 
\operatorname{Hom}_{P'}(W^\vee,
\Mp(V^\vee))
\stackrel{\sim}{\To}
\mathrm{Sol}(\mathfrak n_+;V,W).
\end{equation*}
Equivalently, we have
\begin{equation*}
F_c \otimes \mathrm{id}_{W}\colon 
\mathrm{Hom}_{\mathfrak g',P'}(\Mpp(W^\vee),\Mp(V^\vee))\stackrel{\sim}{\To}
\mathrm{Sol}(\mathfrak n_+;V,W).
\end{equation*}
\end{thm}

\subsection{The case of abelian nilradical $\fn_+$}
\label{sec:abelian}

Now suppose that the nilpotent radical $\fn_+$ is abelian.
In this case differential symmetry breaking operators
$\bD \in \Diff_{G'}(\Cal{V},\Cal{W})$ have constant coefficients.
(See, for instance, \cite[Sect.\ 2]{KuOr24}.)
Thus, as in \eqref{eqn:DN}, one may view $\Diff_{G'}(\Cal{V},\Cal{W})$ as
\begin{equation*}
\Diff_{G'}(\Cal{V},\Cal{W}) \subset \C[\fn_-;\tfrac{\partial}{\partial z}]
 \otimes \mathrm{Hom}_{\C}(V,W).
\end{equation*}
Since $\fn_+$ is regarded as the dual space of $\fn_-$, 
one can define the symbol map
\begin{equation*}
\mathrm{symb}\colon
\C[\fn_-;\tfrac{\partial}{\partial z}]
\To\C[\fn_+;\zeta],
\quad 
\tfrac{\partial}{\partial z_i} \mapsto \zeta_i.
\end{equation*}
\vskip 0.1in

Theorem \ref{thm:abelian} below shows a beautiful relationship 
between the F-method and the symbol map.

\begin{thm}
[{\cite[Cor.\ 4.3]{KP1}}]
\label{thm:abelian}
Suppose that the nilpotent radical
$\mathfrak{n}_+$ is abelian. Then 
the symbol map $\mathrm{symb}$ gives
a linear isomorphism
\begin{equation*}
\Rest \circ\mathrm{symb}^{-1}
\colon 
\mathrm{Sol}(\mathfrak n_+;V,W)
\stackrel{\sim}{\To} 
 \operatorname{Diff}_{G'}(\mathcal V, \mathcal W).
\end{equation*}
Further, the diagram \eqref{eqn:isom3} below commutes:
\begin{equation}\label{eqn:isom3}
\xymatrix@R-=0.7pc@C-=0.5cm{
{}
& 
\mathrm{Sol}(\mathfrak n_+;V,W)
\ar[rddd]_{\sim}^{\;\; \Rest\, \circ \, \mathrm{symb}^{-1}}
&{} \\
&&&\\
&\circlearrowleft{} \\
\operatorname{Hom}_{\mathfrak{g}',P'}(\Mpp(W^\vee),\Mp(V^\vee)) 
  \ar[rr]^{\hspace{25pt}\sim}_{\hspace{25pt}\EuD_{H\to D}}
  \ar[ruuu]_{\sim}^{F_c\otimes\mathrm{id}_W} 
 & {}
 & 
 \operatorname{Diff}_{G'}(\mathcal V, \mathcal W),
}
\end{equation}
where $\Rest$ denotes the restriction map 
from $C^\infty(G/P',\Cal{W})$ to $C^\infty(G'/P',\Cal{W})$.
\end{thm}

\subsection{A recipe of the F-method for abelian nilradical $\fn_+$}
\label{sec:recipe}

By \eqref{eqn:Sol} and Theorem \ref{thm:abelian},
one can classify and construct 
$\bD \in  \operatorname{Diff}_{G'}(\mathcal V, \mathcal W)$ and
$\Phi \in 
\Hom_{\mathfrak{g}',P'}(\Mpp(W^\vee),\Mp(V^\vee))$
by computing $\psi \in \mathrm{Sol}(\mathfrak n_+;V,W)$ as follows.
\vskip 0.1in

\begin{enumerate}

\item[Step 1]
Compute $d\pi_{(\sigma,\lambda)^*}(C)$ and 
$\widehat{d\pi_{(\sigma,\lambda)^*}}(C)$
for $C \in \fn_+'$.
\vskip 0.1in

\item[Step 2]
Classify and construct 
$\psi \in \Hom_{M'A'}(W^\vee, \Pol(\fn_+)\otimes V^\vee)$.
\vskip 0.1in

\item[Step 3]
Solve the F-system \eqref{eqn:Fsys}
for $\psi \in \Hom_{M'A'}(W^\vee, \Pol(\fn_+)\otimes V^\vee)$.
\vskip 0.1in

\item[Step 4]
For $\psi \in \mathrm{Sol}(\mathfrak n_+;V,W)$ obtained in Step 3,
do the following.
\vskip 0.1in

\begin{enumerate}
\item[Step 4a]
Apply $\Rest \circ \symb^{-1}$ 
to $\psi \in \mathrm{Sol}(\mathfrak n_+;V,W)$
to obtain $\bD \in  \operatorname{Diff}_{G'}(\mathcal V, \mathcal W)$.
\vskip 0.1in
\item[Step 4b]
Apply $F_c^{-1} \otimes \id_W$ 
to $\psi \in \mathrm{Sol}(\mathfrak n_+;V,W)$
to obtain $\Phi \in \operatorname{Hom}_{\mathfrak{g}',P'}
(\Mpp(W^\vee), \Mp(V^\vee))$.
\end{enumerate}
\vskip 0.1in

\end{enumerate}

\section{Specialization to $(SL(n+1,\R), SL(n,\R);P, P')$}
\label{sec:SL}

In this section we specialize the general framework 
described in Section \ref{sec:Fmethod}
to the case $(G,G')=(SL(n+1,\R),SL(n,\R))$ with real flag varieties
 $G/P\simeq \RP^n$ and $G'/P'\simeq \RP^{n-1}$. 
Throughout this section we assume $n\geq 2$, unless otherwise specified.

\subsection{Notation}\label{sec:notation}

Let $G = SL(n+1,\R)$ with Lie algebra $\fg(\R)=\f{sl}(n+1,\R)$ for $n\geq 2$.
Let $G'$ denote the closed subgroup of $G$ such that
\begin{equation*}
G'=
\left\{
\begin{pmatrix}
g'&\\
& 1\\
\end{pmatrix}
:
g' \in SL(n,\R)
\right\}\simeq SL(n,\R)
\end{equation*}
with Lie algebra
\begin{equation}\label{eqn:g'}
\fg'(\R)= \left\{
\begin{pmatrix}
X' & \\
 & 0
\end{pmatrix}
:
X' \in \f{sl}(n,\R)
\right\}\simeq \f{sl}(n,\R).
\end{equation}

We put
\begin{equation*}
N_j^+:=E_{1,j+1},
\quad
N_j^-:=E_{j+1,1}
\quad
\text{for $j\in \{1,\ldots, n\}$}
\end{equation*}
and 
\begin{alignat}{2}
H_0&:=\frac{1}{n+1}(nE_{1,1} - \sum^{n+1}_{r=2}E_{r,r})
&&=\frac{1}{n+1}\diag(n, -1, -1, \ldots, -1,-1),\label{eqn:H0}\\
H_0'&:=\frac{1}{n}((n-1)E_{1,1} - \sum^n_{r=2}E_{r,r})
&&=\frac{1}{n}\diag(n-1, -1, -1, \ldots, -1,0),\label{eqn:H0'}
\end{alignat}
where $E_{i,j}$ denote the matrix units. We normalize $H_0$ and $H_0'$ as
$\wH_0:=\frac{n+1}{n}H_0$ and $\wH_0':=\frac{n}{n-1}H_0'$,
namely,
\begin{alignat}{2}
\wH_0&=\frac{1}{n}(nE_{1,1} - \sum^{n+1}_{r=2}E_{r,r})
&&=\frac{1}{n}\diag(n, -1, -1, \ldots, -1,-1),\label{eqn:0801a}\\
\wH_0'&=\frac{1}{n-1}((n-1)E_{1,1} - \sum^n_{r=2}E_{r,r})
&&=\frac{1}{n-1}\diag(n-1, -1, -1, \ldots, -1,0).\label{eqn:0801b}
\end{alignat}
Let
\begin{alignat}{2}\label{eqn:nR}
\fn_{\pm}(\R) &=\Ker(\ad(H_0)\mp\id) 
&&= \Ker(\ad(\wH_0)\mp\tfrac{n+1}{n}\id),\\
\fn_{\pm}'(\R)&=\Ker(\ad(H_0')\vert_{\fg'(\R)}\mp\id) 
&&=\Ker(\ad(\wH_0')\vert_{\fg'(\R)}\mp\tfrac{n}{n-1}\id).\nonumber
\end{alignat}
Then we have
\begin{equation*}\label{eqn:nR2}
\fn_{\pm}(\R)=\spn_{\R}\{N_1^{\pm},\ldots, N_{n-1}^{\pm}, N_{n}^{\pm} \}
\quad
\text{and}
\quad
\fn_{\pm}(\R)=\spn_{\R}\{N_1^{\pm},\ldots, N_{n-1}^{\pm} \}.
\end{equation*}

For $X, Y \in \fg(\R)$, let $\Tr(X,Y)=\text{Trace}(XY)$ denote the trace form of $\fg(\R)$. 
Then $N_i^+$ and $N_j^-$ satisfy $\Tr(N_i^+,N_j^-)=\delta_{i,j}$.
In what follows, we identify the duals $\fn_-(\R)^\vee$ of $\fn_-(\R)$
and $\fn_-'(\R)^\vee$ of $\fn_-'(\R)$
with $\fn_-(\R)^\vee \simeq \fn_+(\R)$ and 
$\fn_-'(\R)^\vee \simeq \fn_+'(\R)$
via the trace form $\Tr(\cdot, \cdot)$.

Let $\fa(\R)= \R \wH_0$ and $\fa'(\R)=\R\wH_0'$.
We also put
\begin{alignat*}{2}
\fm(\R)&:= \left\{
\begin{pmatrix}
0 & \\
 & X
\end{pmatrix}
:
X \in \f{sl}(n,\R)
\right\}
&&\simeq \f{sl}(n,\R),\\
\fm'(\R)&:= \left\{
\begin{pmatrix}
0 && \\
 & X' &\\
 & & 0
\end{pmatrix}
:
X' \in \f{sl}(n-1,\R)
\right\}&&\simeq \f{sl}(n-1,\R).
\end{alignat*}
Here $\f{sl}(1,\R)$ is regarded as $\f{sl}(1,\R)=\{0\}$. 
We remark that although $\fg'(\R) \simeq \fm(\R) \simeq \f{sl}(n,\R)$, these are 
different subalgebras of $\fg(\R)$.

We have 
$\fm(\R)\oplus \fa(\R) = \Ker(\ad(\widetilde{H}_0))$
 and
 $\fm'(\R)\oplus \fa'(\R) = \Ker(\ad(\widetilde{H}_0')\vert_{\fg'(\R)})$.
 The decompositions
$\fg(\R) = \fn_-(\R) \oplus \fm(\R) \oplus \fa(\R) \oplus \fn_+(\R)$
and
$\fg'(\R) = \fn_-'(\R) \oplus \fm'(\R) \oplus \fa'(\R) \oplus \fn_+'(\R)$
are Gelfand--Naimark decompositions of $\fg(\R)$ and $\fg'(\R)$, resepectively.
The subalgebras
$\fp(\R):=\fm(\R) \oplus \fa(\R) \oplus \fn_+(\R)$ and
$\fp'(\R):=\fm'(\R) \oplus \fa'(\R) \oplus \fn_+'(\R)$ 
are maximal parabolic 
subalgebras of $\fg(\R)$ and $\fg'(\R)$, respectively.
It is remarked that $\fn_{\pm}(\R)$ and $\fn_{\pm}'(\R)$ are abelian. 

Let $P$ be  the normalizer $N_G(\fp(\R))$ of $\fp(\R)$ in $G$
and $P'$ the normalizer $N_{G'}(\fp'(\R))$ of $\fp'(\R)$ in $G'$.
We write $P=MAN_+$ and $P'=M'A'N_+'$
for the Langlands decompositions of $P$ and $P'$ corresponding to 
$\fp(\R)=\fm(\R) \oplus \fa(\R) \oplus \fn_+(\R)$ and 
$\fp'(\R)=\fm'(\R) \oplus \fa'(\R) \oplus \fn_+'(\R)$,
respectively.
Then we have
\begin{alignat*}{3}
&A = \exp(\fa(\R)) &&= \exp(\R \wH_0), \qquad 
&&N_+=\exp(\fn_+(\R)),\\
&A' = \exp(\fa'(\R)) &&= \exp(\R \wH_0'), \qquad
&&N_+'=\exp(\fn_+'(\R)). 
\end{alignat*}
The groups $M$ and $M'$ are  given by
\begin{alignat*}{2}
M&= 
\left\{
\begin{pmatrix}
\det(g)^{-1} &\\
& g\\
\end{pmatrix}
:
g \in SL^{\pm}(n,\R)
\right\}
&&\simeq SL^{\pm}(n,\R),\\
M'&= 
\left\{
\begin{pmatrix}
\det(g')^{-1} & & \\
& g' &\\
&&1
\end{pmatrix}
:
g' \in SL^{\pm}(n-1,\R)
\right\}
&&\simeq SL^{\pm}(n-1,\R).
\end{alignat*}
Here $SL^{\pm}(1,\R)$ is regarded as
$SL^{\pm}(1,\R) = \{\pm 1\}$.
As $M$ and $M'$ are not connected, 
let $M_0$ and $M_0'$ denote the identity components of $M$ and $M'$,
respectively.
Then $M_0 \simeq SL(n,\R)$ and $M_0' \simeq SL(n-1,\R)$.
For
\begin{equation*}
\gamma = \diag(-1, 1, \ldots, 1, -1, 1) \in M' \subset M,
\end{equation*}
we have
\begin{equation*}
M/M_0 = \{[I_{n+1}], [\gamma]\} \simeq \Z/2\Z
\quad
\text{and}
\quad
M'/M_0' = \{[I_{n+1}]', [\gamma]'\} \simeq \Z/2\Z,
\end{equation*}
where $I_{n+1}$ denotes the $(n+1) \times (n+1)$ identity matrix and 
$[g] = g M_0$ and $[g]' = g M'_0$.

\begin{rem}\label{rem:A}
We have
$P' \subset P$, $M'A' \subset MA$, $M' \subset M$, $N_{\pm}'\subset N_{\pm}$,
but $A' \not \subset A$. Indeed, each element
\begin{equation*}
a' = \diag(t, t^{\frac{-1}{n-1}}, \ldots, t^{\frac{-1}{n-1}}, 1) \in A'
\end{equation*}
can be decomposed as
$a' = a'_M a'_A$ with $a'_M \in M$ and $a'_A \in A$, where
\begin{align*}
a'_M 
&= \diag(1,t^{\frac{-1}{n(n-1)}},\ldots, t^{\frac{-1}{n(n-1)}} ,t^{\frac{1}{n}})
\in M,\\
a'_A
&= \diag(t,t^{\frac{-1}{n}},\ldots ,t^{\frac{-1}{n}} ,t^{\frac{-1}{n}}) 
\in A.
\end{align*}

\end{rem}


For a closed subgroup $J$ of $G$, we denote by $\Irr(J)$ and $\Irr(J)_{\fin}$
the sets of equivalence classes of irreducible representations of $J$  and 
finite-dimensional irreducible representations of $J$, respectively.

For $\lambda,\nu \in \C$, 
we define one-dimensional representations 
$\C_\lambda=(\chi^\lambda, \C)$ of $A=\exp(\R\wH_0)$ 
and
$\C_\nu = ((\chi')^\nu,\C)$ of $A' = \exp(\R \wH_0')$  by
\begin{equation}\label{eqn:chi}
\chi^\lambda \colon \exp(t \wH_0) \longmapsto \exp(\lambda t)
\quad
\text{and}
\quad
(\chi')^\nu \colon \exp(t \wH_0') \longmapsto \exp(\nu t).
\end{equation}
Then $\Irr(A)$ and $\Irr(A')$ are given by
\begin{equation*}
\Irr(A)=\{\C_\lambda : \lambda \in \C\} \simeq \C
\quad
\text{and}
\quad
\Irr(A')=\{\C_\nu : \nu \in \C\} \simeq \C.
\end{equation*}

Let  $\sgn$ denote the sign character of $\R^{\times}$.
For $\alpha \in \{\pm\}$, we then define a one-dimensional representation 
$(\sgn^\ga, \C)$ of $M$  by
\begin{equation}\label{eqn:20241106}
\begin{pmatrix}
\det(g)^{-1} &\\
& g\\
\end{pmatrix}
\longmapsto
\sgn^\alpha(\det(g)).
\end{equation}
where
\begin{equation}\label{eqn:20241108}
\sgn^\alpha(\det(g)) 
:= 
\begin{cases}
1 & \text{if $\alpha=+$},\\
\sgn(\det(g)) & \text{if $\alpha=-$}.
\end{cases}
\end{equation}
Then 
$\Irr(M)_\fin = \Irr(SL^{\pm}(n,\R))_\fin$
and
$\Irr(M')_\fin = \Irr(SL^{\pm}(n-1,\R))_\fin$
are given by
\begin{align*}
\Irr(M)_{\fin}&\simeq
\{\sgn^\ga \otimes \xi: 
(\alpha, \xi) \in \{\pm\} \times \Irr(SL(n,\R))_{\fin}\},\\
\Irr(M')_{\fin}&\simeq
\{\sgn^\beta \otimes \varpi: 
(\beta, \varpi) \in \{\pm\} \times \Irr(SL(n-1,\R))_{\fin}\}.
\end{align*}

Since 
$\Irr(P)_{\fin}\simeq \Irr(M)_{\fin} \times \Irr(A)$, the set $\Irr(P)_\fin$ can be 
parametrized as
\begin{equation*}
\Irr(P)_{\fin}\simeq 
\{\pm\} \times \Irr(SL(n,\R))_{\fin} \times \C.
\end{equation*}
Likewise, we have
\begin{equation*}
\Irr(P')_{\fin}\simeq 
\{\pm\} \times \Irr(SL(n-1,\R))_{\fin} \times \C.
\end{equation*}
In particular,
for $n=2$, we  have 
\begin{equation}\label{eqn:n2}
\Irr(P')_{\fin}\simeq 
\{\pm\} \times \{\triv\} \times \C.
\end{equation}

\begin{rem}\label{rem:sym}
One needs to be careful 
that the parametrization of  $\Irr(SL^{\pm}(1,\R))$ is not unique.
Indeed, for $m\geq 1$,
we define $\gamma_m \in SL^{\pm}(m,\R)$ as 
\begin{equation*}
\gamma_m= 
\begin{cases}
\diag(1,\ldots,1,-1) & \text{for $m\geq 2$}, \\
-1 &\text{for $m=1$}, 
\end{cases}
\end{equation*}
so that $SL^{\pm}(m,\R) = \langle \gamma_m \rangle \ltimes SL(m,\R)$.
Let $\sym^k_{m}$ denote the irreducible representation of $SL(m,\R)$ on
$S^k(\C^m)$. Here, we regard $\sym^k_1$ with  $\sym^k_1 = \triv$ 
as the irreducible representation of $SL(1,\R) = \{1\}$ on $\C$ 
for all $k \in \Z_{\geq 0}$.

Now consider $\sgn^\delta \otimes \sym^k_m \in \Irr(SL^\pm(m,\R))_{\fin}$ for 
$m\geq 1$. 
The irreducible $SL^{\pm}(m,\R)$-representation
$\sgn^{\delta} \otimes \sym^k_m$ is defined as
\begin{equation*}
(\sgn^\delta\otimes \sym^k_m)(g)e_1^{k_1}\cdots e_m^{k_m}
=\sgn^\delta(g)(ge_1)^{k_1}\cdots (ge_m)^{k_m},
\end{equation*}
where $e_j$ are the standard basis elements of $\C^m$.
For $m=1$, we then have 
\begin{align*}
(\sgn^\delta \otimes \sym^k_1)(\pm 1) 1^k
&=\sgn(\pm 1)^\delta (\pm 1 \cdot 1)^k\\
&=\sgn(\pm 1)^\delta (\pm 1)^k 1^k\\
&=\sgn(\pm 1)^{\delta +k }1^k\\
&=(\sgn^{\delta+k}\otimes \triv)(\pm 1) 1^k.
\end{align*}
Here, 
for $\delta \in \{\pm\}\equiv \{\pm 1\}$ 
and $k \in \Z_{\geq 0}$, we mean 
$\delta + k \in \{\pm\}$ by
\begin{equation*}
\delta+k = 
\begin{cases}
+ & \text{if $\delta = (-1)^k$},\\
- & \text{if $\delta = (-1)^{k+1}$}.
\end{cases}
\end{equation*}
Therefore, we have 
\begin{equation}\label{eqn:sym0826a}
\sgn^{\delta+k} \otimes \sym^k_1 =
\sgn^{\delta}\otimes \triv.
\end{equation}

Likewise, let $\poly^k_{m}$ denote the irreducible representation 
of $SL(m,\R)$ on the space $\text{Pol}^k(\C^m)$ 
of polynomial functions on $\C^n$ of homogeneous degree $k$.
As for $\sym^k_m$, we regard $\poly^k_{1}$ 
as $\poly^k_1 = \triv$ for all $k \in \Z_{\geq 0}$.
Since $\poly^k_m \simeq (\sym^k_m)^\vee$,
the identity \eqref{eqn:sym0826a} implies that
\begin{equation}\label{eqn:poly0826b}
\sgn^{\delta+k} \otimes \poly^k_1 =
\sgn^{\delta}\otimes \triv.
\end{equation}

\end{rem}

For 
$(\ga, \xi, \lambda) \in \{\pm\} \times \Irr(SL(n,\R))_{\fin} \times \C$
and
$(\gb, \varpi, \nu) \in \{\pm\} \times \Irr(SL(n-1,\R))_{\fin} \times \C$,
we write
\begin{equation}\label{eqn:Ind}
I(\xi,\lambda)^\ga 
= 
\Ind_{P}^G\left((\sgn^\ga \otimes \xi)\boxtimes \C_\lambda\right)
\quad
\text{and}
\quad
J(\varpi,\nu)^\gb 
= 
\Ind_{P'}^{G'}\left((\sgn^\gb \otimes \varpi)\boxtimes \C_\nu\right)
\end{equation}
for (unnormalized) parabolically induced representations 
of $G$ and of $G'$, respectively.
For instance, the unitary axis of 
$J(\triv, \nu)^\gb$ is $\text{Re}(\nu) = \frac{n}{2}$, 
where $\triv$ denotes the trivial representation of $SL(n-1,\R)$.
(See, for instance, \cite[p.\ 102]{vDM99} and \cite[p.\ 298]{HL99}. 
We remark that the complex parameters ``$\mu$'' in \cite{vDM99},
``$\alpha$'' in \cite{HL99}, and ``$\nu$'' in this paper are related as 
$\mu= \alpha = -\nu$.)

For the representation spaces $V$ and $W$
of 
$(\ga, \xi, \lambda) \in \Irr(P)_{\fin}$
and of $(\gb, \varpi, \nu) \in \Irr(P')_{\fin}$, respectively,
we write 
\begin{equation}\label{eqn:Verma}
\Mp(\xi,\lambda)^\ga 
=\Cal{U}(\fg) \otimes_{\Cal{U}(\fp)}V
\quad
\text{and}
\quad
\Mpp(\varpi,\nu)^\gb 
=\Cal{U}(\fg') \otimes_{\Cal{U}(\fp')}W.
\end{equation}

In the next section we classify and construct 
differential symmetry breaking operators
\begin{equation*}
\bD \in \Diff_{G'}(I(\triv, \lambda)^\ga, J(\varpi,  \nu)^\beta)
\end{equation*}
and $(\fg', P')$-homomorphisms
\begin{equation*}
\Phi \in \Hom_{\fg', P'}(\Mpp(\varpi,\nu)^\ga, \Mp(\triv, \lambda)^\gb).
\end{equation*}

If $n=2$, then \eqref{eqn:poly0826b} and \eqref{eqn:sym0826a} show that, 
for $k \in \Z_{\geq 0}$, we have
\begin{align}
J(\triv, \nu)^\beta
&=
\Ind_{P'}^{SL(2,\R)}\left((\sgn^\gb \otimes \triv)\boxtimes \C_\nu\right) \nonumber\\
&=
\Ind_{P'}^{SL(2,\R)}((\sgn^{\gb+k} \otimes \poly^k_1)\boxtimes \C_\nu)
\nonumber\\
&=
J(\poly^k_1, \nu)^{\beta+k} \label{eqn:poly0826c}
\end{align}
and
\begin{equation}\label{eqn:sym0826c}
M^{\f{sl}(2,\C)}_{\fp'}(\triv,\nu)^\gb
=
M^{\f{sl}(2,\C)}_{\fp'}(\sym^k_1,\nu)^{\gb+k}.
\end{equation}

\section{Differential symmetry breaking operators $\bD$ and $(\fg',P')$-homomorphisms $\Phi$}
\label{sec:results}

The objective of this section is to state the main results of 
the classification and 
construction of  
differential symmetry breaking operators 
$\bD \in \Diff_{G'}(I(\triv, \lambda)^\ga, J(\varpi,  \nu)^\beta)$ as well as 
$(\fg',P')$-homomorphism 
$\Phi \in \Hom_{\fg', P'}(\Mpp(\varpi,\nu)^\ga, \Mp(\triv, \lambda)^\gb)$.
These are achieved in Theorems
\ref{thm:DSBO2a} and
\ref{thm:DSBO2b}
for $\bD$ and Theorems 
\ref{thm:Hom2a}
and
\ref{thm:Hom2b}
for $\Phi$.
In addition, we also present the classification of $\fg'$-homomorphisms between
generalized Verma modules in Section \ref{subsec:g-hom}.

The proofs  will be discussed in detail in Sections \ref{sec:proof1} and 
\ref{sec:proof2} in accordance with the recipe of the F-method. 
In this section we assume $n\geq 2$, unless otherwise 
specified.

\subsection{Differential symmetry breaking operators $\bD$}
\label{sec:DSBO}

We start with the classification and construction results of 
differential symmetry breaking operators $\bD$.
In this subsection,
let $(\ga, \gb;\varpi;\lambda,\nu) \in \{\pm\}^2 \times \Irr(SL(n-1,\R))_{\fin} \times \C^2$ parametrize a pair $(I(\triv,\lambda)^\ga, J(\varpi,\nu)^\beta)$ of induced representations
of $G$ and $G'$, respectively. 
If $n=2$, then, via the identity \eqref{eqn:poly0826c}, we regard
$(\ga,\gb;\triv;\lambda,\nu)$ as 
\begin{equation}\label{eqn:poly0826d}
(\ga,\gb;\triv;\lambda,\nu) = (\ga,\gb+k;\poly^k_1;\lambda,\nu)
\end{equation}
for $k \in \Z_{\geq 0}$. 
We then define subsets
\begin{equation*}
\gL^{(n+1,n)}_{SL,j}\subset \{\pm\}^2 \times \Irr(SL(n-1,\R))_{\fin} \times \C^2
\end{equation*}
for $j=1,2$ as
\begin{align}
\gL^{(n+1,n)}_{SL,1}&:=\{
(\alpha, \alpha+m;\triv; \lambda, 
\lambda+m) : \alpha \in \{\pm\}, \lambda \in \C, \; \text{and} \; m\in \Z_{\geq 0}\},
\label{eqn:SL1}\\[3pt]
\gL^{(n+1,n)}_{SL,2}&:=\{
(\alpha, \alpha+m+\ell;\poly_{n-1}^\ell; 1-(m+\ell), 
1+\tfrac{\ell}{n-1}) : \alpha \in \{\pm\} \; \text{and} \; \ell, m\in \Z_{\geq 0}\}.
\label{eqn:SL2}
\end{align}
If $n=2$, then $\gL^{(3,2)}_{SL,j}$ for $j=1,2$ are given by
\begin{align}
\gL^{(3,2)}_{SL,1}&=\{
(\alpha, \alpha+m;\triv;\lambda, \lambda+m)
: \alpha \in \{\pm\}, \lambda \in \C,\text{and} \; m\in \Z_{\geq 0}\},\label{eqn:SL1b}\\[3pt]
\gL^{(3,2)}_{SL,2}
&=\{
(\alpha, \alpha+m+\ell; \poly^\ell_1; 1-(m+\ell), 1+\ell) :
 \alpha \in \{\pm\} \; \text{and} \; \ell, m\in \Z_{\geq 0}\}\label{eqn:SL2b}\\
&=\{
(\alpha, \alpha+m; \triv; 1-(m+\ell), 1+\ell) :
 \alpha \in \{\pm\} \; \text{and} \; \ell, m\in \Z_{\geq 0}\}.\label{eqn:SL2c}
\end{align}

For $n=2$, we further put
\begin{align}\label{eqn:SL3}
\gL^{(3,2)}_{SL,+}
&:=\{
(\alpha, \alpha+m+\ell; \poly^\ell_1; 1-(m+\ell), 1+\ell) : 
\alpha \in \{\pm\}, m\in \Z_{\geq 0}, \; \text{and} \; \ell\in 1+\Z_{\geq 0}\}\\
&=\{
(\alpha, \alpha+m; \triv; 1-(m+\ell), 1+\ell) : 
\alpha \in \{\pm\}, m\in \Z_{\geq 0}, \; \text{and} \; \ell\in 1+\Z_{\geq 0}\}.
\end{align}

We set
\begin{equation}\label{eqn:SL}
\gL^{(n+1,n)}_{SL}:=\gL^{(n+1,n)}_{SL,1} \cup \gL^{(n+1,n)}_{SL,2}.
\end{equation}
As
\begin{equation*}
\gL^{(3,2)}_{SL,+} \subset \gL^{(3,2)}_{SL,2} \subset \gL^{(3,2)}_{SL,1},
\end{equation*}
we have 
\begin{equation*}
\gL^{(3,2)}_{SL}=\gL^{(3,2)}_{SL,1}.
\end{equation*}

We consider the cases $n\geq 3$ and $n=2$, separately.

\begin{thm}
\label{thm:DSBO1a}
Let $n\geq 3$.
The following three conditions on
 $(\alpha,\beta; \varpi; \lambda, \nu) 
\in \{\pm\}^2\times \Irr(SL(n-1,\R))_{\fin} \times \C^2 $ are 
equivalent.
\begin{enumerate}
\item[\emph{(i)}] 
$\Diff_{G'}\big(I(\triv, \lambda)^\alpha, J(\varpi, \nu)^\beta) \neq \{0\}$.
\item[\emph{(ii)}] $\dim \Diff_{G'}\big(I(\triv, \lambda)^\alpha, J(\varpi, \nu)^\beta) =1 $.
\item[\emph{(iii)}] 
$(\alpha,\beta; \varpi; \lambda, \nu) \in \gL^{(n+1,n)}_{SL}$.
\end{enumerate}
\end{thm}

\begin{thm}
\label{thm:DSBO1b}
Let $n=2$.
The following three conditions on
 $(\alpha,\beta;\lambda, \nu) \in \{\pm\}^2\times \C^2 $ are 
equivalent.
\begin{enumerate}
\item[\emph{(i)}] 
$\Diff_{G'}\big(I(\triv, \lambda)^\alpha, J(\triv, \nu)^\beta\big) \neq \{0\}$.
\item[\emph{(ii)}] 
$\dim \Diff_{G'}\big(I(\triv, \lambda)^\alpha, J(\triv, \nu)^\beta\big) \in \{1,2\} $.
\item[\emph{(iii)}] 
$(\alpha,\beta; \triv; \lambda, \nu) \in \gL^{(3,2)}_{SL,1}$.
\end{enumerate}
The dimension is two if and only if 
$(\alpha,\beta; \triv; \lambda, \nu) \in \gL^{(3,2)}_{SL,+}$.
\end{thm}

We next consider the explicit formula of 
$\bD \in \Diff_{G'}(I(\triv, \lambda)^\ga, J(\varpi, \nu)^\gb)$.
We write
\begin{equation*}
\Pol^k(\C^{n-1}) = \C^k[y_1,\ldots, y_{n-1}].
\end{equation*}

In what follows, we identify $\R^{n-1}$ as a subspace of $\R^n$ with
\begin{equation}\label{eqn:N}
\R^{n-1}\simeq \{(x_1, \ldots, x_{n-1}, 0) : x_j \in \R\}.
\end{equation}
Then, 
as in \eqref{eqn:DN},
we understand 
$\bD \in \Diff_{G'}(I(\triv, \lambda)^\ga, J(\varpi,\nu)^\gb)$ 
for $(\alpha,\beta; \varpi; \lambda, \nu) \in \gL^{(n+1,n)}_{SL}$
as a map
\begin{equation*}
\bD\colon C^\infty(\R^{n}) \To 
C^{\infty}(\R^{n-1})\otimes \C^k[y_1,\ldots, y_{n-1}]
\end{equation*}
via the diffeomorphisms
\begin{alignat}{2}\label{eqn:coord}
\R^{n} 
&\stackrel{\sim}{\To} N_-, \quad (x_1, \ldots, x_{n-1}, x_n) 
&&\mapsto \exp(x_1 N^-_1 + \cdots +  x_{n-1} N^-_{n-1}+x_nN^-_n),\\
\R^{n-1} 
&\stackrel{\sim}{\To} N_-', \quad (x_1, \ldots, x_{n-1},0) 
&&\mapsto \exp(x_1 N^-_1 + \cdots +  x_{n-1} N^-_{n-1}).\nonumber
\end{alignat}

For $\ell \in \Z_{\geq 0}$, we put
\begin{equation*}
\Xi_\ell':=\{(\ell_1, \ldots, \ell_{n-1}) \in (\Z_{\geq 0})^{n-1} : \sum_{j=1}^{n-1} 
\ell_j= \ell\}.
\end{equation*}

\noindent
For $\textbf{l} = (\ell_1, \ldots, \ell_{n-1})\in \Xi_\ell'$, we write
\begin{align*}
y_{\mathbf{l}}  &= 
y_1^{\ell_1}\cdots y_{n-1}^{\ell_{n-1}},\\[3pt]
\wy_{\mathbf{l}} &= \frac{1}{\ell_1! \cdots \ell_{n-1}!} 
\cdot
y_{\mathbf{l}},\\[3pt]
\frac{\partial^\ell}{\partial x^{\mathbf{l}}} \
&= \frac{\partial^\ell}{\partial x_1^{\ell_1} \cdots \partial x_{n-1}^{\ell_{n-1}}}.
\end{align*}

Let 
$\Rest_{x_n=0} \colon C^\infty(\R^n) \to C^\infty(\R^{n-1})$ be the restriction operator
from $\R^n$ to $\R^{n-1}$ via the inclusion $\R^{n-1} \hookrightarrow \R^n$ in
\eqref{eqn:N}. Namely, for $f(x',x_n) \in C^\infty(\R^n)$ with 
$x' = (x_1, \ldots, x_{n-1})$, we have
\begin{equation*}
(\Rest_{x_n=0}f)(x') = f(x',0).
\end{equation*}

\noindent
For $m, \ell \in \Z_{\geq 0}$,
we define 
$\bD_{(m,\ell)} \in \Diff_\C(C^\infty(\R^{n}), C^\infty(\R^{n-1})\otimes 
\C^{m+\ell}[y_1, \ldots, y_{n-1}])$ by
\begin{equation}\label{eqn:DSBO}
\bD_{(m,\ell)} :=
\Rest_{x_n=0} \circ
\frac{\partial^m}{\partial x_n^m}\sum_{\mathbf{l} \in \Xi_\ell'} \frac{\partial^\ell}{\partial x^{\mathbf{l}}}
\otimes 
\wy_{\mathbf{l}}.
\end{equation}

\noindent
In particular, we have
\begin{equation*}
\bD_{(m,0)} = \Rest_{x_n=0} \circ \frac{\partial^m}{\partial x_n^m}.
\end{equation*}

\begin{thm}\label{thm:DSBO2a}
Let $n\geq 3$. Then we have
\begin{equation*}
\Diff_{G'}\big(I(\triv, \lambda)^{\alpha}, J(\varpi, \nu)^{\beta}\big)
=
\begin{cases}
\C\bD_{(m,0)} 
& \text{if $(\alpha, \beta;\varpi; \lambda, \nu)\in \Lambda^{(n+1,n)}_{SL,1}$,}\\[3pt]
\C \bD_{(m,\ell)} 
& \text{if $(\alpha, \beta;\varpi; \lambda, \nu)\in \Lambda^{(n+1,n)}_{SL,2}$,}\\
\{0\} & \text{otherwise.}
\end{cases}
\end{equation*}
\end{thm}

\begin{thm}\label{thm:DSBO2b}
Let $n=2$. Then we have
\begin{equation*}
\Diff_{G'}\big(I(\triv, \lambda)^{\alpha}, J(\triv, \nu)^{\beta}\big)
=
\begin{cases}
\C\bD_{(m,0)} 
& \text{if $(\alpha, \beta;\triv \lambda, \nu) 
\in 
\gL^{(3,2)}_{SL,1}\backslash \gL^{(3,2)}_{SL,+}$,}\\[3pt]
\C \bD_{(m+2\ell,0)} \oplus \C \bD_{(m, \ell)}
& \text{if $(\alpha, \beta;\triv; \lambda, \nu)\in 
\gL^{(3,2)}_{SL,+}$,}\\
\{0\} & \text{otherwise.}
\end{cases}
\end{equation*}

\end{thm}

\subsection{$(\fg', P')$-homomorphisms}

We next consider $(\fg', P')$-homomorphisms $\Phi$.
As for Section \ref{sec:DSBO}, 
$(\alpha,\beta; \sigma; s, r) \in 
\{\pm\}^2\times \Irr(SL(n-1,\C))_{\fin} \times \C^2$ indicates
a pair $(\Mpp(\sigma,r)^\gb, \Mp(\triv, s)^\ga\big)$ of generalized Verma modules
of $(\fg, P)$ and $(\fg',P')$, respectively.
If $n=2$, then, via the identity \eqref{eqn:sym0826c}, we regard
$(\ga,\gb;\triv;s,r)$ as 
\begin{equation}\label{eqn:sym0826d}
(\ga,\gb;\triv;s,r) = (\ga,\gb+k;\sym^k_1;s,r)
\end{equation}
for $k \in \Z_{\geq 0}$. 

For $n\geq 2$,
define $\gL^{(n+1,n)}_{(\fg', P'),j} \subset 
\{\pm\}^2\times \Irr(SL(n-1,\R))_{\fin} \times \C^2$ for $j=1,2$ as
\begin{align}
\gL^{(n+1,n)}_{(\fg', P'),1}&:=\{
(\ga,\ga+m;\triv; s,  s-m) : \alpha \in \{\pm\},s\in \C,  \; \text{and} \; m\in \Z_{\geq 0}\},
\label{eqn:gP1} \\[3pt]
\gL^{(n+1,n)}_{(\fg', P'),2}&:=\{
(\ga,\ga+m+\ell;\sym_{n-1}^\ell; (m+\ell)-1,  -(1+\tfrac{\ell}{n-1})) : \alpha \in \{\pm\}\; \text{and} \; \ell, m\in \Z_{\geq 0}\}.\label{eqn:gP}
\end{align}
For $n=2$, the sets $\gL^{(3,2)}_{(\fg', P'),j}$ for $j=1,2$ are given by
\begin{align*}
\gL^{(3,2)}_{(\fg', P'),1}&=\{
(\alpha, \alpha+m;\triv; s, s-m )
: \alpha \in \{\pm\}, s \in \C \; \text{and} \; m\in \Z_{\geq 0}\},\\[3pt]
\gL^{(3,2)}_{(\fg', P'),2}
&=\{
(\alpha, \alpha+m+\ell; \sym^\ell_1; (m+\ell)-1, -(1+\ell) )
: \alpha \in \{\pm\} \; \text{and} \; \ell, m\in \Z_{\geq 0}\}\\
&=\{
(\alpha, \alpha+m; \triv; (m+\ell)-1, -(1+\ell) )
: \alpha \in \{\pm\} \; \text{and} \; \ell, m\in \Z_{\geq 0}\}.
\end{align*}
Further, we put
\begin{align*}
\Lambda^{(3,2)}_{(\fg',P'),+}
&:=\{
(\alpha, \alpha+m+\ell; \sym^\ell_1; (m+\ell)-1, 
-(1+\ell)) : \alpha \in \{\pm\}, m\in \Z_{\geq 0}, \; \text{and} \; \ell\in 1+\Z_{\geq 0}\}\\
&=\{
(\alpha, \alpha+m; \triv; (m+\ell)-1, 
-(1+\ell)) : \alpha \in \{\pm\}, m\in \Z_{\geq 0}, \; \text{and} \; \ell\in 1+\Z_{\geq 0}\}.
\end{align*}
We set 
\begin{equation*}
\gL^{(n+1,n)}_{(\fg',P')}:=\gL^{(n+1,n)}_{(\fg',P'),1} 
\cup \gL^{(n+1,n)}_{(\fg',P'),2}.
\end{equation*}
Since
\begin{equation*}
\Lambda^{(3,2)}_{(\fg',P'),+} \subset 
\gL^{(3,2)}_{(\fg',P'),2} \subset \gL^{(3,2)}_{(\fg',P'),1},
\end{equation*}
we have
\begin{equation*}
\gL^{(3,2)}_{(\fg',P')} = \gL^{(3,2)}_{(\fg',P'),1}.
\end{equation*}

As in Section \ref{sec:DSBO}, we consider the cases $n\geq 3$ and $n =2$,
separately.

\begin{thm}\label{thm:Hom1a}
Let $n\geq 3$. The following three conditions on 
$(\alpha,\beta; \sigma; s, r) \in 
\{\pm\}^2\times \Irr(SL(n-1,\C))_{\fin} \times \C^2$
are equivalent.
\begin{enumerate}
\item[\emph{(i)}] 
$\Hom_{\fg',P'}\big(\Mpp(\sigma,r)^\gb, \Mp(\triv, s)^\ga\big)\neq \{0\}$.
\item[\emph{(ii)}] 
$\dim \Hom_{\fg', P'}\big(\Mpp(\sigma,r)^\gb, \Mp(\triv, s)^\ga\big) =1 $.
\item[\emph{(iii)}] 
$(\alpha,\beta; \sigma; s, r) \in \gL^{(n+1,n)}_{(\fg',P')}$.
\end{enumerate}
\end{thm}

\begin{thm}\label{thm:Hom1b}
Let $n =2$. The following three conditions on 
$(\alpha,\beta; s, r) \in \{\pm\}^2\times  \C^2$
are equivalent.
\begin{enumerate}
\item[\emph{(i)}] 
$\Hom_{\fg',P'}\big(\Mpp(\triv,r)^\gb, \Mp(\triv, s)^\ga\big)\neq \{0\}$.
\item[\emph{(ii)}] 
$\dim \Hom_{\fg', P'}\big(\Mpp(\triv,r)^\gb, \Mp(\triv, s)^\ga\big) \in \{1,2\}$.
\item[\emph{(iii)}] 
$(\alpha,\beta; \triv; s, r) \in \gL^{(3,2)}_{(\fg',P'),1}$.
\end{enumerate}
The dimension is two if and only if 
$(\alpha,\beta; \triv; s, r) \in \gL^{(3,2)}_{(\fg',P'),+}$.
\end{thm}

To give the explicit formulas of 
$\Phi\in \Hom_{\fg',P'}(\Mpp(\sigma,r)^\gb, \Mp(\triv, s)^\ga)$,
we write
\begin{equation*}
S^k(\C^{n-1}) = \C^k[e_1, \ldots, e_{n-1}],
\end{equation*}
where $e_j$ are the standard basis elements of $\C^{n-1}$. 

For $\mathbf{l} = (\ell_1, \ldots, \ell_{n-1})\in \Xi_\ell'$, we write
\begin{alignat*}{2}
e_{\mathbf{l}} &= e_1^{\ell_1} \cdots e_{n-1}^{\ell_{n-1}} &&\in S^\ell(\C^{n-1}),\\
N_{\mathbf{l}}^- &= (N_1^-)^{\ell_1} \cdots (N_{n-1}^-)^{\ell_{n-1}}
&&\in S^\ell(\fn_-).
\end{alignat*}
\noindent
Observe that we have 
\begin{equation*}
\C^\ell[y_1, \ldots, y_{n-1}] 
= \Pol^\ell(\C^{n-1}) = S^\ell((\C^{n-1})^\vee)\simeq S^\ell(\C^{n-1})^\vee.
\end{equation*}
We then define $y_j\in (\C^{n-1})^\vee$ in such a way that  $y_i(e_j) = \delta_{i,j}$,
which gives 
$\widetilde{y}_{\mathbf{l}}(e_{\mathbf{l}'})=\delta_{\mathbf{l},\mathbf{l}'}$
for $\mathbf{l},\mathbf{l}' \in \Xi_\ell'$.

For $m, \ell \in \Z_{\geq 0}$, 
we define $\Phi_{(m,\ell)} \in 
\Hom_\C(S^\ell(\C^{n-1}), S^{m+\ell}(\fn_-))$ by means of
\begin{equation}\label{eqn:Hom}
\Phi_{(m,\ell)}
:= 
(N_n^-)^m
\sum_{\mathbf{l} \in \Xi_\ell'} 
N_{\mathbf{l}}^-
\otimes 
(e_{\mathbf{l}})^\vee
= 
(N_n^-)^m
\sum_{\mathbf{l} \in \Xi_\ell'} 
N_{\mathbf{l}}^-
\otimes 
\widetilde{y}_{\mathbf{l}}.
\end{equation}
In particular, we have 
\begin{equation*}
\Phi_{(m,0)}
= (N_n^-)^m.
\end{equation*}

\noindent
Since $M_\fp(\triv, s)^\ga \simeq S(\fn_-)$ as linear spaces,
we have
\begin{equation*}
\Phi_{(m,\ell)} \in \Hom_{\C}(S^\ell(\C^{n-1}), \Mp(\triv, s)^\ga).
\end{equation*}

\noindent
Further, the following hold.

\begin{thm}\label{thm:Hom2a}
Let $n\geq 3$. We have
\begin{equation*}
\Hom_{\fg',P'}\big(\Mpp(\sigma,r)^\gb, \Mp(\triv, s)^\ga\big)
=
\begin{cases}
\C\Phi_{(m,0)} \quad &\text{if $(\sigma;s,r)\in \gL^{(n+1,n)}_{(\fg',P'),1}$},\\[3pt]
\C\Phi_{(m,\ell)} \quad &\text{if $(\sigma;s,r)\in \gL^{(n+1,n)}_{(\fg',P'),2}$},\\
\{0\} \quad &\text{otherwise}.
\end{cases}
\end{equation*}

\end{thm}

\begin{thm}\label{thm:Hom2b}
For $n = 2$, we have
\begin{equation*}
\Hom_{\fg',P'}\big(\Mpp(\triv,r)^\gb, \Mp(\triv, s)^\ga\big)
=
\begin{cases}
\C\Phi_{(m,0)} \quad 
&\text{if $(\triv;s,r)\in 
\gL^{(3,2)}_{(\fg',P'),1}\backslash  \gL^{(3,2)}_{(\fg',P'),+}$},\\[3pt]
\C\Phi_{(m+2\ell,0)}\oplus \C\Phi_{(m,\ell)} \quad 
&\text{if $(\triv;s,r)\in \gL^{(3,2)}_{(\fg',P'),+}$},\\
\{0\} \quad &\text{otherwise}.
\end{cases}
\end{equation*}
\end{thm}

Here, by abuse of notation, we regard 
$\Phi_{(m,\ell)}$ as a map
\begin{equation*}
\Phi_{(m,\ell)} \in 
\Hom_{\fg',P'}\big(\Mpp(\sigma,r)^\gb, \Mp(\triv, s)^\ga\big)
\end{equation*}
defined by
\begin{equation}\label{eqn:Phi-def}
\Phi_{(m,\ell)}(u\otimes w) := u\Phi_{(m,\ell)}(w)
\quad 
\text{for $u \in \Cal{U}(\fg)$ and $w \in S^\ell(\C^{n-1})$.}
\end{equation}

\subsection{Classification of $\fg'$-homomorphisms}
\label{subsec:g-hom}

Finally, we consider $\fg'$-homomorphisms.
For $n\geq 2$,  we define $\gL^{(n+1,n)}_{\fg',j} \subset 
\Irr(\f{sl}(n-1,\C))_{\fin} \times \C^2$ for  $j=1,2$ such that
\begin{align*}
\gL^{(n+1,n)}_{\fg',1}&:=\{
(\triv; s,  s-m) :  s\in \C\; \text{and} \; m\in \Z_{\geq 0}\},\\
\gL^{(n+1,n)}_{\fg',2}&:=\{
(\sym_{n-1}^\ell; (m+\ell)-1,  -(1+\tfrac{\ell}{n-1})) : \ell, m\in \Z_{\geq 0}\}.
\end{align*}
For $n=2$, we further put
\begin{equation*}
\Lambda^{(3,2)}_{\fg',+}
:=\{(\triv; (m+\ell)-1,  -(1+\ell)) : m\in \Z_{\geq 0}\; \text{and}\; 
\ell \in 1+\Z_{\geq 0}\}.
\end{equation*}
We have
\begin{equation*}
\Lambda^{(3,2)}_{\fg',+}  \subset \Lambda^{(3,2)}_{\fg',2} \subset \gL^{(3,2)}_{\fg',1}.
\end{equation*}

As in \eqref{eqn:Verma},
for $(\eta, s) \in \Irr(\f{sl}(n,\C))_\fin\times \C$ and 
$(\sigma,r) \in \Irr(\f{sl}(n-1,\C))_\fin\times \C$,
we define generalized Verma modules $\Mp(\eta,s)$ and  $\Mp(\sigma,r)$
as a $\fg$-module and $\fg'$-module, respectively.

\begin{thm}\label{thm:Hom3a}
For $n\geq 3$, we have
\begin{equation*}
\Hom_{\fg'}\big(\Mpp(\sigma,r), \Mp(\triv, s)\big)
=
\begin{cases}
\C\Phi_{(m,0)} \quad &\text{if $(\sigma;s,r)\in \gL^{(n+1,n)}_{\fg',1}$},\\[3pt]
\C\Phi_{(m,\ell)} \quad &\text{if $(\sigma;s,r)\in \gL^{(n+1,n)}_{\fg',2}$},\\
\{0\} \quad &\text{otherwise}.
\end{cases}
\end{equation*}
\end{thm}

\begin{thm}\label{thm:Hom3b}
For $n = 2$, we have
\begin{equation*}
\Hom_{\fg'}\big(\Mpp(\triv,r), \Mp(\triv, s)\big)
=
\begin{cases}
\C\Phi_{(m,0)} \quad 
&\text{if $(\triv;s,r)\in \gL^{(3,2)}_{\fg',1}\backslash  \gL^{(3,2)}_{\fg',+}$},\\[3pt]
\C\Phi_{(m+2\ell,0)}\oplus \C\Phi_{(m,\ell)} \quad &\text{if $(\triv;s,r)\in \gL^{(3,2)}_{\fg',+}$},\\
\{0\} \quad &\text{otherwise}.
\end{cases}
\end{equation*}
\end{thm}

\begin{rem}
Let $n\geq 2$.
Since the generalized Verma module $\Mpp(\triv,s-m)$ is of scalar type, 
the $\fg'$-homomorphisms
$\varphi_{(m,0)} \colon \Mpp(\triv,s-m) \to \Mp(\triv, s)$ are all injective
(cf.\ \cite[Prop.\ 9.11]{Hum08}). Theorems \ref{thm:Hom3a} and \ref{thm:Hom3b} then 
imply that we have
\begin{equation}\label{eqn:GVM04}
\bigoplus_{m=0}^\infty \Mpp(\triv,s-m)
\hookrightarrow
\Mp(\triv, s).
\end{equation}
In Section \ref{sec:GVM}, we shall show that
$\Mp(\triv, s)\vert_{\fg'}$
is in fact isomorphic to 
$\bigoplus_{m=0}^\infty \Mpp(\triv,s-m)$
for all $s \in \C$ as $\fg'$-modules
 (see Theorem \ref{thm:GVM31a}). Further, we shall also discuss the 
 multiplicity-two phenomenon for $n=2$ from 
 a viewpoint of branching laws (see Remark \ref{rem:branching916}).
\end{rem}

\section{Proofs for the classification and construction of $\bD$ and $\Phi$:
Case $n\geq 3$}
\label{sec:proof1}

In the present and next sections, 
we follow the recipe of the F-method in Section \ref{sec:recipe}
to prove the theorems in Section \ref{sec:results}.
Since the arguments are slightly different between
the cases $n\geq 3$ and $n=2$, we consider the two cases, separately.

In this section we deal with the case $n\geq 3$, that is, we show 
Theorems \ref{thm:DSBO1a}, \ref{thm:DSBO2a}, \ref{thm:Hom1a}, 
 \ref{thm:Hom2a}, and \ref{thm:Hom3a}. The case $n=2$ is considered separately
in Section \ref{sec:proof2}.

\subsection{Step 1: 
Compute $d\pi_{(\xi,\lambda)^*}(C)$ and 
$\widehat{d\pi_{(\xi,\lambda)^*}}(C)$
for $C \in \fn_+'$.}
\label{sec:Step1}

For $\xi = \alpha \otimes \triv$ and $\lambda \equiv \chi^\lambda$, 
we simply write
\begin{equation*}
\dpi_{\lambda^*} = \dpi_{(\alpha\otimes\triv,\chi^\lambda)^*}
\end{equation*}
with $\lambda^*=2\rho(\fn_+)-\lambda d\chi$. 

The operators
$d\pi_{(\xi,\lambda)^*}(N_j^+)$ and 
$\widehat{d\pi_{(\xi,\lambda)^*}}(N_j^+)$
are already computed in \cite[Sect.\ 5]{KuOr24} for $N_j^+ \in \fn_+$.
We thus simply recall those formulas from the cited paper in this subsection.
We remark that the formulas are not only for 
$N_j^+ \in \fn_+'$ but also for $N_j^+ \in \fn_+$ (full nilpotent radical).

We write
\begin{equation*}
\Pol(\fn_+) = \C[\zeta_1, \ldots, \zeta_{n-1}, \zeta_n],
\end{equation*}
where $(\zeta_1, \ldots, \zeta_{n-1}, \zeta_n)$ is 
the dual coordinates to
$(z_1, \ldots, z_{n-1}, z_n)$ of $\fn_-$, which
corresponds to the coordinates 
$(x_1, \ldots, x_{n-1}, x_n)$ of $\fn_-(\R)$ in \eqref{eqn:coord}.
Let 
$E_x = \sum_{j=1}^{n} x_j\frac{\partial}{\partial x_j}$ 
and
$E_\zeta = \sum_{j=1}^{n} \zeta_j\frac{\partial}{\partial \zeta_j}$
denote
the Euler homogeneity operator for $x$ and $\zeta$, respectively.
For $j \in \{1,\ldots, n-1,n\}$,
we write $\vartheta_j = \zeta_j \frac{\partial}{\partial \zeta_j}$ 
for the Euler operator for $\zeta_j$ such that 
$E_\zeta = \sum_{j=1}^{n} \vartheta_j$.

\begin{prop}[{\cite[Props.\ 5.1 and 5.4]{KuOr24}}]\label{prop:dNj1}
For $j \in \{1, \ldots, n-1, n\}$, we have 
\begin{align}
d\pi_{\lambda^*}(N_j^+)
&=x_j\{ (n-\lambda) 
+E_x\},\nonumber \\[3pt] 
-\zeta_j
\widehat{\dpi_{\lambda^*}}(N_j^+)
&=\vartheta_j(\lambda -1 + E_\zeta).
\end{align} 
\end{prop}

Observe that $E_\zeta\vert_{\Pol^k(\fn_+)}$ is simply given by
$E_\zeta\vert_{\Pol^k(\fn_+)} = k \cdot \id$. 
Then, for $j \in \{1, \ldots, n-1\}$, 
 we have
\begin{equation}\label{eqn:dNj3}
-\zeta_j\widehat{\dpi_{\lambda^*}}(N_j^+)\vert_{\Pol^k(\fn_+)}
=(\lambda-1+k)\vartheta_j.
\end{equation}

\subsection{Step 2: 
Classify and construct 
$\psi \in \Hom_{M'A'}(W^\vee, \Pol(\fn_+)\otimes V^\vee)$.
}

For $(\varpi, W) \in \Irr(M_0')_{\fin}$, we write
\begin{equation*}
W_{\beta} = \C_{\beta} \otimes W
\end{equation*}
for 
the representation
space of $(\beta, \varpi) \in \Irr(M')_{\fin}$, where 
$\C_\gb$ denotes the one-dimensional representation 
$\C_\gb = (\sgn^\beta, \C)$ of $M'$ defined as in 
\eqref{eqn:20241106}.
Similarly, for $\ga \in \{\pm\}$, we define 
the $M$-representation $\Pol(\fn_+)_\ga$ by
\begin{equation}\label{eqn:Pol}
\Pol(\fn_+)_\ga = \C_\ga \otimes \Pol(\fn_+).
\end{equation}
In this step, we wish to classify and construct
\begin{equation*}
\psi \in 
\Hom_{M'A'}(W_{\beta}^\vee\boxtimes \C_{-\nu}, 
\Pol(\fn_+)_{\alpha}
\otimes \C_{-\lambda}).
\end{equation*}

We start by observing the $M'A'$-decomposition
$\Pol(\fn_+)\vert_{M'A'} = \C[\zeta_1,\ldots, \zeta_{n-1}, \zeta_n]\vert_{M'A'}$.
To do so, the following lemma is useful. 

\begin{lem}\label{lem:MA}
The following hold.

\begin{enumerate}


\item[\emph{(1)}]
$(M', \Ad_\#, \C^m[\zeta_n]) 
\hspace{42pt}
\simeq (SL^\pm(n-1,\R), \sgn^m \otimes \triv, \C)$.

\item[\emph{(2)}]
$(M', \Ad_\#, \C^\ell[\zeta_1,\ldots, \zeta_{n-1}]) 
\hspace{2pt}
\simeq (SL^\pm(n-1,\R), \sgn^\ell \otimes \sym^\ell_{n-1}, S^\ell(\C^{n-1}))$.

\item[\emph{(3)}]
$A'$ acts on $\C^m[\zeta_n]$ by a character $(\chi')^{-m}$.

\item[\emph{(4)}]
$A'$ acts on $\C^\ell[\zeta_1,\ldots, \zeta_{n-1}]$ by a character 
$(\chi')^{-\frac{n}{n-1}\ell}$.

\end{enumerate}

Here $\chi'$ is the character of $A'$ defined in \eqref{eqn:chi}.

\end{lem}

\begin{proof}
A direct computation.
\end{proof}

\begin{rem}\label{rem:sym0826a}
If $n=2$, then, by \eqref{eqn:sym0826a}, we have 
\begin{align*}
(M', \Ad_\#, \C^\ell[\zeta_1]) 
&\simeq (\Z/2\Z, \sgn^\ell \otimes \sym^\ell_{1}, S^\ell(\C))\\
&=(\Z/2\Z, \triv \otimes \triv, \C).
\end{align*}
\end{rem}

It follows from Lemma \ref{lem:MA} that 
$\C[\zeta_1, \ldots, \zeta_{n-1},\zeta_n]\vert_{M'A'}$ decomposes irreducibly as
\begin{align}\label{eqn:MA21}
\C[\zeta_1, \ldots, \zeta_{n-1},\zeta_n]\vert_{M'A'}
&=
\bigoplus_{m, \ell \in \Z_{\geq 0}} 
\C^m[\zeta_n] 
\C^\ell[\zeta_1,\ldots,\zeta_{n-1}]\nonumber\\
&\simeq
\bigoplus_{m, \ell \in \Z_{\geq 0}} 
(\sgn^{m+\ell}\otimes\sym^\ell_{n-1}) \boxtimes (-(m+\tfrac{n}{n-1}\ell)),
\end{align}
where  $-(m+\tfrac{n}{n-1}\ell)$ indicates the weight of the character of $A'$.
Therefore, 
we have
\begin{equation*}
\big(\Pol(\fn_+)_{\alpha}
\otimes \C_{-\lambda}\big)\vert_{M'A'}
\simeq 
\bigoplus_{m, \ell \in \Z_{\geq 0}} 
(\sgn^{\ga+m+\ell}\otimes\sym^\ell_{n-1}) \boxtimes (-(\lambda+m+\tfrac{n}{n-1}\ell)).
\end{equation*}
We remark that the $M'A'$-representations appeared in \eqref{eqn:MA21} are all inequivalent.

Put 
\begin{equation}\label{eqn:paraMA}
\gL^{(n+1,n)}_{M'A'}:=
\{(\ga, \ga+m+\ell; \poly^\ell_{n-1}; \lambda, \lambda + m+\tfrac{n}{n-1}\ell):
\lambda \in \C\; \text{and}\; m, \ell \in \Z_{\geq 0}\}.
\end{equation}

\begin{prop}\label{prop:MA}
The following conditions on 
$(\ga,\gb;\varpi;\lambda,\nu) \in 
\{\pm\} \times \Irr(SL(n-1,\R))_{\fin} \times \C^2$ are equivalent.
\begin{enumerate}
\item[\emph{(i)}]
$\Hom_{M'A'}(W^\vee_\beta \otimes \C_{-\nu}, \Pol(\fn_+)_\ga \otimes \C_{-\lambda})
\neq \{0\}$.
\item[\emph{(ii)}]
$\dim
\Hom_{M'A'}(W^\vee_\beta \otimes \C_{-\nu}, \Pol(\fn_+)_\ga \otimes \C_{-\lambda})=1$.
\item[\emph{(iii)}]
$(\ga,\gb;\varpi;\lambda,\nu) \in \gL^{(n+1,n)}_{M'A'}$.
\end{enumerate}
\end{prop}

\begin{proof}
Since $(\sym_{n-1}^k, S^k(\C^{n-1}))^\vee \simeq (\poly_{n-1}^k,\Pol^k(\C^{n-1}))$,
the assertions simply follow from the preceding arguments.
\end{proof}

Now, for $\mathbf{l}=(\ell_1, \ldots, \ell_{n-1}) \in \Xi_\ell'$,  we write
\begin{equation*}
\zeta_{\mathbf{l}}= \zeta_1^{\ell_1} \cdots \zeta_{n-1}^{\ell_{n-1}}
\in \C^\ell[\zeta_1,\ldots, \zeta_{n-1}].
\end{equation*}
We then define 
$\psi_{(m,\ell)} \in 
\Hom_\C(S^\ell(\C^{n-1}),\C^m[\zeta_n] \C^\ell[\zeta_1,\ldots, \zeta_{n-1}])$ by
\begin{equation}\label{eqn:psi}
\psi_{(m,\ell)} = 
\zeta_n^m\sum_{\mathbf{l}\in \Xi_\ell'}
\zeta_{\mathbf{l}} \otimes \widetilde{y}_{\mathbf{l}},
\end{equation}
where $\widetilde{y}_{\mathbf{l}}\in \Pol^k(\C^{n-1})\simeq S^k(\C^{n-1})^\vee$ are regarded as 
the dual basis of $e_{\mathbf{l}} \in S^k(\C^{n-1})$.

\begin{prop}\label{prop:sym}
We have
\begin{equation*}
\Hom_{M'A'}(W_{\beta}^\vee\boxtimes \C_{-\nu}, 
\Pol(\fn_+)_{\alpha}\otimes \C_{-\lambda})=
\begin{cases}
\C \psi_{(m,\ell)} & \text{if 
$(\ga,\gb;\varpi;\lambda,\nu) \in \gL^{(n+1,n)}_{M'A'}$,}\\
\{0\} & \text{otherwise.}
\end{cases}
\end{equation*}
\end{prop}

\begin{proof}
As $\psi_{(m,\ell)}$ maps 
$\psi_{(m,\ell)}\colon e_{\mathbf{l}} \mapsto \zeta_n^m \zeta_{\mathbf{l}}$,
it satisfies the desired $M'A'$-equivariance if 
$(\ga,\gb;\varpi;\lambda,\nu) \in \gL^{(n+1,n)}_{M'A'}$.
Thus, we have
$\psi_{(m,\ell)} \in \Hom_{M'A'}(W_{\beta}^\vee\boxtimes \C_{-\nu}, 
\Pol(\fn_+)_{\alpha}\otimes \C_{-\lambda})$.
Now the multiplicity-one property from Proposition \ref{prop:MA} concludes
the assertion.
\end{proof}

\subsection{Step 3: 
Solve the F-system
for $\psi \in \Hom_{M'A'}(W^\vee, \Pol(\fn_+)\otimes V^\vee)$.
}

For $(\ga, \gb; \varpi; \lambda,\nu) \in \{\pm\}^2\times \Irr(SL(n-1,\R))_\fin \times 
\C^2$, we put
\begin{align*}
&\Sol(\fn_+; \triv_{\alpha, \lambda}, \varpi_{\beta, \nu})\\[3pt]
&:=
\{ \psi \in 
\Hom_{M'A'}(W_{\beta}^\vee\boxtimes \C_{-\nu}, 
\Pol(\fn_+)_{\alpha}\otimes \C_{-\lambda}): 
\text{
$\psi$ solves the F-system \eqref{eqn:Fsys21} below.}\}.
\end{align*}
\begin{equation}\label{eqn:Fsys21}
(\widehat{\dpi_{\lambda^*}}(N_j^+)\otimes \id_W)\psi =0 
\quad
\text{for all $j \in \{1,\ldots, n-1\}$}.
\end{equation}

Since
\begin{equation*}
\Sol(\fn_+; \triv_{\alpha, \lambda}, \varpi_{\beta, \nu})
\subset
\Hom_{M'A'}(W_{\beta}^\vee\boxtimes \C_{-\nu}, 
\Pol(\fn_+)_{\alpha}\otimes \C_{-\lambda}),
\end{equation*}
it follows  from
Proposition \ref{prop:MA}  that
if $\Sol(\fn_+; \triv_{\alpha, \lambda}, \varpi_{\beta, \nu})\neq \{0\}$,
then $(\alpha, \beta; \varpi; \lambda, \nu)$ is of the form
\begin{equation}\label{eqn:param}
(\alpha, \beta; \varpi; \lambda, \nu)=
(\ga, \ga+m+\ell; \poly^\ell_{n-1}; \lambda, \lambda + m+\tfrac{n}{n-1}\ell)
\end{equation}
for some $m, \ell \in \Z_{\geq 0}$. Further, by Proposition \ref{prop:sym}, 
it suffices to solve the PDE 
\begin{equation*}
(\widehat{\dpi_{\lambda^*}}(N_j^+)\otimes \id_W)\psi_{(m,\ell)} =0 
\quad
\text{for all $j \in \{1,\ldots, n-1\}$},
\end{equation*}
which is equivalent to solving
\begin{equation}\label{eqn:Fsys2}
(-\zeta_j\widehat{\dpi_{\lambda^*}}(N_j^+)\otimes \id_W)\psi_{(m,\ell)} =0 
\quad
\text{for all $j \in \{1,\ldots, n-1\}$}.
\end{equation}

Recall from \eqref{eqn:SL1} and \eqref{eqn:SL2} that we have 
\begin{align*}
\gL^{(n+1,n)}_{SL,1}&=\{
(\alpha, \alpha+m;\triv; \lambda, 
\lambda+m) : \alpha \in \{\pm\}, \lambda \in \C, \; \text{and} \; m\in \Z_{\geq 0}\},\\[3pt]
\gL^{(n+1,n)}_{SL,2}&=\{
(\alpha, \alpha+m+\ell;\poly_{n-1}^\ell; 1-(m+\ell), 
1+\tfrac{\ell}{n-1}) : \alpha \in \{\pm\} \; \text{and} \; \ell, m\in \Z_{\geq 0}\}.
\end{align*}

\begin{thm}\label{thm:Sol1}
Let $n\geq 3$. We have 
\begin{equation*}
\Sol(\fn_+; \triv_{\alpha, \lambda}, \varpi_{\beta, \nu})
=
\begin{cases}
\C\psi_{(m,0)} 
& \text{if $(\alpha, \beta;\varpi; \lambda, \nu)\in \Lambda^{(n+1,n)}_{SL,1}$,}\\[3pt]
\C \psi_{(m,\ell)} 
& \text{if $(\alpha, \beta;\varpi; \lambda, \nu)\in \Lambda^{(n+1,n)}_{SL,2}$,}\\
\{0\} & \text{otherwise.}
\end{cases}
\end{equation*}

\end{thm}

\begin{proof}
We wish to solve \eqref{eqn:Fsys2}.
Recall from \eqref{eqn:dNj3} that
$-\zeta_j\widehat{\dpi_{\lambda^*}}(N_j^+)\vert_{\Pol^k(\fn_+)}$ is given by
\begin{equation*}
-\zeta_j\widehat{\dpi_{\lambda^*}}(N_j^+)\vert_{\Pol^k(\fn_+)}
=(\lambda-1+k)\vartheta_j.
\end{equation*}
Since 
\begin{equation*}
\psi_{(m,\ell)} =
\sum_{\mathbf{l}\in \Xi_\ell'}
\zeta_n^m\zeta_{\mathbf{l}} \otimes \widetilde{y}_{\mathbf{l}}
\in \Pol^{m+\ell}(\fn_+) \otimes \Pol^\ell(\C^{n-1}),
\end{equation*}
the left-hand side of \eqref{eqn:Fsys2} amounts to
\begin{align*}
(-\zeta_j\widehat{\dpi_{\lambda^*}}(N_j^+)\otimes \id_W)\psi_{(m,\ell)} 
&=
\sum_{\mathbf{l}\in \Xi_\ell'}
-\zeta_j\widehat{\dpi_{\lambda^*}}(N_j^+)(\zeta_n^m\zeta_{\mathbf{l}})
\otimes \widetilde{y}_{\mathbf{l}}\\
&=
\sum_{\mathbf{l}\in \Xi_\ell'}
(\lambda-1+m+\ell)\vartheta_j(\zeta_n^m\zeta_{\mathbf{l}})
\otimes \widetilde{y}_{\mathbf{l}}\\
&=
\zeta_n^m 
\sum_{\mathbf{l}\in \Xi_\ell'}
(\lambda-1+m+\ell)\vartheta_j(\zeta_{\mathbf{l}})
\otimes \widetilde{y}_{\mathbf{l}}.
\end{align*}
Thus, one wishes to solve
\begin{equation}\label{eqn:theta}
(\lambda-1+m+\ell)\vartheta_j(\zeta_{\mathbf{l}}) =0
\quad
\text{for all $j \in \{1,\ldots, n-1\}$ and $\mathbf{l} \in \Xi_\ell'$}.
\end{equation}
Since
$(\lambda-1+m+\ell)\vartheta_j(\zeta_{\mathbf{l}}) 
=(\lambda-1+m+\ell)\ell_j\zeta_{\mathbf{l}}$,
Equation
\eqref{eqn:theta} holds if and only if $\ell=0$ or $\lambda = 1-(m+\ell)$.
Now the theorem follows from \eqref{eqn:param}
with the observation that
$\gL^{(n+1,n)}_{SL,1}=\gL^{(n+1,n)}_{M'A'}$ for $\ell =0$ and 
$\gL^{(n+1,n)}_{SL,2}=\gL^{(n+1,n)}_{M'A'}$ for $\lambda = 1-(m+\ell)$.

\end{proof}

\subsection{Step 4: 
Apply $\Rest_{x_n=0}\circ \symb^{-1}$ 
and
$F_c^{-1} \otimes \id_W$ 
to the solution 
$\psi \in \mathrm{Sol}(\mathfrak n_+;V,W)$.
}
\label{sec:Step4}

Observe that 
$\bD_{(m,\ell)}$ in \eqref{eqn:DSBO} and $\Phi_{(m,\ell)}$ 
in \eqref{eqn:Hom} are 
given by
\begin{align*}
\bD_{(m,\ell)}
&=
\Rest_{x_n=0} \circ
\frac{\partial^m}{\partial x_n^m}\sum_{\mathbf{l} \in \Xi_\ell'} \frac{\partial^\ell}{\partial x^{\mathbf{l}}}
\otimes 
\wy_{\mathbf{l}}\\
&=
\Rest_{x_n=0} \circ
\sum_{\mathbf{l} \in \Xi_\ell'} 
\symb^{-1}(\zeta_n^m\zeta_{\mathbf{l}})
\otimes 
\wy_{\mathbf{l}}\\
&= \Rest_{x_n=0}\circ \symb^{-1}(\psi_{(m,\ell)})
\end{align*}
and
\begin{equation*}
\Phi_{(m,\ell)}
= 
(N_n^-)^m
\sum_{\mathbf{l} \in \Xi_\ell'} 
N_{\mathbf{l}}^-
\otimes 
\widetilde{y}_{\mathbf{l}}
=\sum_{\mathbf{l} \in \Xi_\ell'} 
F_c^{-1}(\zeta_n^m\zeta_\mathbf{l})
\otimes 
\widetilde{y}_{\mathbf{l}}
=(F_c^{-1} \otimes \id_{W})(\psi_{(m,\ell)}).
\end{equation*}
\noindent
Now we are ready to prove Theorems \ref{thm:DSBO1a}, \ref{thm:DSBO2a}, \ref{thm:Hom1a},
and \ref{thm:Hom2a}.

\begin{proof}[Proof of 
Theorems \ref{thm:DSBO1a}, \ref{thm:DSBO2a}, \ref{thm:Hom1a},
and \ref{thm:Hom2a}]
By Theorem \ref{thm:abelian}, we have
\begin{align*}
\Diff_{G'}(I(\triv, \lambda)^\ga, I(\varpi,  \nu)^\beta)
&=
(\Rest_{x_n=0}\circ 
\symb^{-1})(\Sol(\fn_+; \triv_{\alpha, \lambda}, \varpi_{\beta, \nu})),\\
\Hom_{\fg', P'}(\Mpp(\varpi^\vee,-\nu)^\gb, \Mp(\triv, -\lambda)^\ga)
&=
(F_c^{-1}\otimes \id_{W})(\Sol(\fn_+; \triv_{\alpha, \lambda}, \varpi_{\beta, \nu})).
\end{align*}
Since $(\poly_{n-1}^\ell)^\vee \simeq \sym_{n-1}^\ell$, 
the proposed assertions follow from Theorem \ref{thm:Sol1}.
\end{proof}

We end this section by discussing the proof of Theorem \ref{thm:Hom3a}.

\begin{proof}[Proof of Theorem \ref{thm:Hom3a}]

Let $P_0' = M_0'A'N_+'$ be a Langlands decomposition of 
 the identity component of the parabolic subgroup $P'=M'A'N_+'$. 
For $(\varpi; \lambda, \nu) \in \Irr(SL(n-1,\R))_\fin\times \C^2$, 
we let
\begin{equation*}
\Sol(\fn_+; \triv_{\lambda}, \varpi_{\nu})_0
=
\{ \psi \in 
\Hom_{M_0'A'}(W^\vee\boxtimes \C_{-\nu}, 
\Pol(\fn_+)\otimes \C_{-\lambda}): 
\text{
$\psi$ solves the F-system \eqref{eqn:Fsys2}.}\}.
\end{equation*}
As $P_0'$ is connected, we have 
 \begin{equation*}
\Hom_{\fg', P_0'}\big(\Mpp(\sigma,r), \Mp(\triv, s)\big)
=
\Hom_{\fg'}\big(\Mpp(\sigma,r), \Mp(\triv, s)\big),
\end{equation*}
which yields  a linear isomorphism
\begin{equation}\label{eqn:HD0}
F_c \otimes \id_W \colon
\Hom_{\fg'}(\Mpp(\varpi^\vee, -\nu), \Mp(\triv,-\lambda))
\stackrel{\sim}{\To}
\Sol(\fn_+; \triv_{\lambda}, \varpi_{\nu})_0.
\end{equation}
Now Theorem \ref{thm:Hom3a} follows from the same arguments
in Sections \ref{sec:Step1}--\ref{sec:Step4}.
\end{proof}

For later convenience, we state
the classification of $\Sol(\fn_+; \triv_{\lambda}, \varpi_{\nu})_0$.
Put
\begin{align*}
\accentset{\circ}{\gL}^{(n+1,n)}_{SL,1}&:=\{
(\triv;\lambda, \lambda+m)
:  \lambda \in \C \; \text{and} \; m\in \Z_{\geq 0}\},\\
\accentset{\circ}{\gL}^{(n+1,n)}_{SL,2}&:=\{
(\poly^\ell_{n-1};1-(m+\ell), 1+\tfrac{\ell}{n-1}) :  m,  \ell\in \Z_{\geq 0}\}.
\end{align*}

\begin{prop}\label{prop:Sol30a}
We have 
\begin{equation*}
\Sol(\fn_+; \triv_{\lambda}, \varpi_{\nu})_0
=
\begin{cases}
\C\psi_{(m,0)} 
& \text{if $(\varpi; \lambda, \nu)\in 
\accentset{\circ}{\gL}^{(n+1,n)}_{SL,1}$,}\\[3pt]
\C \psi_{(m,\ell)} 
& \text{if $(\varpi; \lambda, \nu)\in
\accentset{\circ}{\gL}^{(n+1,n)}_{SL,2}$,}\\
\{0\} & \text{otherwise.}
\end{cases}
\end{equation*}
\end{prop}

\begin{proof}
Since
the arguments are identical to Theorem \ref{thm:Sol1},
we omit the proof.
\end{proof}

\section{Proofs for the classification and construction of $\bD$ and $\Phi$:
Case $n=2$}
\label{sec:proof2}

The aim of this section is to prove the theorems in Section \ref{sec:results} 
for the case $n=2$, namely,
Theorems \ref{thm:DSBO1b}, \ref{thm:DSBO2b}, \ref{thm:Hom1b},
\ref{thm:Hom2b}, and \ref{thm:Hom3b}. 
Throughout this section we assume $n=2$; in particular, we have 
$(G,G') = (SL(3,\R), SL(2,\R))$ and
$\Pol(\fn_+) = \C[\zeta_1,\zeta_2]$.

Observe that since $G'=SL(2,\R)$, the $M'$-part of 
the parabolic subgroup $P'=M'A'N_+'$ of $G'$ is given by $M'\simeq \Z/2\Z$;
thus,
\begin{equation*}
\Irr(P')_{\fin}\simeq 
\{\pm\} \times \{\triv\} \times \C.
\end{equation*}
\noindent
So, the space of solutions to the F-system in concern is 
\begin{align*}
\Sol(\fn_+; \triv_{\alpha, \lambda}, \triv_{\beta, \nu})
=
\{ \psi \in 
\Hom_{M'A'}(\C_{\beta}\boxtimes \C_{-\nu}, 
\Pol(\fn_+)_{\alpha}\otimes \C_{-\lambda}): 
\text{
$\psi$ solves \eqref{eqn:Fsys3} below.}\}.
\end{align*}
\begin{equation}\label{eqn:Fsys3}
\widehat{\dpi_{\lambda^*}}(N_1^+)\psi =0.
\end{equation}

In the present case, as opposed to the case $n\geq 3$, 
the space $\Hom_{M'A'}(\C_{\beta}\boxtimes \C_{-\nu}, 
\Pol(\fn_+)_{\alpha}\otimes \C_{-\lambda})$ 
could have higher multiplicity. 
Thus,
to simplify the exposition,
we first follow the recipe of the F-method  for 
$P_0' = M_0'A'N_+'=A'N_+'$. We shall consider the parity condition coming from
$M'$ in the end. 

As in the previous section, we put
\begin{equation*}
\Sol(\fn_+; \triv_{\lambda}, \triv_{\nu})_0
:=
\{ \psi \in 
\Hom_{A'}(\C_{-\nu}, 
\Pol(\fn_+)\otimes \C_{-\lambda}): 
\text{
$\psi$ solves \eqref{eqn:Fsys3}.}\}.
\end{equation*}
Then we shall proceed with the following steps.

\begin{enumerate}

\item[Step 1]
Classify and construct 
$\psi \in \Hom_{A'}(\C_{-\nu}, \Pol(\fn_+)\otimes \C_{-\lambda})$.
\vskip 0.1in

\item[Step 2]
Solve  \eqref{eqn:Fsys3}
for $\psi \in \Hom_{A'}(\C_{-\nu}, \Pol(\fn_+)\otimes \C_{-\lambda})$.
\vskip 0.1in

\item[Step 3]
Consider the parity condition on
 $\psi \in \Sol(\fn_+; \triv_{\lambda}, \triv_{\nu})_0$.
\vskip 0.1in

\item[Step 4]
Apply $\Rest_{x_2=0}\circ \symb^{-1}$ 
and
$F_c^{-1}$ 
to the solution $\psi \in
\Sol(\fn_+; \triv_{\alpha, \lambda}, \triv_{\beta, \nu})$.

\vskip 0.1in

\end{enumerate}

As the identity component $P_0'$ is considered, 
our arguments naturally include the proof of 
Theorem \ref{thm:Hom3b} as in the end of the previous section.

\subsection{Step 1: 
Classify and construct 
$\psi \in \Hom_{A'}(\C_{-\nu}, \Pol(\fn_+)\otimes \C_{-\lambda})$.
}
It follows from Lemma \ref{lem:MA} that we have
\begin{equation}\label{eqn:MA3a}
\C^m[\zeta_2] \C^\ell[\zeta_1] 
\simeq
-(m+2\ell),
\end{equation}
where  $-(m+2\ell)$ indicates the weight of the character of $A'$.

\begin{prop}\label{prop:MA2}
The following conditions on 
$(\lambda, \nu) \in \C^2$ are equivalent.
\begin{enumerate}
\item[\emph{(i)}]
$\Hom_{A'}(\C_{-\nu}, \Pol(\fn_+)\otimes \C_{-\lambda})\neq \{0\}$.
\item[\emph{(ii)}]
$\nu-\lambda \in \Z_{\geq 0}$.
\end{enumerate}
\end{prop}

\begin{proof}
By \eqref{eqn:MA3a}, 
the decomposition
$\C[\zeta_1, \zeta_2]\vert_{A'}$ is given as
\begin{align*}
\C[\zeta_1,\zeta_2]\vert_{A'}
=
\bigoplus_{k \in \Z_{\geq 0}}
\bigoplus_{m+2\ell =k} 
\C^m[\zeta_2] \C^\ell[\zeta_1]
\simeq
\bigoplus_{k \in \Z_{\geq 0}}
\bigoplus_{m+2\ell =k} (-k).
\end{align*}
Therefore,
\begin{equation*}
\big(\Pol(\fn_+)
\otimes \C_{-\lambda}\big)\vert_{A'}
\simeq 
\bigoplus_{k \in \Z_{\geq 0}}
\bigoplus_{m+2\ell =k} -(\lambda+k),
\end{equation*}
which shows the proposed assertion.
\end{proof}

As in \eqref{eqn:psi}, 
we write
\begin{equation*}
\psi_{(m, \ell)} = \zeta_2^m\zeta_1^\ell
\in \C^{m+\ell}[\zeta_1,\zeta_2].
\end{equation*}
Then,
\begin{align}\label{eqn:MA4}
\Hom_{A'}(\C_{-(\lambda+k)}, \Pol(\fn_+)\otimes \C_{-\lambda})
&=\spn_{\C}
\{
\psi_{(m,\ell)}: m+2\ell = k
\}\\
&=\spn_{\C}
\{
\psi_{(k-2\ell,\ell)}: \ell  = 0,\ldots, [\tfrac{k}{2}]
\}.
\end{align}
In particular, we have 
\begin{equation*}
\dim \Hom_{A'}(\C_{-(\lambda+k)}, \Pol(\fn_+)\otimes \C_{-\lambda})
=[\tfrac{k}{2}]+1.
\end{equation*}

\subsection{Step 2: 
Solve \eqref{eqn:Fsys3}
for $\psi \in \Hom_{A'}(\C_{-\nu}, \Pol(\fn_+)\otimes \C_{-\lambda})$.
}

The aim of this step is to determine $\Sol(\fn_+; \triv_{\lambda}, \triv_{\nu})_0$,
that is,
we wish to solve
\begin{equation*}
\widehat{\dpi_{\lambda^*}}(N_1^+)\psi =0,
\end{equation*}
which is equivalent to solving
\begin{equation*}
-\zeta_1\widehat{\dpi_{\lambda^*}}(N_1^+)\psi =0.
\end{equation*}

We put
\begin{align*}
\accentset{\circ}{\gL}^{(3,2)}_{SL,1}&:=\{
(\lambda, \lambda+m)
:  \lambda \in \C \; \text{and} \; m\in \Z_{\geq 0}\},\\
\accentset{\circ}{\gL}^{(3,2)}_{SL,2}&:=\{
(1-(m+\ell), 1+\ell) :  m,  \ell\in \Z_{\geq 0}\},\\
\accentset{\circ}{\gL}^{(3,2)}_{SL,+}&:=\{
(1-(m+\ell), 1+\ell) :  m\in \Z_{\geq 0} \; \text{and} \; \ell\in 1+\Z_{\geq 0}\}.
\end{align*}
We have 
\begin{equation*}
\accentset{\circ}{\gL}^{(3,2)}_{SL,+} \subset 
\accentset{\circ}{\gL}^{(3,2)}_{SL,2} \subset
\accentset{\circ}{\gL}^{(3,2)}_{SL,1}.
\end{equation*}

\begin{prop}\label{prop:MA3}
Let $n=2$. We have
\begin{equation*}
\Sol(\fn_+; \triv_{\lambda}, \triv_{\nu})_0
=
\begin{cases}
\C \psi_{(m,0)} &\text{if $(\lambda,\nu) \in 
\accentset{\circ}{\gL}^{(3,2)}_{SL,1}
\setminus \accentset{\circ}{\gL}^{(3,2)}_{SL,+}$, }\\
\C\psi_{(m+2\ell, 0)} \oplus \C\psi_{(m,\ell)}
&\text{if $(\lambda,\nu) \in \accentset{\circ}{\gL}^{(3,2)}_{SL,+}$,}\\
\{0\} &\text{otherwise}.
\end{cases}
\end{equation*}
\end{prop}

\begin{proof}
As
$\Sol(\fn_+; \triv_{\lambda}, \triv_{\nu})_0\subset
\Hom_{A'}(\C_{-\nu}, \Pol(\fn_+)\otimes \C_{-\lambda})$,
it follows from Proposition \ref{prop:MA2} that if 
$\Sol(\fn_+; \triv_{\lambda}, \triv_{\nu})_0\neq \{0\}$, then
$\nu- \lambda \in \Z_{\geq 0}$. 

Let $\nu- \lambda =k \in \Z_{\geq 0}$. 
It then follows from \eqref{eqn:MA4} that 
\begin{equation*}
\psi \in \Hom_{A'}(\C_{-(\lambda+k)}, \Pol(\fn_+)\otimes \C_{-\lambda})
\end{equation*}
is of the form
\begin{equation*}
\psi = 
\sum_{m+2\ell = k}
c_\ell\, \psi_{(m, \ell)}
=
\sum_{\ell=0}^{[k/2]}
c_\ell\, \psi_{(k-2\ell, \ell)}
\end{equation*}
for some $c_\ell \in \C$.
Since $\psi_{(k-2\ell, \ell)}=\zeta_2^{k-2\ell}\zeta_1^\ell \in \C^{k-\ell}[\zeta_1,\zeta_2]$,
by \eqref{eqn:dNj3}, we have 
\begin{align*}
-\zeta_1\widehat{\dpi_{\lambda^*}}(N_1^+)\psi
&=
\sum_{\ell =0}^{[k/2]}c_\ell \,
(-\zeta_1\widehat{\dpi_{\lambda^*}}(N_1^+)\psi_{(k-2\ell, \ell)})\\
&=
\sum_{\ell =0}^{[k/2]}c_\ell \,
(\lambda-1+k-\ell)\vartheta_1(\psi_{(k-2\ell, \ell)})\\
&=
\sum_{\ell =0}^{[k/2]}
 c_\ell \, \ell (\lambda-1+k-\ell)\psi_{(k-2\ell, \ell)}.
\end{align*}

As the degrees of $\psi_{(k-2\ell,\ell)}$ are all different, the polynomials
$\psi_{(k-2\ell, \ell)}$ are linearly independent.
Therefore,
$-\zeta_1\widehat{\dpi_{\lambda^*}}(N_1^+)\psi=0$
if and only if 
\begin{equation*}
c_\ell \, \ell(\lambda-1+k-\ell)=0
\quad 
\text{for all $\ell \in \{0,1,\ldots, [\tfrac{k}{2}]\}$},
\end{equation*}
which 
is further equivalent to the following conditions:
\begin{enumerate}
\item[(I)] 
$c_\ell =0$ for $\ell \neq 0$, or
\item[(II)]
$\lambda = 1-(k-\ell_0)$ for some $\ell_0$ and $c_j=0$ for $j \neq 0, \ell_0$.
\end{enumerate}

First, suppose that Case (I) holds. 
If $\ell=0$, then $k=m+2\cdot 0 = m$. Therefore, 
\begin{equation*}
\psi = 
\sum_{\ell =0}^{[k/2]}
 c_\ell \, \psi_{(k-2\ell, \ell)}
=c_0\, \psi_{(k,0)}
=c_0\, \psi_{(m,0)}.
\end{equation*}
Since $\lambda \in \C$ can be any complex number,  we have
\begin{equation*}
\psi_{(m,0)}  \in 
\Hom_{A'}(\C_{-(\lambda+m)}, \Pol(\fn_+)\otimes \C_{-\lambda})
\quad
\text{for all $\lambda \in \C$}.
\end{equation*}
Since $m=k \in \Z_{\geq 0}$ is arbitrary, this shows that
if $(\lambda,\nu)=(\lambda,\lambda-m) \in \accentset{\circ}{\gL}^{(3,2)}_{SL,1}$, then
\begin{equation*}
\C\psi_{(m,0)} \subset \Sol(\fn_+; \triv_{\lambda}, \triv_{\nu})_0.
\end{equation*}

Next, suppose that Case (II) holds.
Write $m_0=k-2\ell_0$. Then,
\begin{equation*}
\psi = 
\sum_{\ell =0}^{[k/2]}
c_\ell\, \psi_{(k-2\ell, \ell)}
=c_0\, \psi_{(k,0)} + c_{\ell_0}\, \psi_{(k-2\ell_0,\ell_0)}
=c_0\, \psi_{(m_0+2\ell_0,0)} + c_{\ell_0}\, \psi_{(m_0,\ell_0)}.
\end{equation*}
Since $\lambda = 1-(k-\ell_0)=1-(m_0+\ell_0)$ and $\lambda+k = 1+\ell_0$,
we have
\begin{equation*}
\psi_{(m_0+2\ell_0,0)}, \psi_{(m_0,\ell_0)} \in 
\Hom_{A'}(\C_{-(1+\ell_0)}, \Pol(\fn_+)\otimes \C_{-(1-(m_0+\ell_0))}).
\end{equation*}
As for Case (I), since $k = m_0 + \ell_0$ is arbitrary, this shows that
if $(\lambda,\nu)=(1-(m+\ell), 1+\ell) \in \accentset{\circ}{\gL}^{(3,2)}_{SL,+}$,
then 
\begin{equation*}
\C\psi_{(m+2\ell,0)} \oplus \C \psi_{(m,\ell)} \subset
\Sol(\fn_+; \triv_{\lambda}, \triv_{\nu})_0.
\end{equation*}

Since
$\accentset{\circ}{\gL}^{(3,2)}_{SL,+} \subset 
\accentset{\circ}{\gL}^{(3,2)}_{SL,2} \subset
\accentset{\circ}{\gL}^{(3,2)}_{SL,1}$,
it follows from the arguments on (I) and (II) that
\begin{equation*}
\dim \Sol(\fn_+; \triv_{\lambda}, \triv_{\nu})_0=
\begin{cases}
1 & \text{if $(\lambda,\nu) \in 
\accentset{\circ}{\gL}^{(3,2)}_{SL,1}
\setminus \accentset{\circ}{\gL}^{(3,2)}_{SL,+}$},\\
2& \text{if $(\lambda,\nu) \in \accentset{\circ}{\gL}^{(3,2)}_{SL,+}$},\\
0 & \text{otherwise}.
\end{cases}
\end{equation*}
This concludes the proposition.
\end{proof}

\subsection{Step 3: 
Consider the parity condition on
 $\psi \in \Sol(\fn_+; \triv_{\lambda}, \triv_{\nu})_0$.}
 Recall from \eqref{eqn:SL1b} and \eqref{eqn:SL2c} that
 we have  
 \begin{align*}
\gL^{(3,2)}_{SL,1}&=\{
(\alpha, \alpha+m;\triv;\lambda, \lambda+m)
: \alpha \in \{\pm\}, \lambda \in \C,\text{and} \; m\in \Z_{\geq 0}\},\\
\gL^{(3,2)}_{SL,+}&=\{
(\alpha, \alpha+m; \triv; 1-(m+\ell), 1+\ell) : 
\alpha \in \{\pm\}, m\in \Z_{\geq 0}, \; \text{and} \; \ell\in 1+\Z_{\geq 0}\}.
\end{align*}

 \begin{prop}\label{prop:MA4}
 We have 
 \begin{equation*}
\Sol(\fn_+; \triv_{\ga, \lambda}, \triv_{\beta, \nu})
=
\begin{cases}
\C \psi_{(m,0)} &\text{if $(\lambda,\nu) \in 
\gL^{(3,2)}_{SL,1}\setminus \gL^{(3,2)}_{SL,+}$, }\\
\C\psi_{(m+2\ell, 0)} \oplus \C\psi_{(m,\ell)}
&\text{if $(\lambda,\nu) \in \gL^{(3,2)}_{SL,+}$,}\\
\{0\} &\text{otherwise}.
\end{cases}
 \end{equation*}
 \end{prop}
 
 \begin{proof}
For $n=2$,
it follows from Lemma \ref{lem:MA} and Remark \ref{rem:sym0826a} that
 $M'$ acts on $\C\psi_{(m,0)}=\C\zeta_2^m$ and 
 $\C\psi_{(m,\ell)}=\C \zeta_2^m\zeta_1^\ell$
 by $\sgn^m \otimes \triv$ and 
 $\sgn^{m+\ell}\otimes \sym^\ell_1 = \sgn^m \otimes \triv$, respectively.
Proposition \ref{prop:MA3} then concludes the assertion.
\end{proof}

\subsection{Step 4: 
Apply $\Rest_{x_2=0}\circ \symb^{-1}$ 
and
$F_c^{-1}$ 
to the solution $\psi \in\Sol(\fn_+; \triv_{\alpha, \lambda}, \triv_{\beta, \nu})$
}

Now we finish the proof of the theorems in concern.

\begin{proof}[Proof of 
Theorems \ref{thm:DSBO1b}, \ref{thm:DSBO2b}, \ref{thm:Hom1b},
\ref{thm:Hom2b}, and \ref{thm:Hom3b}]
As in Section \ref{sec:Step4}, we apply $\Rest_{x_2=0}\circ \symb^{-1}$ and $F_c^{-1}$
to the polynomial solutions $\psi$ in Proposition \ref{prop:MA4} to 
obtain
Theorems \ref{thm:DSBO1b}, \ref{thm:DSBO2b}, \ref{thm:Hom1b},
and \ref{thm:Hom2b}. 
The application of  $F_c^{-1}$ to $\psi$ in Proposition \ref{prop:MA3}
concludes Theorem \ref{thm:Hom3b}. This ends the proof. 
\end{proof}

\section{Factorization identities of $\bD_{(m,\ell)}$ and $\Phi_{(m,\ell)}$}
\label{sec:factor1}

The aim of this section is to show the factorization identities of 
differential symmetry breaking operators $\bD_{(m,\ell)}$
and $(\fg', P')$-homomorphisms $\Phi_{(m,\ell)}$.
Such factorization identities are obtained in
Theorem \ref{thm:Emb} for $\Phi_{(m,\ell)}$ and 
 Theorem \ref{thm:Proj} for $\bD_{(m,\ell)}$.
Throughout this section we assume $n \geq 2$, unless otherwise stated.

\subsection{$G$-intertwining differential operators $\D_k$ 
and $(\fg, P)$-homomorphisms $\varphi_k$}
\label{sec:71}

As preliminaries,
we first recall from \cite{KuOr24} the classification of $G$-intertwining differential 
operators 
\begin{equation*}
\D \in \Diff_G(I(\triv, \lambda)^\alpha, I(\xi, \tau)^\delta)
\end{equation*}
and $(\fg, P)$-homomorphisms 
\begin{equation*}
\varphi \in
\Hom_{\fg, P}(\Mp(\tau,u)^\delta, \Mp(\triv, s)^\ga).
\end{equation*}

We define $\Lambda^{n+1}_{SL} 
\subset \{\pm\}^2 \times \Irr(SL(n,\R))_{\fin} \times \C^2$ by means of
\begin{equation*}
\Lambda^{n+1}_{SL}:=\{
(\alpha, \alpha+k;\poly_{n}^k; 1-k, 
1+\tfrac{k}{n} ): \alpha \in \{\pm\} \; \text{and} \; k\in \Z_{\geq 0}\}.
\end{equation*}

\noindent
For $k \in \Z_{\geq 0}$, we set
\begin{equation*}
\Xi_k:=\{(k_1, \ldots, k_{n-1}, k_n) \in (\Z_{\geq 0})^{n} : \sum_{j=1}^{n} 
k_j= k\}
\end{equation*}
and define $\D_k \in \Diff_\C(C^\infty(\R^{n}), C^\infty(\R^{n})\otimes 
\C^k[y_1, \ldots, y_{n-1}, y_{n}])$ as
\begin{equation}\label{eqn:IDO}
\D_k := 
\sum_{\mathbf{k} \in \Xi_k} \frac{\partial^k}{\partial x^{\mathbf{k}}}
\otimes \widetilde{y}_{\mathbf{k}}.
\end{equation}

\noindent
For $k=0$, we understand $\D_0$ as the identity operator $\D_0 = \id$.

\begin{rem}
The differential operator $\D_k$ can also be expressed as follows. 
For $\mathbf{k}, \mathbf{k}' \in \Xi_k$ with
$\mathbf{k}= (k_1, \ldots, k_{n-1}, k_n)$ and  
$\mathbf{k}'= (k_1', \ldots, k_{n-1}', k_n')$,
write
\begin{equation*}
\mathbf{k}+\mathbf{k}' = 
(k_1+k_1', \ldots, k_{n-1}+k_{n-1}', k_n+k_n').
\end{equation*}
We then define
a multiplication on 
$
\C[\tfrac{\partial}{\partial x_1}, \ldots, \tfrac{\partial}{\partial x_{n-1}},
\tfrac{\partial}{\partial x_n}]
\otimes
\C[y_1, \ldots, y_{n-1}, y_{n}]
$
by
\begin{equation*}
(\frac{\partial^k}{\partial x^{\mathbf{k}}} \otimes y_{\mathbf{k}})
\cdot
(\frac{\partial^k}{\partial x^{\mathbf{k}'}} \otimes y_{\mathbf{k}'})
=
\frac{\partial^{2k}}{\partial x^{\mathbf{k}+\mathbf{k}'}} 
\otimes y_{\mathbf{k}+\mathbf{k}'}.
\end{equation*}
Then $\D_k$ can be given by
\begin{equation*}
\D_k 
=
\frac{1}{k!} 
\big(
\sum_{j=1}^n \tfrac{\partial}{\partial x_j} \otimes y_j
\big)^k.
\end{equation*}

\end{rem}

\begin{thm}[{\cite[Thm.\ 4.5]{KuOr24}}]
For $(\ga, \delta; \xi; \lambda, \tau) \in \{\pm\}^2 \times 
\Irr(SL(n,\R)) \times \C^2$,
we have
\begin{equation*}
\Diff_G(I(\triv, \lambda)^\alpha, I(\xi, \tau)^\delta)
=
\begin{cases}
\C\id & \text{if $(\delta, \xi, \tau) = (\alpha, \triv, \lambda)$,}\\
\C \D_k & \text{if $(\alpha, \delta;\xi; \lambda, \tau)\in \Lambda^{n+1}_{SL}$,}\\
\{0\} & \text{otherwise.}
\end{cases}
\end{equation*}
\end{thm}

For $(\fg, P)$-homomorphisms $\varphi$, 
we first define $\Lambda^{n+1}_{(\fg, P)} \subset 
\{\pm\}^2\times \Irr(\f{sl}(n,\C))_{\fin} \times \C^2$ by
\begin{equation}\label{eqn:gPn}
\Lambda^{n+1}_{(\fg, P)}:=\{
(\ga,\ga+k;\sym_{n}^k; k-1,  -(1+\tfrac{k}{n})) : \alpha \in \{\pm\} \; \text{and} \; k\in \Z_{\geq 0}\}.
\end{equation}

We define $\varphi_k \in 
\Hom_\C(S^k(\C^{n}), S^k(\fn_-))$ by means of
\begin{equation}\label{eqn:Hom2}
\varphi_k 
:= \sum_{\mathbf{k} \in \Xi_k} 
N_{\mathbf{k}}^-
\otimes 
(e_{\mathbf{k}})^\vee
= \sum_{\mathbf{k} \in \Xi_k} 
N_{\mathbf{k}}^-
\otimes 
\widetilde{y}_{\mathbf{k}}.
\end{equation}

\begin{thm}[{\cite[Thm.\ 4.8]{KuOr24}}]
\label{thm:Hom}
For 
$(\ga,\delta; \tau; s, u) \in \{\pm\}^2\times \Irr(SL(n,\R))_{\fin} \times \C^2$,
we have
\begin{equation*}
\Hom_{\fg, P}(\Mp(\tau,u)^\delta, \Mp(\triv, s)^\ga)
=
\begin{cases}
\C\id & \text{if $(\delta, \tau, u) = (\ga, \triv, s)$,}\\
\C \varphi_k & \text{if $(\ga, \delta, \tau; s, u)\in \gL^{n+1}_{(\fg,P)}$,}\\
\{0\} & \text{otherwise.}
\end{cases}
\end{equation*}
\end{thm}

As for $\Phi_{(m,\ell)}$ in \eqref{eqn:Phi-def}, we regard 
$\varphi_k$ as a map
\begin{equation*}
\varphi_k \in 
\Hom_{\fg, P}(\Mp(\tau,u)^\delta, \Mp(\triv, s)^\ga)
\end{equation*}
defined by
\begin{equation}\label{eqn:Phi-def2}
\varphi_k(u\otimes w) := u\varphi_k(w)
\quad 
\text{for $u \in \Cal{U}(\fg)$ and $w \in S^k(\C^{n})$.}
\end{equation}

In what follows, we write $\D_k$ and $\D_k'$ for 
$G$- and $G'$-intertwining differential operators, respectively.
The same convention is employed for $(\fg, P)$-homomorphism 
$\varphi_k$ and $(\fg', P')$-homomorphism $\varphi_k'$.

\subsection{Factorization identities for $\Phi_{(m,\ell)}$}
\label{sec:factorGVM}

We first consider the factorization identities of 
$(\fg',P')$-homomorphisms $\Phi_{(m,\ell)}$.
It is clear from \eqref{eqn:Hom} and \eqref{eqn:Hom2} that
the $(\fg',P')$-homomorphism $\Phi_{(m,\ell)}$ can be factored as
\begin{equation*}
\Phi_{(m,\ell)} = \Phi_{(m,0)}\circ \varphi_\ell'.
\end{equation*}
In other words, the following diagram commutes.
\begin{equation*}
 \xymatrix@=13pt{
 \Mpp(\triv,\ell-1)^{\ga+m} 
 \ar@{}[rd]|{\circlearrowleft}
 \ar[rr]^{\Phi_{(m,0)}}
 && \Mp(\triv,(m+\ell)-1)^\ga\\
 &&\\
 \Mpp(\sym^\ell_{n-1},-(1+\tfrac{\ell}{n-1}))^{\ga+m+\ell}
\ar[uu]^{\varphi_{\ell}'}
 \ar[rruu]^{\Phi_{(m,\ell)}} 
}
\end{equation*}

The aim of this subsection is to give another factorization identity of 
$\Phi_{(m,\ell)}$, that is, we wish to complete the lower left corner
of the following diagram:
\begin{equation*}
 \xymatrix@=13pt{
 \Mpp(\triv,\ell-1)^{\ga+m} 
 \ar@{}[rd]|{\circlearrowleft}
 \ar[rr]^{\Phi_{(m,0)}}
 && \Mp(\triv,(m+\ell)-1)^\ga\\
 &&\\
 \Mpp(\sym^\ell_{n-1},-(1+\tfrac{\ell}{n-1}))^{\ga+m+\ell}
\ar[uu]^{\varphi_{\ell}'}
 \ar[rr] \ar[rruu]^{\Phi_{(m,\ell)}} 
&&\fbox{some generalized Verma module}
\ar[uu]
\ar@{}[luu]|{\circlearrowleft}
}
\end{equation*}

For this purpose, we introduce some notation.
For $\mathbf{l} = (\ell_1, \ldots, \ell_{n-1}) \in \Xi_{\ell}'$
and $m \in \Z_{\geq 0}$, we write
\begin{equation*}
e_{(m,\mathbf{l})} 
=e_n^me_{\mathbf{l}}
=  e_n^m e_1^{\ell_1}\cdots e_{n-1}^{\ell_{n-1}}.
\end{equation*}
Also, we write
\begin{equation*}
\C[e'] = \C[e_1, \ldots, e_{n-1}]
\quad
\text{and}
\quad
\C[e',e_n] = \C[e_1,\ldots, e_{n-1}, e_n].
\end{equation*}
Then, for $m,\ell \in \Z_{\geq 0}$, we define
\begin{equation*}
\Emb_{(m,\ell)} \in \Hom_{\C}(\C^\ell[e'], \C^{m+\ell}[e',e_n])
\end{equation*}
by means of
\begin{equation}\label{eqn:Emb}
\Emb_{(m,\ell)} 
:= \sum_{\mathbf{l}\in\Xi_{\ell}'}
(e_{\mathbf{l}})^\vee \otimes e_{(m,\mathbf{l})}
= \sum_{\mathbf{l}\in\Xi_{\ell}'}
\wy_{\mathbf{l}} \otimes e_{(m,\mathbf{l})}.
\end{equation}

\begin{prop}\label{prop:Emb1}
We have
\begin{equation*}
\Hom_{M'}(\C^\ell[e'], \C^{m+\ell}[e',e_n]) = \C \Emb_{(m,\ell)}.
\end{equation*}
\end{prop}

\begin{proof}
By definition, we have
\begin{equation*}
\Emb_{(m,\ell)}\colon 
\sum_{\mathbf{l}\in\Xi_{\ell}'}
a_{\mathbf{l}}\, e_{\mathbf{l}} 
\, \longmapsto \,
\sum_{\mathbf{l}\in\Xi_{\ell}'}
a_{\mathbf{l}}\, e_{(m,\mathbf{l})}
= e_n^m\sum_{\mathbf{l}\in\Xi_{\ell}'}
a_{\mathbf{l}}\, e_{\mathbf{l}}.
\end{equation*}
Thus $\Im(\Emb_{(m,\ell)}) = \C^m[e_n]\C^{\ell}[e']$.
Since $M'$ acts on $\C^k[e_n]$ trivially, this shows that $\Emb_{(m,\ell)}$ respects
the $M'$-action. Now the proposition follows from the fact that
$\C^m[e_n]\C^{\ell}[e']$ has multiplicity-one in $\C^{m+\ell}[e',e_n]$.
\end{proof}

We put
\begin{equation*}
\mu':=1+\tfrac{\ell}{n-1}
\quad 
\text{and}
\quad
\mu:=1+\tfrac{m+\ell}{n}.
\end{equation*}
As in \eqref{eqn:Pol},
we define 
$M$-representation $\C^\ell[e',e_n]_\ga$
and
$M'$-representations $\C^\ell[e']_\ga$ 
for $\ga \in \{\pm\}$
as
\begin{equation*}
\C^\ell[e']_\ga = \C_{\ga} \otimes \C^\ell[e']
\quad
\text{and}
\quad
\C^\ell[e',e_n]_\ga = \C_{\ga} \otimes \C^\ell[e',e_n].
\end{equation*}

\begin{prop}\label{prop:Emb2a}
We have
\begin{equation*}
\Hom_{M'A'}(\C^\ell[e']_\ga \boxtimes \C_{-\mu'},
\C^{m+\ell}[e',e_n]_\ga\boxtimes \C_{-\mu})
=\C \Emb_{(m,\ell)}.
\end{equation*}
\end{prop}

\begin{proof}
By Proposition \ref{prop:Emb1}, it suffices to show
\begin{equation*}
\Emb_{(m,\ell)}\in
\Hom_{M'A'}(\C^\ell[e']_\ga \boxtimes \C_{-\mu'},
\C^{m+\ell}[e',e_n]_\ga\boxtimes \C_{-\mu}).
\end{equation*}
As $\Im(\Emb_{(m,\ell)}) = \C^m[e_n]\C^{\ell}[e']$
and $\Emb_{(m,\ell)}$ respects the $M'$-action, 
it is further enough to show 
\begin{equation*}
\C^\ell[e']\boxtimes \C_{-\mu'} 
\simeq
\C^m[e_n]\C^\ell[e']\boxtimes \C_{-\mu}
\end{equation*}
as $A'$-modules.
Since $A'$ acts on $\C^\ell[e']\boxtimes \C_{-\mu'}$ by a character 
with weight $-\mu'$, one wishes to show that $A'$ acts on 
$\C^m[e_n]\C^\ell[e']\boxtimes \C_{-\mu}$ also by $-\mu'$.

Observe that
the action of $M'A'$ on $\C^m[e_n]\C^\ell[e']\boxtimes \C_{-\mu}$
comes from the restriction of that of $MA$ on $\C^{m+\ell}[e',e_n]\boxtimes \C_{-\mu}$.
Recall from Remark \ref{rem:A} that $a'$ can be decomposed as
$a' = a'_M a'_A$ with $a'_M \in M$ and $a'_A \in A$. For 
\begin{equation*}
a' = \diag(t, t^{\frac{-1}{n-1}}, \ldots, t^{\frac{-1}{n-1}}, 1) \in A',
\end{equation*}
we have
\begin{align*}
a'_M 
&= \diag(1,t^{\frac{-1}{n(n-1)}},\ldots, t^{\frac{-1}{n(n-1)}}, t^{\frac{1}{n}})
\in M,\\
a'_A
&= \diag(t,t^{\frac{-1}{n}},\ldots, t^{\frac{-1}{n}}, t^{\frac{-1}{n}}) 
\in A.
\end{align*}
Thus, for $e_n^mp(e')\otimes \mathbb{1}_{-\mu}
\in \C^m[e_n]\C^\ell[e']\boxtimes \C_{-\mu}$, we have 
\begin{align*}
a'\cdot (e_n^mp(e')\otimes \mathbb{1}_{-\mu})
&=(a'_M \cdot e_n^m)(a'_M \cdot p(e'))\otimes (a'_A \cdot \mathbb{1}_{-\mu})\\
&=t^{\tfrac{m}{n}}\cdot 
t^{\tfrac{-\ell}{n(n-1)}} \cdot t^{-(1+\tfrac{m+\ell}{n})}(e_n^mp(e')\otimes \mathbb{1}_{-\mu})\\
&=t^{-(1+\tfrac{\ell}{n-1})}(e_n^mp(e')\otimes \mathbb{1}_{-\mu}).
\end{align*}
Therefore, $A'$ acts on $\C^m[e_n]\C^\ell[e']\boxtimes \C_{-\mu}$ also 
by $-\mu' = -(1+\tfrac{\ell}{n-1})$. 
Now the proposition follows.
\end{proof}

Now we define
\begin{equation*}
\wEmb_{(m,\ell)} \in 
\Hom_{M'A'}(\C^\ell[e']_\ga\boxtimes \C_{-\mu'},
\Cal{U}(\fn_-) \otimes (\C^{m+\ell}[e',e_n]_\ga\boxtimes \C_{-\mu}))
\end{equation*}
as a map
\begin{align*}
\wEmb_{(m,\ell)}\colon p(e')\otimes \mathbb{1}_{\mu'}
\longmapsto
1\otimes \Emb_{(m,\ell)}(p(e')) \otimes \mathbb{1}_{\mu}.
\end{align*}

We inflate the $MA$-representation
$\C^{m+\ell}[e',e_n]_\ga\boxtimes \C_{-\mu}$ 
to a $P$-representation by letting $N_+$ act trivially.
Similarly, we inflate the $M'A'$-representation 
$\C^\ell[e']_\ga \boxtimes \C_{-\mu'}$
to a $P'$-representation. 
Then, as $P' \subset P$, we have
\begin{equation*}
\wEmb_{(m,\ell)} \in \Hom_{P'}(\C^\ell[e']_\ga\boxtimes \C_{-\mu'},
\Mp(\sym^{m+\ell}_n, -\mu)^\ga).
\end{equation*}
Equivalently, 
by applying to $\wEmb_{(m,\ell)}$ the same convention
as \eqref{eqn:Phi-def} and \eqref{eqn:Phi-def2}, we have
\begin{equation*}
\wEmb_{(m,\ell)} \in \Hom_{\fg',P'}(\Mpp(\sym^\ell_{n-1},-\mu')^\ga,
\Mp(\sym^{m+\ell}_n, -\mu)^\ga).
\end{equation*}

\noindent
As a consequence, we obtain the following.

\begin{prop}
We have
\begin{equation*}
\Hom_{\fg',P'}(\Mpp(\sym^\ell_{n-1},-\mu')^\ga,
\Mp(\sym^{m+\ell}_n, -\mu)^\ga) = \C\wEmb_{(m,\ell)}.
\end{equation*}
\end{prop}

\begin{proof}
This is a direct consequence of Proposition \ref{prop:Emb2a}
and the preceding arguments.
\end{proof}

Now we are ready to show another factorization identity of $\Phi_{(m,\ell)}$.

\begin{thm}\label{thm:Emb}
Let $n\geq 2$.
For $(\alpha,\beta; \sigma; s, r) \in \gL^{(n+1,n)}_{(\fg', P'),2}$,
the $(\fg', P')$-homomorphism $\Phi_{(m,\ell)}$ can be factored as follows:
\begin{equation}\label{eqn:Homfactor}
\Phi_{(m,\ell)} 
= \Phi_{(m,0)} \circ \varphi'_{\ell}
=\varphi_{m+\ell} \circ \wEmb_{(m,\ell)}.
\end{equation}
Equivalently, the following diagram commutes.
\begin{equation}\label{eqn:comm1}
\begin{aligned}
 \xymatrix@=13pt{
 \Mpp(\triv,\ell-1)^{\ga+m} 
 \ar@{}[rd]|{\circlearrowleft}
 \ar[rr]^{\Phi_{(m,0)}}
 && \Mp(\triv,(m+\ell)-1)^\ga\\
 &&\\
 \Mpp(\sym^\ell_{n-1},-(1+\tfrac{\ell}{n-1}))^{\ga+m+\ell}
\ar[uu]^{\varphi_{\ell}'}
 \ar[rr]_{\wEmb_{(m,\ell)}} \ar[rruu]^{\Phi_{(m,\ell)}} 
&&\Mp(\sym^{m+\ell}_n,-(1+\tfrac{m+\ell}{n}))^{\ga+m+\ell}
\ar[uu]_{\varphi_{m+\ell}}
\ar@{}[luu]|{\circlearrowleft}
}
\end{aligned}
\end{equation}
\end{thm}

\begin{proof}
Since any element in 
$\Mpp(\sym^\ell_{n-1},-(1+\tfrac{\ell}{n-1}))^{\ga+m+\ell}$ is of the form 
$\sum_{\mathbf{l} \in \Xi'_{\ell}}u_{\mathbf{l}} \otimes e_{\mathbf{l}}$
for $u_{\mathbf{l}} \in \Cal{U}(\fg)$ and $e_{\mathbf{l}} \in S^\ell(\C^{n-1})$,
it suffices to show the identities \eqref{eqn:Homfactor} for 
$u_{\mathbf{l}} \otimes e_{\mathbf{l}}$ for $\mathbf{l} \in \Xi_\ell'$. 

Recall from \eqref{eqn:Hom} that $\Phi_{(m,\ell)}$ is given by
\begin{equation*}
\Phi_{(m,\ell)}
= 
(N_n^-)^m
\sum_{\mathbf{l} \in \Xi_\ell'} 
N_{\mathbf{l}}^-
\otimes 
(e_{\mathbf{l}})^\vee
= 
(N_n^-)^m
\sum_{\mathbf{l} \in \Xi_\ell'} 
N_{\mathbf{l}}^-
\otimes 
\widetilde{y}_{\mathbf{l}}.
\end{equation*}
Then, by \eqref{eqn:Phi-def}, we have 
\begin{equation*}
\Phi_{(m,\ell)} (u_{\mathbf{l}} \otimes e_{\mathbf{l}})
=u_{\mathbf{l}} \Phi_{(m,\ell)}(e_{\mathbf{l}})
=u_{\mathbf{l}} (N^-_n)^m N^-_{\mathbf{l}}.
\end{equation*}

By \eqref{eqn:Hom2} and \eqref{eqn:Phi-def2},
the element 
$(\Phi_{(m,0)} \circ \varphi'_{\ell})(u_{\mathbf{l}} \otimes e_{\mathbf{l}})$ is given by
\begin{align*}
(\Phi_{(m,0)} \circ \varphi'_{\ell})(u_{\mathbf{l}} \otimes e_{\mathbf{l}})
&=u_{\mathbf{l}} (\Phi_{(m,0)} \circ \varphi'_{\ell})(e_{\mathbf{l}})\\
&=u_{\mathbf{l}} \Phi_{(m,0)}(N^-_{\mathbf{l}})\\
&=u_{\mathbf{l}} (N^-_n)^m N^-_{\mathbf{l}}.
\end{align*}

Finally, we have 
\begin{align*}
(\varphi_{m+\ell} \circ \wEmb_{(m,\ell)})(u_{\mathbf{l}} \otimes e_{\mathbf{l}})
&=u_{\mathbf{l}}(\varphi_{m+\ell} \circ \wEmb_{(m,\ell)})(e_{\mathbf{l}})\\
&=u_{\mathbf{l}}\varphi_{m+\ell}(1\otimes e^m_n e_{\mathbf{l}})\\
&=u_{\mathbf{l}} (N^-_n)^m N^-_{\mathbf{l}}.
\end{align*}
This completes the proof.
\end{proof}

\begin{rem}\label{rem:sym0826b}
It follows from \eqref{eqn:sym0826c} that if $n=2$, then the factorization 
identity \eqref{eqn:comm1} becomes
\begin{equation*}
\begin{aligned}
 \xymatrix@=13pt{
 \Mpp(\triv,\ell-1)^{\ga+m} 
 \ar@{}[rd]|{\circlearrowleft}
 \ar[rr]^{\Phi_{(m,0)}}
 && \Mp(\triv,(m+\ell)-1)^\ga\\
 &&\\
 \Mpp(\sym^\ell_{1},-(1+\ell))^{\ga+m+\ell}
\ar[uu]^{\varphi_{\ell}'}
 \ar[rr]_{\wEmb_{(m,\ell)}} \ar[rruu]^{\Phi_{(m,\ell)}} 
 \ar@{}[d]|{\rotatebox{90}{$=$}}
&&\Mp(\sym^{m+\ell}_2,-(1+\tfrac{m+\ell}{2}))^{\ga+m+\ell}
\ar[uu]_{\varphi_{m+\ell}}
\ar@{}[luu]|{\circlearrowleft}\\
 \Mpp(\triv,-(1+\ell))^{\ga+m} &&
}
\end{aligned}
\end{equation*}
\end{rem}

\begin{rem}
In Section \ref{sec:GVM}, we shall discuss the commutative diagram \eqref{eqn:comm1}
from an aspect of the branching laws of generalized Verma modules
(see Remark \ref{rem:branching916}). 
\end{rem}

\subsection{Factorization identities for $\bD_{(m,\ell)}$}
\label{sec:factorDSBO}

Now we consider the factorization identities of the differential symmetry breaking 
operator $\bD_{(m,\ell)}$. Recall that 
$\wy_{\mathbf{l}}\otimes e_{(m,\mathbf{l})}$ may be identified with
\begin{equation}\label{eqn:wy}
\wy_{\mathbf{l}}\otimes e_{(m,\mathbf{l})}
=\wy_{\mathbf{l}}\otimes (\wy_{(m,\mathbf{l})})^\vee
\in \Hom_\C(\C^{m+\ell}[y',y_n], \C^\ell[y']),
\end{equation}
where $\wy_{(m,\mathbf{l})} = y_n^m\wy_{\mathbf{l}}$.
Also, recall that, by the duality theorem (Theorem \ref{thm:duality}), we have
\begin{align*}
\Hom_{\fg',P'}(\Mpp(\sym^\ell_{n-1},-\mu')^\ga, \Mp(\sym^{m+\ell}_n,-\mu)^\ga)
\stackrel[\EuD_{H\to D}]{\sim}{\To}
\Diff_{G'}(I(\poly^{m+\ell}_n,\mu)^\ga,
J(\poly^\ell_{n-1},\mu')^\ga).
\end{align*}
Write
\begin{equation*}
\wProj_{(m,\ell)} = \EuD_{H\to D}(\wEmb_{(m,\ell)}).
\end{equation*}
Via the identification \eqref{eqn:wy}, we have 
\begin{equation*}
\wProj_{(m,\ell)} = \Rest_{x_n=0} \circ
\sum_{\mathbf{l} \in \Xi'_{\ell}}
\id \otimes \wy_{\mathbf{l}} \otimes e_{(m,\mathbf{l})}.
\end{equation*}
That is, for $F(x',x_n) \in C^\infty(\R^{n})\otimes \C^\ell[y']$ with
\begin{align*}
F(x',x_n) 
&=\sum_{\mathbf{k} \in \Xi_{m+\ell}} f_{\mathbf{k}}(x',x_n) \otimes \wy_{\mathbf{k}}\\
&=\sum_{r=0}^{m+\ell}
\sum_{\mathbf{r} \in \Xi_r'}f_{(m+\ell-r, \mathbf{r})}(x',x_n) \otimes 
\wy_{(m+\ell-r,\mathbf{r})},
\end{align*}
we have
\begin{align*}
(\wProj_{(m,\ell)}F)(x')
&=\sum_{\mathbf{l} \in \Xi_\ell} f_{(m,\mathbf{l})}(x',0)
\otimes \wy_{\mathbf{l}}.
\end{align*}

The following is the differential-operator counterpart of Theorem \ref{thm:Emb}.

\begin{thm}\label{thm:Proj}
Let $n\geq 2$. For
$(\ga,\gb;\varpi;\lambda ,\nu) \in 
\gL^{(n+1,n)}_{SL,2}$,
the differential symmetry breaking operator
$\bD_{(m,\ell)}$ can be factored as follows:
\begin{equation*}
\bD_{(m,\ell)} 
= \D'_\ell \circ \bD_{(m,0)}
=  \wProj_{(m,\ell)}\circ \D_{m+\ell}.
\end{equation*}
Equivalently, the following diagram commutes.
\begin{equation}\label{eqn:factor}
\begin{aligned}
\xymatrix@=13pt{
I(\triv, 1-(m+\ell))^\ga
 \ar[dd]_{\D_{m+\ell}}
  \ar[rrdd]^{\bD_{(m,\ell)}}
  \ar[rr]^{\bD_{(m,0)}}
 && J(\triv, 1-\ell)^{\ga+m}
  \ar[dd]^{\D'_\ell}
  \ar@{}[ldd]|{\circlearrowleft}
 \\
&&\\
I(\poly^{m+\ell}_n, 1+\tfrac{m+\ell}{n})^{\ga+m+\ell}
\ar[rr]_{\wProj_{(m,\ell)}}   \ar@{}[ruu]|{\circlearrowleft}
&&
J(\poly^\ell_{n-1}, 1+\tfrac{\ell}{n-1})^{\ga+m+\ell}
}
\end{aligned}
\end{equation}

\end{thm}

\begin{proof}
This simply follows from Theorem \ref{thm:Emb} and the duality theorem.
\end{proof}

\begin{rem}
As for Remark \ref{rem:sym0826b},
it follows from \eqref{eqn:poly0826c} that if $n=2$, then the factorization 
identity \eqref{eqn:factor} is given as
\begin{equation*}
\begin{aligned}
\xymatrix@=13pt{
I(\triv, 1-(m+\ell))^\ga
 \ar[dd]_{\D_{m+\ell}}
  \ar[rrdd]^{\bD_{(m,\ell)}}
  \ar[rr]^{\bD_{(m,0)}}
 && J(\triv, 1-\ell)^{\ga+m}
  \ar[dd]^{\D'_\ell}
  \ar@{}[ldd]|{\circlearrowleft}
 \\
&&\\
I(\poly^{m+\ell}_2, 1+\tfrac{m+\ell}{2})^{\ga+m+\ell}
\ar[rr]_{\wProj_{(m,\ell)}}   \ar@{}[ruu]|{\circlearrowleft}
&&
J(\poly^\ell_{1}, 1+\ell)^{\ga+m+\ell}
\ar@{}[d]|{\rotatebox{90}{$=$}}\\
&&
J(\triv, 1+\ell)^{\ga+m}
}
\end{aligned}
\end{equation*}

\end{rem}

\begin{rem}
The commutative diagram \eqref{eqn:factor} implies that we have
\begin{equation*}
\bD_{(m,0)}\vert_{\Ker(\D_{m+\ell})}\colon 
\Ker(\D_{m+\ell}) \To \Ker(\D'_\ell)
\end{equation*}
(see \eqref{eqn:intro-factor4}).
In Proposition \ref{prop:im4} below, we shall show that  
$\Im(\bD_{(m,0)}\vert_{\Ker(\D_{m+\ell})})= \Ker(\D'_\ell)$.
\end{rem}

\section{The $SL(n,\R)$-representations on the image $\Im(\bD)$}
\label{sec:image}

The aim of this section is to determine the image $\Im(\bD)$ of the 
differential symmetry breaking operators
$\bD=\bD_{(m,0)}, \bD_{(m,\ell)}$ 
in the commutative diagram \eqref{eqn:factor}
for $(\ga,\gb;\varpi;\lambda ,\nu) \in 
\gL^{(n+1,n)}_{SL,2}$,
where
\begin{equation*}
\gL^{(n+1,n)}_{SL,2}=\{
(\alpha, \alpha+m+\ell;\poly_{n-1}^\ell; 1-(m+\ell), 
1+\tfrac{\ell}{n-1}) : \alpha \in \{\pm\} \; \text{and} \; \ell, m\in \Z_{\geq 0}\}.
\end{equation*}
Here, by abuse of notation, the image $\Im(\bD)$ is understood as the underlying $(\fg',K')$-module. 
These are achieved in Section \ref{sec:image2}.

To make the argument simpler, we only consider the case $m, \ell \in 1+\Z_{\geq 0}$.
In this section we assume $n\geq 3$, unless otherwise specified.

\subsection{Preliminaries on $I(\triv, \lambda)^\ga$ and $J(\triv, \nu)^\gb$}

We first recall from \cite{KuOr24} necessary facts on the induced 
representations $I(\triv, \lambda)^\ga$  of $G=SL(n+1,\R)$ and 
$J(\triv, \nu)^\gb$ of $G'=SL(n,\R)$. 

\begin{fact}[cf.\ \cite{HL99, MS14, vDM99}]\label{thm:vDM}
Let $n\geq 3$.
For $\gb \in \{\pm\} \equiv \{\pm 1 \}$ and $\nu \in \C$, 
the induced representation $J(\triv, \nu)^\gb$ enjoys the following.

\begin{enumerate}
\item[\text{(1)}]
The induced representation $J(\triv, \nu)^\gb$ is irreducible except the following
two cases.
\vskip 0.05in
\begin{enumerate}
\item[\text{(A)}]
$\nu \in -\Z_{\geq 0}$ and $\gb = (-1)^\nu$.
\vskip 0.05in

\item[\text{(B)}]
$\nu \in n+\Z_{\geq 0}$ and $\gb = (-1)^{\nu+n}$.
\end{enumerate}

\vskip 0.1in

\item[\text{(2)}]
For Case (A) with $\nu = -m$, there exists 
a finite-dimensional irreducible subrepresentation
$F_{G'}(-m)^\gb \subset J(\triv, -m)^\gb$ such that 
$J(\triv, -m)^\gb/F_{G'}(-m)^\gb$ is irreducible and 
infinite-dimensional. 

\vskip 0.1in

\item[\text{(3)}]
For Case (B) with $\nu =n+m$, 
there exists an infinite-dimensional
irreducible subrepresentation
$T_{G'}(n+m)^\gb \subset J(\triv, n+m)^\gb$ such that 
$J(\triv, n+m)^\gb/T_{G'}(n+m)^\gb$
is irreducible and finite-dimensional.

\vskip 0.1in

\item[\text{(4)}]
For $m \in \Z_{\geq 0}$,
the following non-split exact sequences of Fr{\'e}chet $G'$-modules hold:
\begin{alignat*}{4}
\{0\} &\To F_{G'}(-m)^\gb &&\To J(\triv, -m)^\gb &&\To T_{G'}(n+m)^\gb &&\To \{0\},\\
\{0\} &\To T_{G'}(n+m)^\gb &&\To J(\triv, n+m)^\gb &&\To F_{G'}(-m)^\gb &&\To \{0\}.
\end{alignat*}
\end{enumerate}
\end{fact}

\begin{thm}[{\cite[Thm.\ 6.5]{KuOr24}}]
\label{thm:IDO3}
Let $n \geq 3$, For $\gb \in \{\pm\} \equiv \{\pm 1\}$ and
$k \in \Z_{\geq 0}$,
the kernel $\Ker(\D_k')$ and image $\Im(\D_k')$ of
$G'$-intertwining differential operator
\begin{equation*}
\D_k' \colon 
J(\triv, 1-k)^\gb \To J(\poly^k_{n-1}, 1+\tfrac{k}{n-1})^{\gb+k}
\end{equation*}
are given as follows.
\begin{enumerate}
\item[\emph{(1)}]
$k=0:$ We have
\begin{equation*}
\Ker(\D_0')^\beta = \{0\}
\quad
\text{and}
\quad
\Im(\D_0')^\beta = J(\triv, 1)^\gb,
\end{equation*}
\item[\emph{(2)}]
$k \in 1+\Z_{\geq 0}:$
We have 
\begin{align*}
\Ker(\D_k')^\beta &= 
\begin{cases}
F_{G'}(1-k)^\beta & \text{if $\gb = (-1)^{1-k}$,}\\[3pt]
\{0\} & \text{otherwise},
\end{cases}\\[3pt]
\Im(\D_k')^\beta &\simeq
\begin{cases}
T_{G'}(n+k-1)^\beta & \text{if $\gb = (-1)^{1-k}$,}\\[3pt]
J(\triv,1-k)^\gb & \text{otherwise}.
\end{cases}
\end{align*}
\end{enumerate}

\end{thm}

In what follows, we simply write
$F_{G'}(-k)$, $T_{G'}(n+k)$, $\Ker(\D'_k)$, and $\Im(\D'_k)$.
Likewise,
we denote by $F_{G}(-k)$ and $T_{G}(n+1+k)$  the composition factors
of $I(\triv,\lambda)^\ga$.

To end this subsection, we compute 
the submodule $F_G(1-k)$ of $I(\triv, \lambda)^\ga$ in the noncompact picture
of $I(\triv, \lambda)^\ga$,
that is, the realization of $I(\triv, \lambda)^\ga$ as 
$I(\triv, \lambda)^\ga \subset C^\infty(N_-) \simeq C^\infty(\R^n)$. 

Let $\dpi_{\lambda}$ be the infinitesimal representation of 
$\fg= \fn_-\oplus\fl \oplus \fn_+$ on 
$I(\triv, \lambda)^\ga$ in the noncompact picture. 
Then, as in \eqref{eqn:dpi3},
for $f \in I(\triv, \lambda)^\ga \subset C^\infty(N_-)$,
we have
\begin{equation}\label{eqn:formula2}
d\pi_{\lambda}(X)f(\bar{n})
=\lambda d\chi((\Ad(\bar{n}^{-1})X)_\fl)f(\bar{n})
-\left(dR((\Ad(\cdot^{-1})X)_{\fn_-})f\right)(\bar{n}).
\end{equation}
Here $d\chi$ denotes the differential of the character $\chi$ defined in 
\eqref{eqn:chi}.
(For the details of \eqref{eqn:formula2}, see, for instance, 
\cite[Sect.\ 2.2]{KuOr24}.)

The following lemma is used to compute $F_G(1-k)$.

\begin{lem}\label{lem:dpi}
Let $\dpi_{\lambda}$ be the infinitesimal representation of 
$\fg= \fn_-\oplus\fl \oplus \fn_+$ on 
$I(\triv, \lambda)^\ga$ in the noncompact picture. 
Via the diffeomorphism \eqref{eqn:coord},
the following hold.
\begin{enumerate}
\item[\emph{(1)}] For $N_j^+ \in \fn_+$ with $j \in \{1,\ldots, n\}$, we have
\begin{equation*}
\dpi_{\lambda}(N_j^+) = x_j(\lambda + E_x),
\end{equation*}
where $E_x$ denotes the Euler homogeneity operator for $x$. In particular,
\begin{equation*}
\dpi_{\lambda}(N_j^+)\vert_{\C^a[x_1,\ldots, x_n]} = (\lambda +a) x_j.
\end{equation*}
\vskip 0.1in

\item[\emph{(2)}] For $N_j^- \in \fn_-$ with $j \in \{1,\ldots, n\}$, we have
\begin{equation*}
\dpi_{\lambda}(N_j^-)=-dR(N_j^-)=-\frac{\partial}{\partial x_j},
\end{equation*}
where $dR(N_j^-)$ denotes the infinitesimal right translation of $N_j^-$ on
$\C^\infty(N_-)\simeq \C^\infty(\R^n)$.

\vskip 0.1in

\item[\emph{(3)}] For $Z \in \fl$, we have
\begin{equation*}
\dpi_{\lambda}(Z)\vert_{\C^a[x_1,\ldots, x_n]} \subset \C^a[x_1,\ldots, x_n].
\end{equation*}

\end{enumerate}
\end{lem}

\begin{proof}
Since 
each case can be shown similarly by computing \eqref{eqn:formula2},
we only demonstrate a proof of Case (3) here, provided that Case (2) is proven.

Take $Z \in \fl$. Then, for $\nbar=\exp(\sum_{j=1}^nx_j N_j^-)$, we have 
\begin{equation*}
\Ad(\nbar^{-1})Z = \exp(\ad(\sum_{j=1}^n x_j N_j^-)Z)
=Z -\sum_{j=1}^nx_j[N_j^-,Z].
\end{equation*}
Since $Z\in \fl$ and  $\sum_{j=1}^nx_j[N_j^-,Z] \in \fn_-$,
the value of $\lambda d\chi((\Ad(\nbar^{-1})Z)_{\fl})$ is given by
\begin{equation*}
\lambda d\chi((\Ad(\nbar^{-1})Z)_{\fl}) = \lambda d\chi(Z) \in \C.
\end{equation*}

Next,
observe that , as $[N_j^-, Z] \in \fn_-$, the bracket $[N_j^-, Z]$ is of the form
\begin{equation*}
[N_j^-,Z] = \sum_{r=1}^n a_{rj}N_r^- 
\quad
\text{for some $a_{rj} \in \C$}.
\end{equation*}
Therefore,
\begin{align*}
-dR((\Ad(\nbar^{-1})Z)_{\fn_-}) 
&= -\sum_{j=1}^nx_jdR([N_j^-,Z])\\
&= -\sum_{j,r=1}^na_{rj} x_jdR(N_r^-)\\
&= \sum_{j,r=1}^na_{rj} x_j\frac{\partial}{\partial x_r}.
\end{align*}
We note that the formula for Case (2) is used from line two to line three.
Thus, we have 
\begin{equation*}
dR((\Ad(\nbar^{-1})Z)_{\fn_-}) \vert_{\C^a[x_1,\ldots, x_n]} \subset 
\C^a[x_1, \ldots, x_n].
\end{equation*}
Since $\dpi_\lambda(Z)= 
\lambda d\chi((\Ad(\nbar^{-1})Z)_{\fl}) - dR((\Ad(\nbar^{-1})Z)_{\fn_-})$,
this completes the proof.
\end{proof}

\begin{prop}\label{prop:FG}
Let $k \in 1+\Z_{\geq 0}$.
In the noncompact picture that $I(\triv, 1-k)^\ga \subset C^\infty(\R^n)$, we have 
\begin{equation*}
F_G(1-k) = 
\bigoplus_{a=0}^{k-1}\C^a[x_1, \ldots, x_n].
\end{equation*}

\begin{proof}
Since $F_G(1-k)$ is a finite-dimensional irreducible representation of $G$,
there exists a lowest weight vector $f_0(x_1,\ldots,x_n) \in F_G(1-k)$. As being
a lowest weight vector, we have 
\begin{equation*}
\dpi_{1-k}(N_j^-)f_0(x_1,\ldots, x_n)=0
\quad 
\text{for all $j \in \{1,\ldots, n\}$}.
\end{equation*}
By Lemma \ref{lem:dpi} (2), this is equivalent to 
\begin{equation*}
\frac{\partial}{\partial x_j} f_0(x_1,\ldots, x_n) = 0
\quad 
\text{for all $j \in \{1,\ldots, n\}$},
\end{equation*}
which shows that $f_0$ is a constant function.

Now observe that, by Lemma \ref{lem:dpi} (1), we have 
\begin{equation*}
\dpi_{1-k}(N_j^+)\vert_{\C^a[x_1,\ldots, x_n]} =(1-k + a )x_j.
\end{equation*}
Thus, the space $\Cal{U}(\fn_+)f_0 \subset F_G(1-k)$ is given by
\begin{equation*}
\Cal{U}(\fn_+)f_0 = \bigoplus_{a=0}^{k-1}\C^a[x_1, \ldots, x_n].
\end{equation*}

It follows from Lemma \ref{lem:dpi} that 
$\bigoplus_{a=0}^{k-1}\C^a[x_1, \ldots, x_n]$ is a non-zero $\fg$-submodule
of $F_G(1-k)$.
Now the irreducibility of $F_G(1-k)$ concludes the proposition.
\end{proof}

\end{prop}

\subsection{The image $\Im(\bD)$}
\label{sec:image2}

We now determine the  $SL(n,\R)$-representations on the image $\Im(\bD)$ 
of $\bD=\bD_{(m,0)}, \bD_{(m,\ell)}$.
Throughout this subsection we assume that 
$(\ga,\gb;\varpi;\lambda ,\nu) \in \gL^{(n+1,n)}_{SL,2}$; in particular,
the following commutative diagram holds.
\begin{equation*}
\xymatrix@=13pt{
I(\triv, 1-(m+\ell))^\ga
 \ar[dd]_{\D_{m+\ell}}
  \ar[rrdd]^{\bD_{(m,\ell)}}
  \ar[rr]^{\bD_{(m,0)}}
 && J(\triv, 1-\ell)^{\ga+m}
  \ar[dd]^{\D'_\ell}
  \ar@{}[ldd]|{\circlearrowleft}
 \\
&&\\
I(\poly^{m+\ell}_n, 1+\tfrac{m+\ell}{n})^{\ga+m+\ell}
\ar[rr]_{\wProj_{(m,\ell)}}   \ar@{}[ruu]|{\circlearrowleft}
&&
J(\poly^\ell_{n-1}, 1+\tfrac{\ell}{n-1})^{\ga+m+\ell}
}
\end{equation*}

\begin{prop}\label{prop:im1}
We have
\begin{equation*}
\Im(\bD_{(m,0)})=J(\triv,1-\ell)^{\ga+m}.
\end{equation*}
\end{prop}

\begin{proof}
Without loss of generality, we assume that $J(\triv,1-\ell)^{\ga+m}$ 
is reducible. By Fact \ref{thm:vDM}, the only possibilities of $\Im(\bD_{(m,0)})$ are 
$\{0\}$, $F_{G'}(1-\ell)$, or $J(\triv, 1-\ell)^{\ga +m}$. 
It follows from Theorem \ref{thm:IDO3} that $F_{G'}(1-\ell) = \Ker(\D_\ell')$.
Therefore, if $\Im(\bD_{(m,0)}) \in \{\{0\}, F_{G'}(1-\ell)\}$, then
\begin{equation*}
\bD_{(m,\ell)} = \D_\ell' \circ \bD_{(m,0)} \equiv 0,
\end{equation*}
which contradicts the fact that $\bD_{(m,\ell)} \not\equiv 0$.
Hence, $\Im(\bD_{(m,0)}) = J(\triv, 1-\ell)^{\ga+m}$.
\end{proof}

\begin{prop}\label{prop:im2}
The following hold.
\begin{enumerate}
\item[\emph{(1)}] Suppose $\ga = (-1)^m$. Then we have
\begin{equation*}
\Im(\bD_{(m,\ell)})
=
\begin{cases}
T_{G'}(n+\ell-1) &\text{if $\ell \in 1+2\Z_{\geq 0}$},\\[3pt]
J(\triv, 1-\ell) &\text{if $\ell \in 2(1+\Z_{\geq 0})$}.
\end{cases}
\end{equation*}

\item[\emph{(2)}] Suppose $\ga = (-1)^{m+1}$. Then we have
\begin{equation*}
\Im(\bD_{(m,\ell)})
=
\begin{cases}
J(\triv, 1-\ell) &\text{if $\ell \in 1+2\Z_{\geq 0}$},\\[3pt]
T_{G'}(n+\ell+1) &\text{if $\ell \in 2(1+\Z_{\geq 0})$}.
\end{cases}
\end{equation*}

\end{enumerate}
\end{prop}

\begin{proof}
Since $\bD_{(m,\ell)} = \D_\ell' \circ \bD_{(m,0)}$,
it follows from Proposition \ref{prop:im1} that 
$\Im(\bD_{(m,\ell)}) = \Im(\D'_\ell)$. Now the proposition follows from 
Theorem \ref{thm:IDO3}.
\end{proof}

\subsection{The image $\Im(\bD\vert_{F_{G}(1-(m+\ell))})$}

We next consider the image $\Im(\bD\vert_{F_{G}(1-(m+\ell))})$
of the restricted operator $\bD\vert_{F_{G}(1-(m+\ell))}$
for $\bD=\bD_{(m,0)}, \bD_{(m,\ell)}$.

\begin{prop}\label{prop:im3}
Suppose that $I(\triv, 1-(m+\ell))^\ga$ is reducible. Then we have 
\begin{equation*}
\Im(\bD_{(m,\ell)}\vert_{F_{G}(1-(m+\ell))}) = \{0\}.
\end{equation*}
\end{prop}

\begin{proof}
It follows from Theorem \ref{thm:IDO3} that 
$F_{G}(1-(m+\ell)) = \Ker(\D_{m+\ell})$. Therefore,
\begin{equation*}
\bD_{(m,\ell)}\vert_{F_{G}(1-(m+\ell))}
=(\wProj_{(m,\ell)} \circ \D_{m+\ell}) \vert_{\Ker(\D_{m+\ell})} \equiv 0.
\end{equation*}
\end{proof}

Suppose that  $I(\triv,1-(m+\ell))^\ga$ is reducible.
Then Fact \ref{thm:vDM} shows that $\ga$ satisfies the condition
$\ga = (-1)^{1-(m+\ell)}$, which is equivalent to 
$\ga\cdot (-1)^m = (-1)^{1-\ell}$. 
By Fact \ref{thm:vDM}, this implies that $J(\triv, 1-\ell)^{\ga+m}$ is also reducible;
in particular, in this case, we have 
$\Ker(\D_\ell') = F_{G'}(1-\ell)$. This observation is used in the proof of the 
following proposition.

\begin{prop}\label{prop:im4}
Suppose that $I(\triv, 1-(m+\ell))^\ga$ is reducible. Then we have 
\begin{equation*}
\Im(\bD_{(m,0)}\vert_{F_{G}(1-(m+\ell))}) = F_{G'}(1-\ell)^{\ga+m}.
\end{equation*}
\end{prop}

\begin{proof}
It follows from Theorem \ref{thm:Proj} that
$\bD_{(m,\ell)} 
= \D'_\ell \circ \bD_{(m,0)}$.
Thus, by Proposition \ref{prop:im3}, we have
\begin{equation*}
(\D'_\ell \circ \bD_{(m,0)}) \vert_{F_{G}(1-(m+\ell))}
=\bD_{(m,\ell)}\vert_{F_{G}(1-(m+\ell))}
\equiv 0,
\end{equation*}
which implies that 
\begin{equation*}
\Im(\bD_{(m,0)}\vert_{F_{G}(1-(m+\ell))}) \subset \Ker(\D_\ell')=F_{G'}(1-\ell).
\end{equation*}
The irreducibility of $F_{G'}(1-\ell)$ forces that 
$\Im(\bD_{(m,0)}\vert_{F_{G}(1-(m+\ell))})
=\{0\}$  or $F_{G'}(1-\ell)$. 

It follows from Proposition \ref{prop:FG} that 
$F_G(1-(m+\ell))$ is given by
\begin{equation*}
F_G(1-(m+\ell)) = 
\bigoplus_{a=0}^{(m+\ell)-1}\C^a[x_1, \ldots, x_n].
\end{equation*}
In particular, we have $x^m_n \in F_G(1-(m+\ell))$. 
As $\bD_{(m,0)} = \Rest_{x_n=0}\circ \tfrac{\partial^m}{\partial x_n^m}$,
the function $\bD_{(m,0)} x_n^m$ is given by
\begin{equation*}
\bD_{(m,0)} x_n^m = m! \neq 0.
\end{equation*}
Thus, $\Im(\bD_{(m,0)}\vert_{F_{G}(1-(m+\ell))}) \neq \{0\}$. 
Consequently, we have 
\begin{equation*}
\Im(\bD_{(m,0)}\vert_{F_{G}(1-(m+\ell))}) = F_{G'}(1-\ell). 
\end{equation*}
\end{proof}

\section{Branching laws of generalized Verma modules}
\label{sec:GVM}

The aim of this section is to discuss the branching law
$\Mp(\triv,s)\vert_{\fg'}$ of a generalized Verma module
$\Mp(\triv,s)$ of scalar-type. In addition, we also consider 
the branching law of the image 
$\Im(\varphi_{p+1})$ of the $\fg$-homomorphism
\begin{equation}\label{eqn:GVMp}
\varphi_{p+1}\colon 
\Mp(\sym_n^{p+1},-(1+\tfrac{p+1}{n})) \to \Mp(\triv, p)
\end{equation}
for $p \in \Z_{\geq 0}$, 
where $\varphi_{p+1}$ is defined as in \eqref{eqn:Hom2}.
The branching laws are achieved in Theorems \ref{thm:GVM31a} 
and \ref{thm:GVM31b}.
In this section we assume $n \geq 2$, unless otherwise specified.

\subsection{Kobayashi's character identity}
\label{sec:character}

To compute the branching law of a generalized Verma module,
the decomposition of the formal character is useful.
 A key tool for it is Kobayashi's character identity \cite{Kobayashi12}. 
We then start this section by recalling from \cite{Kobayashi12} 
the character formula in a general framework.

Let $\fg$ be a complex simple Lie algebra. Choose a Cartan subalgebra $\fh$ and 
write $\gD\equiv \gD(\fg, \fh)$ for the set of roots of $\fg$ with respect to $\fh$.
Fix a positive system $\gD^+$ and denote  by 
$\fb$ the Borel subalgebra of $\fg$ associated with $\gD^+$, namely,
$\fb = \fh \oplus \fu_+$ with $\fu_+=\bigoplus_{\ga \in \gD^+} \fg_\ga$. 
Here $\fg_\ga$ is the root space
for $\ga \in \gD^+$. Let $\Cal{O}$ denote the BGG category of $\fg$-modules
whose objects are finitely generated $\fg$-modules that are $\fh$-semisimple
and locally $\fu_+$-finite.

Let $\fp \supset \fb$ be a standard parabolic subalgebra of $\fg$. Write 
$\fp =\fl \oplus \fn_+$ for the Levi decomposition of $\fp$ with $\fh \subset \fl$. 
We put $\gD^+(\fl):=\{\ga \in \gD^+: \fg_\ga \subset \fl\}$. 
We denote by $\Cal{O}^\fp$ the parabolic BGG category, which is a full
subcategory of $\Cal{O}$ whose objects are $\fl$-semisimple and locally
$\fn_+$-finite.

Let $\IP{\cdot}{\cdot}$ denote the inner product on $\fh^*$ induced from 
a non-degenerate symmetric bilinear form 
of $\fg$. For $\ga \in \gD$, we write $\ga^\vee = 2\ga/\IP{\ga}{\ga}$.
Then we put
\begin{equation*}
\gL^+(\fl):=\{\lambda \in \fh^* : \IP{\lambda}{\ga^\vee} \in \Z_{\geq 0}
\;\;
\text{for all $\ga \in \gD^+(\fl)$\}}.
\end{equation*}

For $\lambda \in \gL^+(\fl)$, we denote by $F_\lambda$ the finite-dimensional
simple $\fg$-module with highest weight $\lambda$. By letting $\fn_+$ act trivially,
we regard $F_\lambda$ as a $\fp$-module. We then define the generalized Verma
module $\Mp(\lambda)$ with highest weight $\lambda$ by 
\begin{equation*}
\Mp(\lambda) =\Cal{U}(\fg)\otimes_{\Cal{U}(\fp)}F_{\lambda}.
\end{equation*}

\vskip 0.1in
Let $\fg'$ be a reductive subalgebra of $\fg$ and
take a hyperbolic element $H\in \fg' \subset \fg$, that is,
the eigenvalues of $\ad(H)$ on $\fg$
are all real-valued. We then define the subalgebras
\begin{equation*}
\fn_-(H),
\quad
 \fl(H),
\quad
\text{and}
\quad
\fn_+(H)
\end{equation*}
as the sum of the eigenspaces of negative, zero, and positive 
eiganvalue of $\ad(H)$, respectively, so that we have 
$\fg=\fn_-(H) \oplus \fl(H) \oplus \fn_+(H)$.

Now suppose that $\fp \subset \fg$ is 
a \emph{$\fg'$-compatible} parabolic subalgebra 
determined by a hyperbolic element $H \in \fg'$,
namely, $\fp = \fp(H):=\fl(H)\oplus \fn_+(H)$.
The $\fg'$-compatibility of $\fp$ implies that
$\fp':=\fp \cap \fg'$ is a parabolic subalgebra of $\fg'$
with Levi decomposition
\begin{equation*}
\fp'=\fl' \oplus \fn_+' :=(\fl(H)\cap \fg') \oplus (\fn_+(H)\cap \fg').
\end{equation*}

We choose a Cartan subalgebra
$\fh' \subset \fg'$ in such a way that  $H \in \fh'$ and that
it extends to a Cartan subalgebra $\fh$ of $\fg$. Then we have 
$\fh \subset \fl(H)$
and $\fh' \subset \fg'$. Hereafter, we simply write 
$\fl =\fl(H)$ and $\fn_{\pm}=\fn_{\pm}(H)$.

Given a finite-dimensional vector space $V$, we write 
$S(V)=\bigoplus_{k=0}^\infty S^k(V)$ for the symmetric tensor algebra 
of $V$. 
For finite-dimensional simple $\fg$- and $\fg'$-modules
$F_\lambda$ and $F'_\delta$ with highest weights $\lambda \in \gL^+(\fl)$
and $\delta \in \gL^+(\fl')$, respectively, we set
\begin{equation*}
m(\delta;\lambda):=
\dim
\Hom_{\fl'}(F'_\delta, F_\lambda\vert_{\fl'} \otimes S(\fn_-/\fn_-\cap \fg')).
\end{equation*}
Let $[\Mp(\lambda)]$ and $[\Mpp(\delta)]$ denote the formal characters 
of $\Mp(\lambda)$ and $\Mpp(\delta)$, respectively.

\begin{thm}[{\cite[Thm.\ 3.10]{Kobayashi12}}]\label{thm:bGVM}
Suppose that $\fp=\fl \oplus \fn_+$ is a $\fg'$-compatible parabolic subalgebra
of $\fg$. Then, for $\lambda \in \gL^+(\fl)$,  the following hold.

\begin{enumerate}
\item[\emph{(1)}]
$m(\delta;\lambda) < \infty$ for all $\delta \in \gL^+(\fl')$.

\item[\emph{(2)}]
In the Grothendieck group of $\Cal{O}^{\fp'}$, we have
\begin{equation}\label{eqn:bGVM}
[\Mp(\lambda)\vert_{\fg'}]
\simeq 
\bigoplus_{\delta \in \mathrm{supp}(\lambda)}
m(\delta; \lambda) [\Mpp(\delta)],
\end{equation}
where $\mathrm{supp}(\lambda) = \{\delta \in \gL^+(\fl'):m(\delta;\lambda)\neq 0\}$.
\end{enumerate}
\end{thm}

\vskip 0.1in

In what follows, we resume the notation and normalizations 
specified in Section \ref{sec:SL}. So, we have 
$\fg= \f{sl}(n+1,\C)$ and $\fg' \simeq \f{sl}(n,\C)$ 
realized as in \eqref{eqn:g'}.
We remark that the maximal parabolic subalgebra
$\fp \subset \fg$ in consideration  is not $\fg'$-compatible, as 
the defining hyperbolic element $H_0\in \fh$ in \eqref{eqn:H0} is not in $\fg'$.
However, as $\fp'=\fp\cap \fg'$ with
$\fl' = \fl \cap \fg'$ and $\fn_+'=\fn_+\cap \fg'$,
a careful observation on the proof of Theorem \ref{thm:bGVM}
shows that the character identity \eqref{eqn:bGVM}  holds also in the present case.
In the next subsection, we apply the character identity \eqref{eqn:bGVM} to
show the branching law $[\Mp(\triv,s)\vert_{\fg'}]$ of  the formal character 
of a generalized Verma module $\Mp(\triv,s)$.

\subsection{Branching law of $[\Mp(s)\vert_{\fg'}]$}
\label{sec:GVMa}

For brevity,  we simply write 
\begin{equation*}
\Mp(s) = \Mp(\triv,s) \quad \text{and} \quad 
\Mpp(r)= \Mpp(\triv,r).
\end{equation*}

\begin{thm}\label{thm:GVM1}
Let $n\geq 2$. For $s \in \C$, 
the following isomorphism holds
in the Grothendieck group of $\Cal{O}^{\fp'}$:
\begin{equation}\label{eqn:GVM1a}
[\Mp(s)\vert_{\fg'}] \simeq 
\bigoplus_{m \in \Z_{\geq 0}}[\Mpp(s-m)].
\end{equation}
\end{thm}

\begin{proof}
Observe that we have $\fn_-/\fn_- \cap \fg' \simeq \C N_n^-$. 
It thus follows from Lemma \ref{lem:MA} that we have 
$S^m(\fn_-/\fn_-\cap \fg') \simeq \C_{-m}$ as $\fa'$-modules.
Then the decomposition 
$(\C_{s} \otimes S(\fn_-/\fn_-\cap \fg'))\vert_{\fl'}$ is given as
\begin{align*}
(\C_{s} \otimes S(\fn_-/\fn_-\cap \fg'))\vert_{\fl'}
&=
\bigoplus_{m\in \Z_{\geq 0}}
(\C_{s} \otimes S^m(\fn_-/\fn_-\cap \fg'))\vert_{\fa'}\\
&\simeq
\bigoplus_{m\in \Z_{\geq 0}}\C_{s-m}.
\end{align*}
Now the character identity \eqref{eqn:bGVM} concludes \eqref{eqn:GVM1a}.
\end{proof}

\begin{cor}\label{cor:GVM}
For $s \in \C\backslash \Z_{\geq 0}$, we have 
\begin{equation}\label{eqn:GVM31a}
\Mp(s)\vert_{\fg'} \simeq 
\bigoplus_{m \in \Z_{\geq 0}}\Mpp(s-m).
\end{equation}
\end{cor}

\begin{proof}
By the classification of the reducibility points
of generalized Verma modules of scalar type (cf.\ \cite{BX21, He15, HKZ19}),
if $s\in \C\backslash \Z_{\geq 0}$, then $\Mpp(s-m)$ is 
a simple $\fg'$-module for all $m\in \Z_{\geq 0}$. 
Now the proposed assertion follows from
Theorem \ref{thm:GVM1}.
\end{proof}

In Section \ref{sec:GVMc} below, 
we shall show that the isomorphism \eqref{eqn:GVM31a} indeed holds for any $s \in \C$
by making use of the classification of 
the $\fn'_+$-invariant subspaces of $\Mp(s)$.

\subsection{Branching law of $[\Im(\varphi_{p+1})\vert_{\fg'}]$}
\label{sec:GVMb}

Now we consider
the branching law of the formal character of 
the image $\Im(\varphi_{p+1})$ of the $\fg$-homomorphism
$\varphi_{p+1}$ in \eqref{eqn:Hom2}.

We first recall from \cite{KuOr24} the classification of 
$\fg$-homomorphisms between generalized Verma modules in consideration.
Define $\Lambda^{n+1}_{\fg} \subset 
\Irr(\f{sl}(n,\C))_\fin \times \C^2$ as
\begin{equation*}
\Lambda^{n+1}_{\fg}:=\{
(\sym_{n}^k; k-1, -(1+\tfrac{k}{n})) :  k\in \Z_{\geq 0}\}.
\end{equation*}

\begin{thm}[{\cite[Thm.\ 5.23]{KuOr24}}]\label{thm:KOHom2b}
We have
\begin{equation*}
\Hom_{\fg}(\Mp(\tau,u), \Mp(\triv, s))
=
\begin{cases}
\C\id & \text{if $(\tau, u) = (\triv, s)$,}\\
\C \varphi_k & \text{if $(\tau; s, u)\in \gL_{\fg}^{n+1}$,}\\
\{0\} & \text{otherwise.}
\end{cases}
\end{equation*}
\end{thm}

Now, let $s = p \in \Z_{\geq 0}$. 
By \eqref{eqn:Hom2}, the image
 $\Im(\varphi_{p+1}) \subset \Mp(p)=\Mp(\triv,p)$ 
is 
\begin{equation*}
\Im(\varphi_{p+1}) = \Cal{U}(\fg)\left(
\C^{p+1}[N_1^-,\ldots, N_n^-] \otimes \C_p
\right).
\end{equation*}
Then, as $\fg$-modules, we have 
\begin{equation*}
\Mp(p)/ \Im(\varphi_{p+1})
\simeq S^p(\C^{n+1}).
\end{equation*}
Since $(\sym^{p}_{n+1}, S^p(\C^{n+1}))$
is a simple $\fg$-module,
this implies that $\Im(\varphi_{p+1})$ 
is a unique maximal submodule of $\Mp(p)$.
The formal character $[\Mp(p)]$ then satisfies
\begin{equation}\label{eqn:GVM3d}
[\Mp(p)] \simeq [\Im(\varphi_{p+1})] + [S^p(\C^{n+1})].
\end{equation}

As in Section \ref{sec:factor1}, we write $\varphi_k'$ for
$\fg'$-homomorphisms between generalized Verma modules of $\fg'$.
Then, for $d \in \Z_{\geq 0}$, we have 
\begin{equation}\label{eqn:GVM3c}
[\Mpp(d)] \simeq [\Im(\varphi'_{d+1})] + [S^{d}(\C^{n})].
\end{equation}

Now we are ready to show the branching law of 
$\Im(\varphi_{p+1})$ in the Grothendiek group of $\Cal{O}^{\fp'}$.

\begin{thm}\label{thm:GVM2}
Let $n\geq 2$. For $p \in \Z_{\geq 0}$, 
the following isomorphism holds
in the Grothendieck group of $\Cal{O}^{\fp'}$:
\begin{equation}\label{eqn:GVM3a}
[\Im(\varphi_{p+1})\vert_{\fg'}] \simeq 
\bigoplus_{d=0}^p [\Im (\varphi_{d+1}')] \oplus \bigoplus_{j \geq 1} [\Mpp(-j)].
\end{equation}
Further, for $n=2$, we have 
\begin{equation}\label{eqn:GVM3b}
[\Im(\varphi_{p+1})\vert_{\fg'}] \simeq 
\bigoplus_{d=0}^p 2 \cdot [\Mpp(-(d+2))] \oplus [\Mpp(-1)] \oplus 
\bigoplus_{j\geq p+3} [\Mpp(-j)].
\end{equation}
\end{thm}

\begin{proof}
We first show the isomorphism \eqref{eqn:GVM3a}.
We consider the branching law $[\Mp(p)\vert_{\fg'}]$ in two ways.
First, by \eqref{eqn:GVM1a}, we have 
\begin{align}
[\Mp(p)\vert_{\fg'}] 
&\simeq \bigoplus_{m \in \Z_{\geq 0}}
[\Mpp(p-m)]\\
&= \bigoplus_{d =0}^{p}[\Mpp(d)]
\oplus
\bigoplus_{j\geq 1} [\Mpp(-j)] \nonumber \\
&\simeq 
\bigoplus_{d =0}^{p}[\Im(\varphi'_{d+1})]
\oplus
[S^p(\C^{n+1})\vert_{\fg'}]
\oplus
\bigoplus_{j\geq 1} [\Mpp(-j)].\label{eqn:GVM3e}
\end{align}
We note that
the character identity
 \eqref{eqn:GVM3c} for $[\Mpp(d)]$ and the classical branching law
$S^p(\C^{n+1})\vert_{\fg'} \simeq \bigoplus_{d=0}^pS^d(\C^{n})$
are applied from line two to line three.

On the other hand, by \eqref{eqn:GVM3d}, the formal character 
$[\Mp(p)\vert_{\fg'}]$ also satisfies
\begin{equation}\label{eqn:GVM3f}
[\Mp(p)\vert_{\fg'}] \simeq [\Im(\varphi_{p+1})\vert_{\fg'}] \oplus 
[S^p(\C^{n+1})\vert_{\fg'}].
\end{equation}
By comparing \eqref{eqn:GVM3e} with \eqref{eqn:GVM3f}, we obtain
\begin{equation*}
[\Im(\varphi_{p+1})\vert_{\fg'}] \simeq
\bigoplus_{d =0}^{p}[\Im(\varphi'_{d+1})]
\oplus
\bigoplus_{j\geq 1} [\Mpp(-j)].
\end{equation*}

Now, to show \eqref{eqn:GVM3b}, let $n=2$. In this case, 
the map $\varphi'_{d+1}$ is a $\fg'$-homomorphism
\begin{equation*}
\varphi'_{d+1}\colon \Mpp(-(d+2)) \To \Mpp(d).
\end{equation*}
Since $\Mpp(-(d+2))$ is of scalar type,
$\varphi'_{d+1}$ is injective for all $d \in \{0,\ldots, p\}$
(cf.\ \cite[Prop.\ 9.11]{Hum08}).
Thus, the image $\Im(\varphi'_{d+1})$ is given by
$\Im(\varphi'_{d+1})\simeq \Mpp(-(d+2))$, which shows that
\begin{equation}\label{eqn:GVM30a}
\bigoplus_{d =0}^{p}[\Im(\varphi'_{d+1})] \simeq 
\bigoplus_{d =0}^{p}[\Mpp(-(d+2))].
\end{equation}
On the other hand, we have 
\begin{align}
\bigoplus_{j\geq 1} [\Mpp(-j)]
&=
\bigoplus_{j =2}^{p+2}[\Mpp(-j)] \oplus [\Mpp(-1)] \oplus 
\bigoplus_{j\geq p+3} [\Mpp(-j)]\nonumber\\
&=\bigoplus_{d =0}^{p}[\Mpp(-(2+d))] \oplus [\Mpp(-1)] \oplus 
\bigoplus_{j\geq p+3} [\Mpp(-j)].\label{eqn:GVM30b}
\end{align}
Now \eqref{eqn:GVM3b} follows from \eqref{eqn:GVM3a}, 
\eqref{eqn:GVM30a}, and \eqref{eqn:GVM30b}.
\end{proof}

In the next section, we show that
the actual branching law $\Im(\varphi_{p+1})\vert_{\fg'}$
is indeed given as in 
\eqref{eqn:GVM3a} and \eqref{eqn:GVM3b}.

\subsection{Branching laws of $\Mp(s)\vert_{\fg'}$ and 
$\Im(\varphi_{p+1})\vert_{\fg'}$}
\label{sec:GVMc}

Now we show the actual branching laws of 
$\Mp(s)\vert_{\fg'}$ and $\Im(\varphi_{p+1})\vert_{\fg'}$.
Our idea is to observe $\fn'_+$-subspaces of 
$\Mp(s)$ and $\Im(\varphi_{p+1})$.
In particular, the rest of the arguments is nothing to do with the character 
identity \eqref{eqn:bGVM}.

It follows from 
Propositions \ref{prop:Sol30a} and \ref{prop:MA3}
and the algebraic Fourier transform $F_c$ in \eqref{eqn:Fc} that
\begin{equation}\label{eqn:M300}
\C^m[N_n^-] \otimes \C_s \subset \Mp(s)^{\fn_+'} 
\quad
\text{for all $m \in \Z_{\geq 0}$}.
\end{equation}
As $\fa'$-modules, we have 
\begin{equation}\label{eqn:M30a}
\C^m[N_n^-] \otimes \C_s\simeq \C_{s-m}
\quad
\text{for all $m \in \Z_{\geq 0}$},
\end{equation}
which yields an isomorphism
\begin{equation}\label{eqn:M30}
\C[N_1^-,\ldots, N_{n-1}^-]\C^m[N_n^-] \otimes \C_s 
=
\Cal{U}(\fg')(\C^m[N_n^-] \otimes \C_s)
\simeq
\Mpp(s-m).
\end{equation}

\begin{thm}\label{thm:GVM31a}
Let $n\geq 2$. For any $s \in \C$,  we have
\begin{equation}\label{eqn:GVM31c}
\Mp(s)\vert_{\fg'} \simeq 
\bigoplus_{m \in \Z_{\geq 0}}\Mpp(s-m).
\end{equation}
\end{thm}

\begin{proof}
By \eqref{eqn:M30}, we have 
\begin{align*}
\Mp(s) 
&=\Cal{U}(\fg)\otimes_{\Cal{U}(\fp)}\C_s\\
&=\C[N_1^-,\ldots, N_{n-1}^-, N_n^-]\otimes \C_s\\
&=\bigoplus_{m\in\Z_{\geq 0}} 
\C[N_1^-,\ldots, N_{n-1}^-]\C^m[N_n^-] \otimes \C_s\\
&\simeq
\bigoplus_{m\in\Z_{\geq 0}} 
\Mpp(s-m).
\end{align*}
\end{proof}

In what follows, we write
\begin{equation*}
\C[(N^-)']=
\C[N_1^-,\ldots, N_{n-1}^-]
\quad 
\text{and}
\quad
\C[(N^-)', N_n^-]=\C[N_1^-,\ldots, N_{n-1}^-, N_n^-].
\end{equation*}
Then, by \eqref{eqn:Hom2},
the images
 $\Im(\varphi_{p+1}) \subset \Mp(p)$ for $p \in \Z_{\geq 0}$
and
$\Im(\varphi'_{d+1}) \subset \Mpp(d)$ for $d \in \Z_{\geq 0}$
are given by
\begin{align}
\Im(\varphi_{p+1}) &= 
\Cal{U}(\fg)\left(
\C^{p+1}[(N^-)', N_n^-] \otimes \C_p
\right), \label{eqn:M30A}\\
\Im(\varphi'_{d+1}) &= \Cal{U}(\fg')\left(
\C^{d+1}[(N^-)'] \otimes \C_d
\right),\label{eqn:M30B}
\end{align}
where we have 
\begin{equation}\label{eqn:M30D}
\C^{p+1}[(N^-)', N_n^-] \otimes \C_p
\subset \Mp(p)^{\fn_+}
\quad
\text{and}
\quad
\C^{d+1}[(N^-)'] \otimes \C_d
\subset \Mpp(d)^{\fn_+'}.
\end{equation}

\begin{thm}\label{thm:GVM31b}
Let $n\geq 2$. For $p \in \Z_{\geq 0}$, we have 
\begin{equation}\label{eqn:GVM31A}
\Im(\varphi_{p+1})\vert_{\fg'} \simeq 
\bigoplus_{d=0}^p \Im (\varphi_{d+1}') \oplus \bigoplus_{j \geq 1} \Mpp(-j).
\end{equation}
In partucular, for $n=2$, we have 
\begin{equation}\label{eqn:GVM31B}
\Im(\varphi_{p+1})\vert_{\fg'} \simeq 
\bigoplus_{d=0}^p 2 \cdot \Mpp(-(d+2)) \oplus \Mpp(-1) \oplus 
\bigoplus_{j\geq p+3} \Mpp(-j).
\end{equation}
\end{thm}

\begin{proof}
Since \eqref{eqn:GVM31B} follows from \eqref{eqn:GVM31A}
as in Theorem \ref{thm:GVM2}, it suffices to show \eqref{eqn:GVM31A}.

First observe that
\begin{align}
\C^{p+1}[(N^-)', N_n^-]\otimes \C_p
&=\sum_{b=0}^{p+1}\C^b[\Np]\C^{p+1-b}[\Nn]\otimes \C_p \nonumber \\
&\simeq \sum_{b=0}^{p+1}\C^b[\Np]\otimes \C_{b-1} \nonumber \\
&=\sum_{d=0}^p \C^{d+1}[\Np] \otimes \C_d + (1\otimes \C_{-1}).
\label{eqn:M30C}
\end{align}
We note that \eqref{eqn:M30a} is applied from line one to line two.
Then, by \eqref{eqn:M30A} and \eqref{eqn:M30C}, we have
\begin{align}
\Im(\varphi_{p+1})
&=\Cal{U}(\fg)\left(\C^{p+1}[(N^-)', N_n^-] \otimes \C_p\right) \nonumber \\
&=\C[\Np,\Nn]\C^{p+1}[(N^-)', N_n^-] \otimes \C_p \nonumber \\
&=\sum_{d=0}^p \C[\Np,\Nn]\C^{d+1}[\Np] \otimes \C_d 
+ \C[\Np,\Nn]\otimes \C_{-1} \label{eqn:M30E}\\
&=: \text{(B)} + \text{(A)}.\label{eqn:M31A}
\end{align}

The second term \text{(A)} is given as
\begin{align}
\text{(A)}
= \C[\Np,\Nn]\otimes \C_{-1}
&=\bigoplus_{c=0}^\infty \C[\Np]\C^c[\Nn]\otimes \C_{-1} \nonumber\\
&\simeq \bigoplus_{c=0}^\infty\C[\Np]\otimes \C_{-1-c} \nonumber\\
&=\bigoplus_{j=1}^\infty\C[\Np]\otimes \C_{-j} \nonumber\\
&=\bigoplus_{j=1}^\infty\Mpp(-j).\label{eqn:M30F}
\end{align}
By \eqref{eqn:M300}, we have
$\C_{-j}\simeq \C^c[\Nn]\otimes \C_{-1} \subset \Mp(-1)^{\fn_+'}$,
which verifies the identity from line three to line four.

Next, for \text{(B)}, we have 
\begin{align}
\text{(B)}
&=\sum_{d=0}^p \C[\Np,\Nn]\C^{d+1}[\Np] \otimes \C_d \nonumber \\
&=\sum_{d=0}^p\sum_{c=0}^\infty \C[\Np]\C^c[\Nn]\C^{d+1}[\Np]\otimes \C_d
\nonumber\\
&=\sum_{d=0}^p\C[\Np]\C^{d+1}[\Np]\otimes \C_d
+
\sum_{d=0}^p\sum_{c=1}^\infty \C[\Np]\C^c[\Nn]\C^{d+1}[\Np]\otimes \C_d
\nonumber \\
&=:\text{(B1)} + \text{(B2)}.\label{eqn:M30G}
\end{align}

\noindent
By \eqref{eqn:M30B} and \eqref{eqn:M30D},  \text{(B1)} is given by
\begin{align}
\text{(B1)}
=\sum_{d=0}^p\C[\Np]\C^{d+1}[\Np]\otimes \C_d
&=\bigoplus_{d=0}^p\Cal{U}(\fg')\left(\C^{d+1}[\Np]\otimes \C_d\right)\nonumber \\
&=\bigoplus_{d=0}^p \Im(\varphi'_{d+1}).\label{eqn:M30H}
\end{align}

The actual realizations of $\Mpp(-j)$ in \text{(A)} and $\Im(\varphi'_{d+1})$ in \text{(B1)}
imply that 
the sum $\text{(B1)} + \text{(A)}$ is indeed a direct sum
$\text{(B1)} \oplus \text{(A)}$.
We then claim the following.

\begin{claim}\label{claim:M30}
We have $\text{(B2)} \subset \text{(A)} \oplus \text{(B1)}$.
\end{claim}

Suppose that Claim \ref{claim:M30} holds. Then, 
by \eqref{eqn:M31A}, \eqref{eqn:M30F}, \eqref{eqn:M30G}, and \eqref{eqn:M30H},
we have 
\begin{align*}
\Im(\varphi_{p+1}) 
= \text{(B)}+\text{(A)}
&=\text{(B1)} + \text{(B2)} + \text{(A)}\\
&= \text{(B1)}\oplus \text{(A)}\\
&\simeq 
\bigoplus_{d=0}^p \Im(\varphi'_{d+1}) \oplus \bigoplus_{j=1}^\infty\Mpp(-j),
\end{align*}
which is what we wish to show. Thus, in the rest of the proof, we aim to show 
Claim \ref{claim:M30}.

Write
\begin{equation*}
\text{(C)}=\C[\Np]\C^c[\Nn]\C^{d+1}[\Np]\otimes \C_d,
\end{equation*}
so that 
$\text{(B2)} = \sum_{d=0}^p\sum_{c=1}^\infty \text{(C)}$.

First observe that \text{(A)}, \text{(B)}, and \text{(C)} are given by
\begin{align*}
\text{(A)}
&= \C[\Np,\Nn]\otimes \C_{-1}\\
&=\C[\Np,\Nn]\C^{p+1}[\Nn]\otimes \C_p,\\[7pt]
\text{(B1)}
&=\sum_{d=0}^p\C[\Np]\C^{d+1}[\Np]\otimes \C_d\\
&=\sum_{b=1}^{p+1}\C[\Np]\C^{b}[\Np]\C^{p+1-b}[\Nn]\otimes \C_p, \\[7pt]
\text{(C)}
&= \C[\Np]\C^c[\Nn]\C^{d+1}[\Np]\otimes \C_d\\
&= \C[\Np]\C^c[\Nn]\C^{b}[\Np]\C^{p+1-b}[\Nn]\otimes \C_p\\
&= \C[\Np]\C^{b}[\Np]\C^{p+1+c-b}[\Nn]\otimes \C_p.
\end{align*}

\noindent
Thus, if $c-b\geq 0$, then
\begin{align*}
\text{(C)}
&=\C[\Np]\C^{b}[\Np]\C^{p+1+c-b}[\Nn]\otimes \C_p\\
&\subset\C[\Np,\Nn]\C^{p+1}[\Nn]\otimes \C_p\\
&=\text{(A)}.
\end{align*}

\noindent
If $c-b < 0$, then $-(p+1) \leq c-b \leq -1$. Thus,
the number $p+1+c-b$ is of the form
\begin{equation*}
p+1+c-b = p+1-b_0
\end{equation*}
for some $b_0 \in \{1,\ldots, p+1\}$. Therefore,
\begin{align*}
\text{(C)}
&=\C[\Np]\C^{b}[\Np]\C^{p+1-b_0}[\Nn]\otimes \C_p\\
&\subset 
\sum_{b=1}^{p+1}\C[\Np]\C^{b}[\Np]\C^{p+1-b}[\Nn]\otimes \C_p\\
&=\text{(B1)}.
\end{align*}

\noindent
As $\text{(B2)} = \sum_{d=0}^p\sum_{c=1}^\infty \text{(C)}$, this shows the claim.
\end{proof}

\begin{rem}\label{rem:branching916}
Here are some remarks on Theorems \ref{thm:GVM31a} and \ref{thm:GVM31b}. 
\begin{enumerate}
\item
The inclusion $\Mpp(s-m) \hookrightarrow \Mp(s)$ is given by 
the normal derivative $\Phi_{(m,0)}$ as in Theorems \ref{thm:Hom3a}
and  \ref{thm:Hom3b}. 

\item By \eqref{eqn:GVM31c} and \eqref{eqn:GVM31A}, 
for $p \in \Z_{\geq 0}$ and $d \in [0,p]\cap \Z_{\geq 0}$,
we have 
\begin{equation}\label{eqn:GVMsquare}
\begin{aligned}
\xymatrix@C=1pc @R=1pc{
\Mpp(d) 
\ar@{}[r]|*{\subset}
\ar@{}[d]|{\bigcup}& \Mp(p)\ar@{}[d]|{\bigcup}\\
\Im(\varphi'_{d+1})\ar@{}[r]|*{\subset} & \Im(\varphi_{p+1})
}
\end{aligned}
\end{equation}
The square \eqref{eqn:GVMsquare} corresponds to 
the factorization identity $\Phi_{(m,\ell)}=\Phi_{(m,0)} \circ \varphi_\ell'$
in  \eqref{eqn:comm1}.

\item 
Let $M_j$ for $j=1,2$ denote the two copies of $\Mpp(-(2+d))$
in \eqref{eqn:GVM31B} such that $M_1=\Im(\varphi'_{d+1})\subset \Mpp(d)$. 
Then \eqref{eqn:GVM31B} shows that, for $n=2$,  we have 
\begin{equation}\label{eqn:GVMsquare2}
\begin{aligned}
\xymatrix@C=1pc @R=1pc{
\Mpp(d)\oplus M_2\hspace{10pt}
\ar@{}[r]|*{\subset}
\ar@{}[d]|{\bigcup}& 
\;\;
\Mp(p) 
\ar@{}[d]|{\bigcup}\\
M_1\oplus M_2
\ar@{}[r]|*{\subset} &
\Im(\varphi_{p+1})}
\end{aligned}
\end{equation}
This corresponds to the multiplicity-two phenomenon in
Theorem \ref{thm:Hom3b}. Namely,
the inclusions 
$M_1 \hookrightarrow \Mp(p)$
and $M_2 \hookrightarrow \Mp(p)$ 
are related to $\fg'$-homomorphisms $\Phi_{(m,\ell)}=\Phi_{(p-d,d+1)}$ and 
$\Phi_{(m+2\ell,0)}=\Phi_{(p+d+2,0)}$, respectively.
\end{enumerate}
\end{rem}

\section{Differential symmetry breaking operators $\bD$ for $(GL(n+1,\R), GL(n,\R))$}
\label{sec:GL}

The aim of this section is to classify and construct differential symmetry breaking operators $\bD$ for the pair $(G, G') = (GL(n+1,\R), GL(n,\R))$ with
maximal parabolic subgroups $(P, P')$ such that $G/P \simeq \RP^n$ and $G'/P' \simeq \RP^{n-1}$. We also discuss $G$-intertwining differential operators $\D$ 
and the factorization identities of $\bD$. Those results are achieved in Theorems
\ref{thm:GL-DSBO}, \ref{thm:GL-IDO}, and \ref{thm:GL-factor}. 
In this section we assume $n\geq 2$, unless otherwise specified.

\subsection{
Notation
}
\label{sec:notation2}
We start by introducing some notation. 
Let $G=GL(n+1,\R)$ with Lie algebra $\fg(\R)=\f{gl}(n+1,\R)$
for $n \geq 2$. Let $G'$ denote the closed subgroup of $G$ defined by
\begin{equation*}
G'=
\left\{
\begin{pmatrix}
g'&\\
& 1\\
\end{pmatrix}
:
g' \in GL(n,\R)
\right\}\simeq GL(n,\R)
\end{equation*}
with Lie algebra
\begin{equation*}
\fg'(\R)= \left\{
\begin{pmatrix}
X' & \\
 & 0
\end{pmatrix}
:
X' \in \f{gl}(n,\R)
\right\}\simeq \f{gl}(n,\R).
\end{equation*}

Let $P = \text{Stab}_{G}(\R(1,0,\ldots,0)^t)$ and $P'=G' \cap P$. Then $P$ and $P'$ are 
parabolic subgroups of $G$ and $G'$, respectively, such that 
$G/P \simeq \RP^n$ and $G'/P' \simeq \RP^{n-1}$.
Let  $M$ and $M'$ denote the subgroups of $P$ and $P'$, respectively, 
defined by
\begin{alignat*}{2}
M&:= 
\left\{
\begin{pmatrix}
\eps &\\
& g\\
\end{pmatrix}
:
\eps \in \{\pm 1\}
\;
\text{and}
\;
g \in SL^{\pm}(n,\R)
\right\}
&&\simeq \Z/2\Z \times SL^{\pm}(n,\R),\\
M'&:= 
\left\{
\begin{pmatrix}
\eps & & \\
& g' &\\
&&1
\end{pmatrix}
:
\eps \in \{\pm 1\}
\;
\text{and}
\;
g' \in SL^{\pm}(n-1,\R)
\right\}
&&\simeq \Z/2\Z \times SL^{\pm}(n-1,\R).
\end{alignat*}

We write
\begin{alignat*}{2}
J_0&=\frac{1}{n}(\sum^{n+1}_{r=2}E_{r,r})
&&=\frac{1}{n}\diag(0, 1, 1, \ldots, 1,1),\\
J_0'&=\frac{1}{n-1}(\sum^n_{r=2}E_{r,r})
&&=\frac{1}{n-1}\diag(0, 1, 1, \ldots, 1,0)
\end{alignat*}
and put 
\begin{alignat*}{2}
A_1&:=\exp(\R \wH_0), \quad A_2&&:=\exp(\R J_0),\\
A_1'&:=\exp(\R \wH_0'), \quad A_2'&&:=\exp(\R J_0'),
\end{alignat*}
where $\wH_0$ and $\wH_0'$ are the diagonal matrices defined in
\eqref{eqn:0801a} and \eqref{eqn:0801b}, respectively. 
We then define $A$ and $A'$ by
\begin{equation*}
A=A_1A_2
\quad
\text{and}
\quad
A'=A_1'A_2'.
\end{equation*}

Let $N_+$ and $N_+'$ be the unipotent subgroups defined in Section \ref{sec:notation}.
Then $P=MAN_+$ and $P'=M'A'N_+'$ are Langlands decompositions of $P$ 
and $P'$, respectively.

For $(\lambda_1,\lambda_2),(\nu_1,\nu_2) \in \C^2$, 
we define one-dimensional representations 
$\C_{(\lambda_1,\lambda_2)}=(\chi^{(\lambda_1,\lambda_2)}, \C)$ 
of $A=\exp(\R\wH_0)\exp(\R J_0)$ 
and
$\C_{(\nu_1,\nu_2)} = ((\chi')^{(\nu_1,\nu_2)},\C)$ of 
$A' = \exp(\R \wH_0')\exp(\R J_0')$ by
\begin{alignat*}{1}
\chi^{(\lambda_1,\lambda_2)} 
&\colon \exp(t_1 \wH_0)\exp(t_2 J_0)
 \longmapsto \exp(\lambda_1 t_1)\exp(\lambda_2 t_2),\\[3pt]
(\chi')^{(\nu_1,\nu_2)}&\colon \exp(t_1 \wH_0')\exp(t_2 J_0') \longmapsto 
\exp(\nu_1 t_1)\exp(\nu_2 t_2).
\end{alignat*}
Then $\Irr(A)$ and $\Irr(A')$ are given by
\begin{equation*}
\Irr(A)=\{\C_{(\lambda_1,\lambda_2)} : \lambda_j \in \C\} \simeq \C^2
\quad
\text{and}
\quad
\Irr(A')=\{\C_{(\nu_1,\nu_2)} : \nu_j \in \C\} \simeq \C^2.
\end{equation*}

For $(\alpha_1, \alpha_2) \in \{\pm\}^2$,
a one-dimensional representation $\C_{(\alpha_1,\alpha_2)}$ of $M$ 
is defined  by
\begin{equation*}
\begin{pmatrix}
\eps &\\
& g\\
\end{pmatrix}
\longmapsto
\sgn^{\alpha_1}(\eps) \cdot
\sgn^{\alpha_2}(\det(g)),
\end{equation*}
where, for $b\in \R^\times$, $\sgn^\alpha(b)$ is defined as in
\eqref{eqn:20241108}.
Then 
$\Irr(M)_\fin \simeq \{\pm\}^2\times \Irr(SL^{\pm}(n,\R))_\fin$
and
$\Irr(M')_\fin \simeq \{\pm\}^2\times  \Irr(SL^{\pm}(n-1,\R))_\fin$
are given as
\begin{align*}
\Irr(M)_{\fin}&\simeq
\{\C_{\ga_1} \boxtimes (\C_{\alpha_2} \otimes \xi): 
(\alpha_1, \alpha_2, \xi) \in \{\pm\}^2 \times \Irr(SL(n,\R))_{\fin}\},\\
\Irr(M')_{\fin}&\simeq
\{\C_{\beta_1} \boxtimes (\C_{\beta_2} \otimes \varpi): 
(\beta_1, \beta_2, \varpi) \in \{\pm\}^2 \times \Irr(SL(n-1,\R))_{\fin}\}.
\end{align*}
Since 
$\Irr(P)_{\fin}\simeq \Irr(M)_{\fin} \times \Irr(A)$, the set $\Irr(P)_\fin$ can be 
parametrized by
\begin{equation*}
\Irr(P)_{\fin}\simeq 
\{\pm\}^2 \times \Irr(SL(n,\R))_{\fin} \times \C^2.
\end{equation*}
Similarly, we have
\begin{equation*}
\Irr(P')_{\fin}\simeq 
\{\pm\}^2 \times \Irr(SL(n-1,\R))_{\fin} \times \C^2.
\end{equation*}

For 
$(\bm{\ga}, \xi, \bm{\lambda}) \in \{\pm\}^2 \times \Irr(SL(n,\R))_{\fin} \times \C^2$
with $\bm{\ga} = (\ga_1,\ga_2)$ and $\bm{\lambda}=(\lambda_1, \lambda_2)$,
we write
\begin{equation*}
I(\xi;\bm{\lambda})^{\bm{\ga}} 
= 
\Ind_{P}^G\left(\C_{\ga_1}\boxtimes (\C_{\ga_2} \otimes \xi)\boxtimes \C_{(\lambda_1,\lambda_2)}\right)
\end{equation*}
for (unnormalized) parabolically induced representations of $G$.
Likewise,  for
$(\bm{\gb}, \varpi, \bm{\nu}) \in \{\pm\}^2 \times \Irr(SL(n-1,\R))_{\fin} \times \C^2$
with $\bm{\gb} = (\gb_1,\gb_2)$ and $\bm{\nu}=(\nu_1, \nu_2)$, we write
\begin{equation*}
J(\varpi;\bm{\nu})^{\bm{\gb}} 
= 
\Ind_{P'}^{G'}\left(\C_{\gb_1}\boxtimes(\C_{\gb_2} \otimes \varpi)\boxtimes 
\C_{(\nu_1,\nu_2)}\right).
\end{equation*}

In the next three subsections we shall state the main results
of differential symmetry breaking operators
$\bD$, $G$-intertwining differential operators $\D$, and factorization identities.
In Section \ref{sec:GL-proof}, we shall discuss the proofs of the statements.

\subsection{
Differential symmetry breaking operators $\bD$ for $(GL(n+1,\R), GL(n,\R))$
}
\label{sec:GL-DSBO}

For $n\geq 2$, we define 
\begin{equation*}
\gL^{(n+1,n)}_{GL,j} \subset \{\pm\}^4 \times \Irr(SL(n-1,\R))_{\fin} \times \C^4 
\end{equation*}
for $j=1,2$ as follows.
\begin{equation*}
\gL^{(n+1,n)}_{GL,1}:=\{(\balpha;\bbeta; \varpi; \blambda; \bnu):
\text{$(\balpha,\bbeta; \varpi; \blambda, \bnu)$ satisfies \eqref{eqn:GL2}
below for $\alpha_j \in \{\pm\}$, $\lambda_j \in \C$, and $m \in \Z_{\geq 0}$.}\}
\end{equation*}
\begin{align}\label{eqn:GL2}
(\balpha;\bbeta) &= (\alpha_1, \alpha_2; \alpha_1+m; \alpha_2) \nonumber\\
\varpi &=\triv\\
(\blambda;\bnu)&= (\lambda_1, \lambda_2; \lambda_1+m,\lambda_2)\nonumber
\end{align}

\begin{equation*}
\gL^{(n+1,n)}_{GL,2}:=\{(\balpha;\bbeta; \varpi; \blambda; \bnu):
\text{$(\balpha,\bbeta; \varpi; \blambda, \bnu)$ satisfies \eqref{eqn:GL3} below
for $\alpha_j \in \{\pm\}$ and $\ell, m \in \Z_{\geq 0}$.}\}
\end{equation*}
\begin{align}\label{eqn:GL3}
(\balpha;\bbeta) &= (\alpha_1, \alpha_2;\alpha_1+(m+\ell); \alpha_2) \nonumber\\
\varpi &=\poly_{n-1}^\ell\\
(\blambda;\bnu)&= (1-(m+\ell), \lambda_2; 1+\tfrac{\ell}{n-1},\lambda_2-\tfrac{\ell}{n-1})\nonumber
\end{align}

Further, we put
\begin{equation*}
\gL^{(n+1,n)}_{GL}:=\gL^{(n+1,n)}_{GL,1} \cup \gL^{(n+1,n)}_{GL,2}.
\end{equation*}

Let $\bD_{(m,\ell)} \in \Diff_\C(C^\infty(\R^{n}), C^\infty(\R^{n-1})\otimes 
\C^{m+\ell}[y_1, \ldots, y_{n-1}])$ be the differential operator defined in \eqref{eqn:DSBO}.

\begin{thm}\label{thm:GL-DSBO}
Let $n\geq 2$. Then we have
\begin{equation*}
\Diff_{G'}\big(I(\triv; \bm{\lambda})^{\bm{\alpha}}, J(\varpi; \bm{\nu})^{\bm{\beta}}\big)
=
\begin{cases}
\C\bD_{(m,0)} 
& \text{if $(\bm{\alpha}, \bm{\beta};\varpi; \bm{\lambda}, \bm{\nu})\in \Lambda^{(n+1,n)}_{GL,1}$,}\\[3pt]
\C \bD_{(m,\ell)} 
& \text{if $(\bm{\alpha}, \bm{\beta};\varpi; \bm{\lambda}, \bm{\nu})\in \Lambda^{(n+1,n)}_{GL,2}$,}\\
\{0\} & \text{otherwise.}
\end{cases}
\end{equation*}
\end{thm}

\begin{rem}
As opposed to the $SL$ case, the multiplicity-one property holds even for $n=2$.
\end{rem}

\subsection{
Intertwining differential operators $\D$ for $GL(n,\R)$
}
\label{sec:GL-IDO}

For $n\geq 2$,
define 
\begin{equation*}
\Lambda^{n+1}_{GL} \subset 
 \{\pm\}^4 \times \Irr(SL(n,\R))_{\fin} \times \C^4
\end{equation*}
as follows.
\begin{equation*}
\Lambda^{n+1}_{GL}:=\{(\balpha;\bgamma; \xi; \blambda; \btau):
\text{$(\balpha,\bgamma; \xi; \blambda, \btau)$ satisfies \eqref{eqn:GL1} below
for $\alpha_j \in \{\pm\}$, $\lambda_2 \in \C$, and $k\in \Z_{\geq 0}$.} \}
\end{equation*}
\begin{equation}\label{eqn:GL1}
\begin{aligned}
(\balpha;\bgamma) &= (\alpha_1, \alpha_2; \alpha_1+k; \alpha_2)\\
\xi &=\poly_{n}^k\\
(\blambda;\btau)&= (1-k, \lambda_2; 1+\tfrac{k}{n},\lambda_2-\tfrac{k}{n})
\end{aligned}
\end{equation}
\noindent
Let $\D_k \in \Diff_\C(C^\infty(\R^{n}), C^\infty(\R^{n})\otimes 
\C^k[y_1, \ldots, y_{n-1}, y_{n}])$ be the differential operator defined in 
\eqref{eqn:IDO}.

\begin{thm}
\label{thm:GL-IDO}
Let $n\geq 2$.
Then we have
\begin{equation*}
\Diff_{G}(I(\triv; \blambda)^{\balpha}, I(\xi; \btau)^{\bgamma})
=
\begin{cases}
\C\id & \text{if $(\bm{\delta}, \xi, \bm{\tau}) = (\balpha, \triv, \blambda)$,}\\
\C \D_k & \text{if $(\bm{\alpha}, \bm{\delta};\xi; \bm{\lambda}, \bm{\tau})\in \Lambda^{n+1}_{GL}$,}\\
\{0\} & \text{otherwise.}
\end{cases}
\end{equation*}

\end{thm}

\subsection{
Factorization identities of $\bD_{(m,\ell)}$
}
\label{sec:GL-factor}

We next state the factorization identities of $\bD_{(m,\ell)}$.

\begin{thm}\label{thm:GL-factor}
Let $n\geq 2$.
Also,
let $\D_\ell'$, $\D_k$, and $\wProj_{(m,\ell)}$ be the same operators
considered in Theorem \ref{thm:Proj}.
Then, for
$(\balpha,\bbeta;\varpi;\blambda ,\bnu) \in 
\gL^{(n+1,n)}_{GL,2}$,
the differential symmetry breaking operator
$\bD_{(m,\ell)}$ can be factored as in Theorem \ref{thm:Proj}, namely,
\begin{equation*}
\bD_{(m,\ell)} 
= \D'_\ell \circ \bD_{(m,0)}
=  \wProj_{(m,\ell)}\circ \D_{m+\ell}.
\end{equation*}
Equivalently, the following diagram commutes.
\begin{equation}\label{eqn:factor0826a}
\xymatrix@=13pt{
I(\triv; 1-(m+\ell),\lambda_2)^{(\ga_1,\ga_2)}
 \ar[dd]_{\D_{m+\ell}}
  \ar[rrdd]^{\bD_{(m,\ell)}}
  \ar[rr]^{\bD_{(m,0)}}
 && J(\triv; 1-\ell,\lambda_2)^{(\ga_1+m,\ga_2)}
  \ar[dd]^{\D'_\ell}
  \ar@{}[ldd]|{\circlearrowleft}
 \\
&&\\
I(\poly^{m+\ell}_n; 1+\tfrac{m+\ell}{n},\lambda_2-\frac{m+\ell}{n})^{(\ga_1+m+\ell,\ga_2)}
\ar[rr]_{\wProj_{(m,\ell)}}   \ar@{}[ruu]|{\circlearrowleft}
&&
J(\poly^\ell_{n-1}; 1+\tfrac{\ell}{n-1},\lambda_2-\tfrac{\ell}{n-1})^{(\ga_1+m+\ell,\ga_2)}
}
\end{equation}
\end{thm}

\begin{rem}
If $n=2$, then the factorization identity \eqref{eqn:factor0826a} becomes
\begin{equation*}
\xymatrix@=13pt{
I(\triv; 1-(m+\ell),\lambda_2)^{(\ga_1,\ga_2)}
 \ar[dd]_{\D_{m+\ell}}
  \ar[rrdd]^{\bD_{(m,\ell)}}
  \ar[rr]^{\bD_{(m,0)}}
 && J(\triv; 1-\ell,\lambda_2)^{(\ga_1+m,\ga_2)}
  \ar[dd]^{\D'_\ell}
  \ar@{}[ldd]|{\circlearrowleft}
 \\
&&\\
I(\poly^{m+\ell}_2; 1+\tfrac{m+\ell}{2},\lambda_2-\frac{m+\ell}{2})^{(\ga_1+m+\ell,\ga_2)}
\ar[rr]_{\wProj_{(m,\ell)}}   \ar@{}[ruu]|{\circlearrowleft}
&&
J(\poly^\ell_{1}; 1+\ell,\lambda_2-\ell)^{(\ga_1+m+\ell,\ga_2)}
\ar@{}[d]|{\rotatebox{90}{$=$}}\\
&&
J(\triv; 1+\ell,\lambda_2-\ell)^{(\ga_1+m,\ga_2)}
}
\end{equation*}
\end{rem}

\subsection{
Proofs of Theorems \ref{thm:GL-DSBO}, \ref{thm:GL-IDO}, and \ref{thm:GL-factor}
}
\label{sec:GL-proof}

We only briefly discuss the proofs of Theorems 
\ref{thm:GL-DSBO} and \ref{thm:GL-IDO} and 
Theorem \ref{thm:GL-factor} as these are  similar 
to the ones of Theorems \ref{thm:DSBO2a} and \ref{thm:Proj}.

\subsubsection{
Proof of Theorem \ref{thm:GL-DSBO}
}
\label{sec:GL-proof1}

The proofs of Theorems \ref{thm:GL-DSBO} and \ref{thm:GL-IDO} utilize
the F-method. As the unipotent radicals $N_{\pm}$ for the $GL$-case are the same as
those for the $SL$-case, the system of PDEs to solve is also essentially the same.
The only difference is the $M'A'$-decomposition 
\begin{equation*}
\Pol(\fn_+)\vert_{M'A'} = \C[\zeta_1,\ldots, \zeta_{n-1},\zeta_n]\vert_{M'A'}.
\end{equation*}
Thus, in the present subsection and the next, we only focus on Step 2 of the
recipe of the F-method.

As for the $SL$-case,
the following observation would play a role.

\begin{lem}\label{lem:GL-MA}
The following hold.

\begin{enumerate}

\item[\emph{(1)}]
$(M', \Ad_\#, \C^m[\zeta_n]) 
\hspace{42pt}
\simeq (\Z/2\Z\times SL^\pm(n-1,\R), \sgn^m \boxtimes (\triv \otimes \triv), \C)$.

\item[\emph{(2)}]
$(M', \Ad_\#, \C^\ell[\zeta_1,\ldots, \zeta_{n-1}]) 
\hspace{2pt}
\simeq (\Z/2\Z\times SL^\pm(n-1,\R), \sgn^\ell \boxtimes 
(\triv \otimes \sym^\ell_{n-1}), S^\ell(\C^{n-1}))$.

\item[\emph{(3)}]
$A_1'$ acts on $\C^m[\zeta_n]$ by a character with weight $-m$.

\item[\emph{(4)}]
$A_1'$ acts on $\C^\ell[\zeta_1,\ldots, \zeta_{n-1}]$ by a character 
with weight $-\frac{n}{n-1}\ell$.

\item[\emph{(4)}]
$A_2'$ acts on $\C^m[\zeta_n]$ trivially.

\item[\emph{(5)}]
$A_2'$ acts on $\C^\ell[\zeta_1,\ldots, \zeta_{n-1}]$ by a character with weight
$\tfrac{\ell}{n-1}$.

\end{enumerate}

\end{lem}

\begin{proof}
A direct computation.
\end{proof}

\begin{rem}
As in Remark \ref{rem:sym0826a}, if $n=2$, then 
\begin{align*}
(M', \Ad_\#, \C^\ell[\zeta_1]) 
&\simeq (\Z/2\Z\times \Z/2\Z, \sgn^\ell \boxtimes 
(\triv \otimes \sym^\ell_1), S^\ell(\C))\\
&= (\Z/2\Z\times \Z/2\Z, \triv \boxtimes 
(\triv \otimes \triv), \C).
\end{align*}
\end{rem}

It follows from Lemma \ref{lem:GL-MA} that the decomposition
$\C[\zeta_1, \ldots, \zeta_{n-1},\zeta_n]\vert_{M'A'}$ is given  as
\begin{align}\label{eqn:MA2}
&\C[\zeta_1, \ldots, \zeta_{n-1},\zeta_n]\vert_{M'A'}\\
&=
\bigoplus_{m, \ell \in \Z_{\geq 0}} 
\C^m[\zeta_n]  
\C^\ell[\zeta_1,\ldots,\zeta_{n-1}]\nonumber\\
&\simeq
\bigoplus_{m, \ell \in \Z_{\geq 0}} 
\sgn^{m+\ell}\boxtimes(\triv\otimes\sym^\ell_{n-1}) \boxtimes 
(-(m+\tfrac{n}{n-1}\ell), \tfrac{\ell}{n-1}),
\end{align}
where  $(-(m+\tfrac{n}{n-1}\ell),\tfrac{\ell}{n-1})$ 
indicates the weight of the character of $A'=A_1'A_2'$.

As in \eqref{eqn:Pol}, for $\balpha=(\alpha_1,\alpha_2) \in \{\pm\}^2$, we write
\begin{equation*}
\Pol(\fn_+)_{\balpha} =\C_{(\ga_1,\ga_2)} \otimes \Pol(\fn_+).
\end{equation*}
Then, for $\blambda=(\lambda_1,\lambda_2) \in \C^2$,  we have
\begin{align}\label{eqn:GL-MA}
&\big(\Pol(\fn_+)_{\balpha}
\otimes \C_{-\blambda}\big)\vert_{M'A'} \nonumber\\
&\simeq 
\bigoplus_{m, \ell \in \Z_{\geq 0}} 
\sgn^{\ga_1+(m+\ell)}\boxtimes
(\sgn^{\ga_2} \otimes\sym^\ell_{n-1}) \boxtimes 
(-(\lambda_1+m+\tfrac{n}{n-1}\ell), -(\lambda_2-\tfrac{\ell}{n-1})).
\end{align}
We remark that the representations appeared in 
\eqref{eqn:GL-MA} are all inequivalent even for $n=2$.

The rest of the arguments are proceeded exactly as in the proof of 
Theorem \ref{thm:Sol1} and Section \ref{sec:Step4}. Hence, we omit the details.

\subsubsection{
Proof of Theorem \ref{thm:GL-IDO}
}
\label{sec:GL-proof2}

As in the previous section, we only focus on Step 2 of the recipe of the F-method.

\begin{lem}\label{lem:GL-MA2}
The following hold.

\begin{enumerate}

\item[\emph{(1)}]
$(M, \Ad_\#, \C^k[\zeta_1,\ldots, \zeta_{n-1},\zeta_n]) 
\hspace{2pt}
\simeq (\Z/2\Z\times SL^\pm(n,\R), \sgn^k \boxtimes 
(\triv \otimes \sym^k_{n}), S^k(\C^{n}))$.

\item[\emph{(2)}]
$A_1$ acts on $\C^k[\zeta_1,\ldots, \zeta_{n-1},\zeta_n]$ by a character 
with weight $-\frac{n+1}{n}k$.

\item[\emph{(3)}]
$A_2$ acts on $\C^k[\zeta_1,\ldots, \zeta_{n-1},\zeta_n]$ 
by a character with weight
$\tfrac{k}{n}$.

\end{enumerate}

\end{lem}

\begin{proof}
A direct computation.
\end{proof}

By  Lemma \ref{lem:GL-MA2}, we have
\begin{align}\label{eqn:MA3}
\C[\zeta_1, \ldots, \zeta_{n-1},\zeta_n]\vert_{MA}
&=
\bigoplus_{k \in \Z_{\geq 0}} 
\C^k[\zeta_1,\ldots,\zeta_{n-1}, \zeta_n]\nonumber\\
&\simeq
\bigoplus_{k \in \Z_{\geq 0}} 
\sgn^{k}\boxtimes(\triv\otimes\sym^k_{n}) \boxtimes 
(-\tfrac{n+1}{n}k, \tfrac{k}{n}),
\end{align}
which shows that
\begin{align}\label{eqn:GL-MA4}
\big(\Pol(\fn_+)_{\balpha}
\otimes \C_{-\blambda}\big)\vert_{MA} 
\simeq 
\bigoplus_{k \in \Z_{\geq 0}} 
\sgn^{\ga_1+k}\boxtimes(\sgn^{\ga_2}\otimes\sym^k_{n}) \boxtimes 
(-(\lambda_1+\tfrac{n+1}{n}k), -(\lambda_2-\tfrac{k}{n})).
\end{align}

Clearly, each irreducible constituent appeared in \eqref{eqn:GL-MA4} is 
inequivalent. The rest of the proof is essentially the same as 
\cite[Sect.\ 5]{KuOr24} or that of Theorem \ref{thm:Sol1}. 
Hence, we omit the details.

\subsubsection{
Proof of Theorem \ref{thm:GL-factor}
}
\label{sec:GL-proof3}

The proof of the factorization identities in Theorem \ref{thm:GL-factor}
is also the same as that of Theorem \ref{thm:Proj} in principle.
The only thing that one needs to check is the 
the linear map $\Emb_{(m,\ell)}$ satisfies $M'A'$-equivariance 
also for the $GL$-case. Thus, in this section, we only show
the $M'A'$-equivariance of $\Emb_{(m,\ell)}$.

Recall from \eqref{eqn:Emb} that, for $m,\ell \in \Z_{\geq 0}$, we define
\begin{equation*}
\Emb_{(m,\ell)} \in \Hom_{\C}(\C^\ell[e'], \C^{m+\ell}[e',e_n])
\end{equation*}
as
\begin{equation*}
\Emb_{(m,\ell)} 
:= \sum_{\mathbf{l}\in\Xi_{\ell}'}
(e_{\mathbf{l}})^\vee \otimes e_{(m,\mathbf{l})}
= \sum_{\mathbf{l}\in\Xi_{\ell}'}
\wy_{\mathbf{l}} \otimes e_{(m,\mathbf{l})}.
\end{equation*}

For $\lambda_2 \in \C$, we put
\begin{equation*}
\bmu':=(1+\tfrac{\ell}{n-1}, \lambda_2-\tfrac{\ell}{n-1})
\quad 
\text{and}
\quad
\bmu:=(1+\tfrac{m+\ell}{n},\lambda_2-\tfrac{m+\ell}{n}).
\end{equation*}

\begin{prop}\label{prop:Emb2}
We have
\begin{equation*}
\Hom_{M'A'}(\C^\ell[e']_{\balpha} \boxtimes \C_{-\bmu'},
\C^{m+\ell}[e',e_n]_{\balpha}\boxtimes \C_{-\bmu})
=\C \Emb_{(m,\ell)}.
\end{equation*}
\end{prop}

\begin{proof}
As in the proof Proposition \ref{prop:Emb2a}, it suffices to show
that $\Emb_{(m,\ell)}$ has the desired $A'$-equivariance property.
For $A' = A_1'A_2'$, the $A_1'$-equivariance is already checked in 
the proof of Proposition \ref{prop:Emb2a}. Thus, we only consider 
$A_2'$-equivariance.

Take
\begin{equation*}
b' = \diag(1, t^{\frac{1}{n-1}}, \ldots, t^{\frac{1}{n-1}}, 1) \in A_2'.
\end{equation*}
The element $b'$ can be decomposed as
$b' = b'_M b'_A$ with $b'_M \in M$ and $b'_A \in A_2$, where
\begin{align*}
b'_M 
&= \diag(1,t^{\frac{1}{n(n-1)}},\ldots, t^{\frac{1}{n(n-1)}}, t^{\frac{-1}{n}})
\in M,\\
b'_A
&= \diag(1,t^{\frac{1}{n}},\ldots, t^{\frac{1}{n}}, t^{\frac{1}{n}}) 
\in A_2.
\end{align*}

As $\C^\ell[e']\boxtimes \C_{-\bmu'}$ is an $M'A'$-module by definition,
$A'_2$ acts on $\C^\ell[e']\boxtimes \C_{-\bmu'}$ by a character 
with weight $\lambda_2-\tfrac{\ell}{n-1}$.
Next, observe that the action of $M'A'$ 
on $\C^m[e_n]\C^\ell[e']\boxtimes \C_{-\bmu}$
is the restriction of that of $MA$ on $\C^{m+\ell}[e',e_n]\boxtimes \C_{-\bmu}$.
Thus, for $e_n^mp(e')\otimes \mathbb{1}_{-\bmu'}$, we have 
\begin{align*}
b'\cdot (e_n^mp(e')\otimes \mathbb{1}_{-\bmu})
&=(b'_M \cdot e_n^m)(b'_M \cdot p(e'))\otimes (b'_A \cdot \mathbb{1}_{-\bmu})\\
&=t^{-\tfrac{m}{n}}\cdot 
t^{\tfrac{\ell}{n(n-1)}} \cdot t^{\lambda_2-\tfrac{m+\ell}{n}}(e_n^mp(e')\otimes \mathbb{1}_{-\bmu})\\
&=t^{\lambda_2-\tfrac{\ell}{n-1}}(e_n^mp(e')\otimes \mathbb{1}_{-\bmu}).
\end{align*}
Therefore, $A'_2$ acts on $\C^m[e_n]\C^\ell[e']\boxtimes \C_{-\bmu}$ 
also by $\lambda_2-\tfrac{\ell}{n-1}$. 
Now the proposition follows.
\end{proof}

Since the rest of the arguments is the same as those in Sections \ref{sec:factorGVM} and 
\ref{sec:factorDSBO}, we omit the details.

\begin{rem}\label{rem:FW2}
As commented in Section \ref{sec:FW}, 
it follows from Theorem \ref{thm:GL-DSBO} that,
 for $(\ga_1, \ga_2)=(+, +)$, 
$m \in 2\Z_{\geq 0}$, and $\ell \in 1+\Z_{\geq 0}$, 
we have
\begin{equation*}
\Diff_{G'}(I(\triv;1-(m+\ell),\lambda_2)^{(+,+)}, J(\triv; 1+\ell, \lambda_2-\ell)^{(+,+)}
=\C\bD_{(m,\ell)},
\end{equation*}
where
\begin{equation}\label{eqn:remarkD}
\bD_{(m,\ell)} = \Rest_{x_n=0} \circ
\frac{\partial^m}{\partial x_2^m}\frac{\partial^\ell}{\partial x_1^\ell}.
\end{equation}
Since
$\frac{\partial^\ell}{\partial x_1^\ell}$ is the residue operators 
of a Knapp--Stein operator, 
the formula \eqref{eqn:remarkD} shows that 
$\bD_{(m,\ell)}$ is a composition of
a normal derivative $\frac{\partial^m}{\partial x_2^m}$
and the residue operator $\frac{\partial^\ell}{\partial x_1^\ell}$
of a Knapp--Stein operator.
\end{rem}


\vspace{5pt}

\noindent
\textbf{Acknowledgements.}
The author thanks to Professor Ali Baklouti  and Professor Hideyuki Ishi for their warm hospitality during the 7th Tunisian-Japanese Conference,
Geometric and Harmonic Analysis on Homogeneous Spaces  and 
Applications in Honor of Professor Toshiyuki Kobayashi
in Monastir, Tunisia, November 1--4, 2023.
He is also grateful to Professor Jan Frahm, Professor Bent {\O}rsted, and Professor Toshiyuki Kobayashi for stimulating discussions on this paper.

The author was partially supported by JSPS
Grant-in-Aid for Scientific Research(C) (JP22K03362).




\end{document}